\documentclass[11pt]{amsart}
\usepackage{amsmath,amssymb,amsthm,amsxtra} 
\usepackage{stmaryrd}
\usepackage{times}
\usepackage[margin=1.3in]{geometry}
\usepackage[linktocpage=true, hidelinks, colorlinks=true, allcolors=blue]{hyperref}

\makeatletter
\def\l@subsection{\@tocline{2}{0pt}{2.5pc}{5pc}{}}
\renewcommand\tocchapter[3]{%
  \indentlabel{\@ifnotempty{#2}{\ignorespaces#2.\quad}}#3%
}
\newcommand\@dotsep{4.5}
\def\@tocline#1#2#3#4#5#6#7{\relax
  \ifnum #1>\c@tocdepth 
  \else
    \par \addpenalty\@secpenalty\addvspace{#2}%
    \begingroup \hyphenpenalty\@M
    \@ifempty{#4}{%
      \@tempdima\csname r@tocindent\number#1\endcsname\relax
    }{%
      \@tempdima#4\relax
    }%
    \parindent\z@ \leftskip#3\relax \advance\leftskip\@tempdima\relax
    \rightskip\@pnumwidth plus1em \parfillskip-\@pnumwidth
    #5\leavevmode\hskip-\@tempdima{#6}\nobreak
    \leaders\hbox{$\m@th\mkern \@dotsep mu\hbox{.}\mkern \@dotsep mu$}\hfill
    \nobreak
    \hbox to\@pnumwidth{\@tocpagenum{#7}}\par
    \nobreak
    \endgroup
  \fi}
\makeatother
\AtBeginDocument{%
\makeatletter
\expandafter\renewcommand\csname r@tocindent0\endcsname{0pt}
\makeatother
}
\def\l@subsection{\@tocline{2}{0pt}{2.5pc}{5pc}{}}

\newtheorem{thm}{Theorem}[section]
\newtheorem{lemma}[thm]{Lemma}
\newtheorem{proposition}[thm]{Proposition}

\newtheorem{definition}[thm]{Definition}
\newtheorem{corollary}[thm]{Corollary}

\newtheorem*{problem1}{Problem 1}
\newtheorem*{problem2}{Problem 2}
\newtheorem*{problem3}{Problem 3}
\newtheorem*{problem4}{Problem 4}
\newtheorem*{maintheorem*}{Main Theorem}
\newtheorem*{theorem*}{Theorem}
\newtheorem*{corollary*}{Corollary}

\newcommand{\p}{\mathbb{P}}
\newcommand{\q}{\mathbb{Q}}

\newcommand{\dom}{\mathrm{dom}}
\newcommand{\ran}{\mathrm{ran}}

\newcommand{\add}{\textrm{Add}}

\newcommand{\h}{\mathrm{ht}}
\newcommand{\col}{\mathrm{Col}}

\newcommand{\res}{\upharpoonright}

\newcommand{\ka}{\kappa}

\title{Forcing Over a Free Suslin Tree}

\author{John Krueger}
\address{Department of Mathematics, University of North Texas, 1155 Union Circle \#311430, Denton, TX 76203, USA}
\email{john.krueger@unt.edu}

\author{\v{S}\'{a}rka Stejskalov\'{a}}
\address{Department of Logic, Faculty of Arts, Charles University, Celetná 20, 116 42 Prague 1, Czech Republic}
\email{sarka.stejskalova@ff.cuni.cz}

\date{January 31, 2024; revised July 20, 2026}



\begin{document}

\begin{abstract}
We introduce a forcing for adding almost disjoint 
automorphisms of a normal infinitely splitting 
$\omega_1$-tree $T$ with countable approximations. 
Assuming that $T$ is a free Suslin tree, this forcing is totally proper, 
preserves the Suslinness of $T$, and does not add new cofinal branches 
of $\omega_1$-trees existing in intermediate extensions. 
If $\ka$ is an inaccessible cardinal, then the product of the automorphism forcing of length $\ka$ with the 
L\'{e}vy collapse of $\ka$ to become $\omega_2$ forces that there exists an 
almost Kurepa Suslin tree and there does not exist a Kurepa tree. 
This model solves open problems due to Bilaniuk, Jin, Shelah, and Moore.
\end{abstract}

\maketitle

\tableofcontents

\section{Introduction} \label{Introduction}

A classical and fundamental result in mathematics due to Cantor states that any non-empty 
linearly ordered set without endpoints which is dense, complete, and separable 
is isomorphic to the real number line. 
Suslin asked whether the same conclusion follows if the assumption of 
separability is replaced by the countable chain condition, which means that every 
pairwise disjoint family of open intervals is countable (\cite{suslin}). 
An equivalent question is whether every linearly ordered set 
with the countable chain condition is separable. 
A counter-example to this statement is called a \emph{Suslin line}. 
\emph{Suslin's hypothesis} (\textsf{SH}) 
is the statement that there does not exist a Suslin line. 
A number of authors independently discovered that \textsf{SH} can be 
characterized in terms of trees (\cite{kurepa1936, miller, sierpinski}). 
\textsf{SH} is equivalent to the non-existence of a \emph{Suslin tree}, which is 
an uncountable tree which has no uncountable chain or uncountable antichain.

The first systematic study of trees appeared in the 
dissertation of Kurepa (\cite{kurepa1935}). 
It includes a construction of an \emph{Aronszajn tree}, 
which is an $\omega_1$-tree with no uncountable chains, 
whose existence had been proved by Aronszajn in 1934 
(by an $\omega_1$-tree, we mean 
a tree of height $\omega_1$ with countable levels). 
Shortly after, Kurepa proved the existence of a 
\emph{special} Aronszajn tree 
(a tree is special if it is a union of countably many antichains) (\cite{kurepa1937}). 
Later, Kurepa investigated the question of how many cofinal branches 
exist in $\omega_1$-trees (\cite{kurepa1942}). 
The statement that there exists an $\omega_1$-tree with more than $\omega_1$-many cofinal 
branches became known as \emph{Kurepa's Hypothesis} (\textsf{KH}), and such a tree 
is called a \emph{Kurepa tree}.

The resolution of \textsf{SH} and \textsf{KH} came with the advent of modern 
methods of set theory, namely, constructibility, forcing, and large cardinals 
(\cite{godel, cohen}). 
The consistency of $\neg \textsf{SH}$ 
was established independently by 
Jech and Tennenbaum, who 
defined forcings which add a Suslin tree (\cite{jech67, tennenbaum}). 
The consistency of \textsf{KH} 
was observed to follow from an inaccessible cardinal 
by Bukovsk\'{y} and Rowbottom (\cite{bukovsky, rowbottom}). 
Namely, the L\'{e}vy collapse of an inaccessible cardinal to become $\omega_1$ forces 
the existence of a Kurepa tree.   
Later, \textsf{KH} was shown to be consistent without large cardinals 
by Stewart using a direct forcing construction (\cite{stewart}). 
It was also shown that in the constructible universe $L$, there exists a Suslin tree 
and there exists a Kurepa tree. 
Namely, Jensen proved that 
$\Diamond$ implies the existence of a Suslin tree and 
$\Diamond^+$ implies 
the existence of a Kurepa tree, and these diamond principles hold in $L$.

For the other direction, Silver proved the consistency of $\neg \textsf{KH}$ 
by showing that after forcing with the L\'{e}vy collapse to turn 
an inaccessible cardinal into $\omega_2$, 
there does not exist a Kurepa tree (\cite{silver}). 
Solovay proved that Silver's use of an inaccessible cardinal is necessary, because 
if $\omega_2$ is not an inaccessible cardinal in $L$ then there exists a Kurepa tree. 
Solovay and Tennenbaum proved the consistency of \textsf{SH} 
using their newly developed technique of 
finite support iterations of c.c.c.\ forcings to construct a model of 
Martin's Axiom (\textsf{MA}) together with the negation of the Continuum Hypothesis (\textsf{CH}) 
(\cite{solovaytennenbaum}; also see \cite{martin}). 
Baumgartner isolated a statement about trees which implies \textsf{SH} and follows 
from $\textsf{MA} + \neg \textsf{CH}$, 
namely, that all Aronszajn trees are special 
(\cite{baumgartnerdiss, baumgartnermalitz}). 
The consistency of \textsf{SH} together with \textsf{CH} 
was proved by Jensen (\cite{devlinj}). 
Jensen's proof motivated Shelah's invention of proper forcing, and Shelah 
gave an alternative proof of the consistency of 
$\textsf{SH} + \textsf{CH}$ as an application of his general technique for 
iterating proper forcing while not adding reals (\cite{shelahbook}). 
Both Jensen's and Shelah's models satisfy the stronger statement that all 
Aronszajn trees are special.

Among the earliest topics studied about Suslin trees after their existence 
was shown to be consistent are rigidity and homogeneity (\cite{devlinj, jech72}). 
Jensen proved that $\Diamond$ implies the existence of both a rigid Suslin tree and 
a homogeneous Suslin tree with exactly $\omega_1$-many automorphisms. 
And Jensen proved that $\Diamond^+$ implies the existence of a homogeneous Suslin tree 
with at least $\omega_2$-many automorphisms. 
Reviewing the construction of this last tree from $\Diamond^+$, 
it is easy to verify that it is an example of an 
\emph{almost Kurepa Suslin tree}, which is a Suslin tree which becomes a Kurepa tree 
after forcing with it. 
Motivated by Jensen's results, 
Jech proved that if \textsf{CH} holds and $\ka$ is a cardinal such that 
$2^\omega \le \ka \le 2^{\omega_1}$ and $\ka^\omega = \ka$, 
then there exists a forcing which adds a Suslin tree with 
exactly $\ka$-many automorphisms (\cite{jech72}). 
In the case that $\ka \ge \omega_2$, Jech's forcing gives another example of an 
almost Kurepa Suslin tree.

Jensen's constructions of a rigid Suslin tree and a homogeneous Suslin tree from $\Diamond$ 
identified two important types of trees: 
free Suslin trees and uniformly coherent Suslin trees. 
For any positive $n < \omega$, a Suslin tree $T$ is \emph{$n$-free} 
if for any distinct elements $x_0,\ldots,x_{n-1}$ of the same level of $T$, the product 
tree $T_{x_0} \otimes \cdots \otimes T_{x_{n-1}}$ 
(called a \emph{derived tree with dimension $n$}) is Suslin.\footnote{In contrast, 
Kurepa \cite{kurepa1935} showed that for any Suslin tree $T$, $T \otimes T$ 
is not Suslin as witnessed by an antichain $\{ (x_{\alpha,0},x_{\alpha,1}) : \alpha < \omega_1 \}$ 
such that the meet of $x_{\alpha,0}$ and $x_{\alpha,1}$ is strictly increasing with $\alpha$.} 
And $T$ is \emph{free} if it is $n$-free for all positive $n < \omega$.\footnote{The concept 
of a free Suslin tree goes by different names in the literature. 
Free Suslin trees were originally introduced by Jensen 
as \emph{full Suslin trees} (\cite{jensennotes}). 
Abraham and Shelah refer to free Suslin trees as 
\emph{Suslin trees all of whose derived trees are Suslin} (\cite{AS, AS2}). 
The phrase \emph{free Suslin tree} was used by 
Larson and Shelah-Zapletal (\cite{larson1, shelahzap}).}  
The idea of a free Suslin tree is due to Jensen, and the 
rigid Suslin tree he had constructed earlier from $\Diamond$ is free. 
The homogeneous Suslin tree constructed by Jensen from $\Diamond$ is an example of a 
\emph{uniformly coherent Suslin tree}, 
which means a Suslin tree consisting of countable sequences of 
natural numbers, downwards closed and closed under finite modifications, 
such that any two elements of the tree disagree on at most finitely many 
elements of their domain. 
Any uniformly coherent Suslin tree is homogeneous. 
Uniformly coherent Suslin trees have proven useful in a variety of contexts, including 
in $\p_{\text{max}}$-style constructions involving a Suslin tree, consistency results, 
and in forcing axioms (\cite{larson1, shelahzap, woodin, larsontodor, PFAS}).

Free Suslin trees satisfy some remarkable properties. 
Freeness is the strongest known form of rigidity for Suslin trees. 
Free Suslin trees have the \emph{unique branch property}, which 
means that forcing with a free Suslin tree 
introduces exactly one cofinal branch to it (\cite{fuchshamkins}). 
A free Suslin tree is \emph{forcing minimal} in the sense that after forcing with it, 
there are no intermediate models strictly between the ground model 
and the generic extension (\cite[Lemma 1.5.7]{gido-diss}). 
Any strictly increasing and level preserving map from a free Suslin tree into 
any Aronszajn tree is injective on a club of levels. 
A Suslin tree is free if and only if it satisfies a property which 
is essentially a translation of the definition of an entangled set of reals 
into the context of trees (\cite{jk37, ASentangled}). 
Another noteworthy fact about free Suslin trees is their ubiquitousness. 
The generic Suslin trees of both Jech and Tennenbaum are free. 
Larson proved that if there exists a uniformly coherent Suslin tree, 
then there exists a free Suslin tree (\cite{larson1}). 
Since forcing a Cohen real adds a uniformly coherent Suslin tree, 
it also adds a free Suslin tree (\cite{todorpartition}). 
In fact, it is an open problem due to Shelah and Zapletal whether the existence 
of a Suslin tree implies the existence of a free Suslin tree (\cite{shelahzap}). 
In other words, it could be the case that 
\textsf{SH} is actually equivalent to the non-existence of a free Suslin tree.

In order to motivate the problems which this article addresses, 
we review some of the early forcings for adding $\omega_1$-trees with different properties. 
Jech's forcing for adding a Suslin tree consists of conditions which 
are countable infinitely splitting downwards closed 
normal subtrees of the tree $({}^{<\omega_1} \omega,\subset)$, 
ordered by end-extension (\cite{jech67}). 
Jech's forcing is countably closed, $\omega_2$-c.c.\ assuming \textsf{CH}, 
and adds a free Suslin tree. 
Tennenbaum's forcing for adding a Suslin tree consists of 
conditions which are finite trees whose elements are in $\omega_1$ and whose tree ordering  
is consistent with the ordinal ordering, ordered by end-extension (\cite{tennenbaum}). 
Tennenbaum's forcing is c.c.c.\ and adds a free Suslin tree.

Both Jech's and Tennenbaum's forcings serve as a foundation on which other forcing posets 
for adding $\omega_1$-trees are based. 
A variation of Jech's poset, in which any two 
elements of a condition differ on a finite set and conditions are closed under 
finite modifications, adds a uniformly coherent Suslin tree. 
Another variation adds a Suslin tree together with 
any number of automorphisms of it (\cite{jech72}). 
Stewart's forcing for adding a Kurepa tree consists of conditions of 
the form $(T,f)$, where $T$ is a condition in Jech's forcing 
with successor height and $f$ is an injective function from a countable subset of $\omega_2$ 
into the top level of $T$. 
Conditions are ordered by letting 
$(U,g) \le (T,f)$ if $U$ end-extends $T$, $\dom(f) \subseteq \dom(g)$, 
and for all $\alpha \in \dom(f)$, $f(\alpha) \le_U g(\alpha)$. 
Stewart's forcing is countably closed, and assuming \textsf{CH}, is $\omega_2$-c.c.\ 
and adds an $\omega_1$-tree together with $\omega_2$-many cofinal branches of it.

In light of Stewart's forcing for adding a Kurepa tree based on Jech's forcing, 
a natural question is whether 
there exists a c.c.c.\ forcing for adding a Kurepa tree based on Tennenbaum's forcing. 
Jensen introduced the \emph{generic Kurepa hypothesis} 
(\textsf{GKH}), which states that there exists a Kurepa tree in some c.c.c.\ forcing 
extension (\cite{jensenccc}). 
Jensen and Schlechta proved that \textsf{GKH} is not a theorem 
of \textsf{ZFC}: if $\ka$ is a Mahlo cardinal, then after forcing with the 
L\'{e}vy collapse to turn $\ka$ into $\omega_2$, 
any c.c.c.\ forcing fails to add a Kurepa tree (\cite{jensenschlechta}). 
On the other hand, Jensen proved that $\Box_{\omega_1}$ implies 
the existence of a c.c.c.\ forcing for adding a Kurepa tree (\cite{jensenccc}). 
Since the failure of $\Box_{\omega_1}$ is equiconsistent with a Mahlo cardinal, 
so is the statement $\neg \textsf{GKH}$. 
Later, Veli\v{c}kovi\'{c} defined a c.c.c.\ forcing for adding a 
Kurepa tree which is simpler than Jensen's forcing; 
it is based on Tennenbaum's forcing and uses the function $\rho$ 
of Todor\v{c}evi\'{c} derived from a $\Box_{\omega_1}$-sequence (\cite{boban}).

The forcings of Stewart and Veli\v{c}kovi\'{c} for adding a Kurepa tree 
have size at least $\omega_2$, 
due to the fact that the conditions in these forcing posets  
approximate both an $\omega_1$-tree and 
a sequence of $\omega_2$-many cofinal branches of the tree. 
Jin and Shelah asked whether it is possible to force 
the existence of a Kurepa tree using a forcing of size at most $\omega_1$, especially 
in the context of \textsf{CH} (\cite{jinshelah}). 
This question was motivated in part by the fact that there exists a forcing of size $\omega$ 
which adds a Suslin tree, namely, the forcing for adding one Cohen real (\cite{shelahcohen}). 
The main result of Jin and Shelah \cite{jinshelah} is that 
assuming the existence of an inaccessible cardinal $\ka$, 
there exists a forcing which preserves $\omega_1$, collapses $\ka$ to become $\omega_2$, 
forces that there does not exist a Kurepa tree, and introduces a countably distributive 
Aronszajn tree which when you force with it produces a Kurepa tree. 
Jin and Shelah asked whether it is possible to obtain such a model where the Aronszajn 
tree is replaced by some c.c.c.\ forcing of size at most $\omega_1$.

\begin{problem1}[Jin and Shelah \cite{jinshelah}]
	Is it consistent that \textsf{CH} holds, there does not exist a Kurepa tree, 
	and there exists a c.c.c.\ forcing of size at most $\omega_1$ 
	which forces the existence of a Kurepa tree?
\end{problem1}

As previously mentioned, Jensen proved that $\Diamond^+$ implies the existence of 
a Kurepa tree and a Suslin tree with at least $\omega_2$-many automorphisms. 
In his dissertation written under the supervision of Baumgartner, 
Bilaniuk proved that if $\Diamond$ holds and there exists a Kurepa tree, 
then there exists a Suslin tree with at least $\omega_2$-many automorphisms \cite{bilaniuk}.

\begin{problem2}[Bilaniuk {\cite{bilaniuk}}]
	Is it consistent that $\Diamond$ holds, there does not exist a Kurepa tree, 
	and there exists a Suslin tree with at least $\omega_2$-many 
	automorphisms?
\end{problem2}

In the Jin-Shelah model, forcing with the Aronszajn tree introduces \emph{another} 
tree which is a Kurepa tree. 
On the other hand, an 
\emph{almost Kurepa Suslin tree} is a Suslin tree which \emph{itself} 
becomes a Kurepa tree after forcing with it. 
The following problem is closely related to both Problems 1 and 2 and has been 
worked on by a number of set theorists since Bilaniuk's dissertation.

\begin{problem3}[Folklore]
	Is it consistent that there exists an almost Kurepa Suslin tree and 
	there does not exist a Kurepa tree?\footnote{Concerning the relationship 
	between Problems 2 and 3, 
	start with a model with a Mahlo cardinal $\ka$ and a uniformly coherent Suslin tree $T$. 
	After forcing with $\col(\omega_1, < \! \ka) * \add(\omega,\omega_2)$, 
	$T$ is a Suslin tree with $\omega_2$-many automorphisms. 
	But $T$ is not an almost Kurepa Suslin tree, 
	since by the result of Jensen and Schlechta, no c.c.c.\ forcing can introduce 
	a Kurepa tree in this model.} 
\end{problem3}

Baumgartner introduced the idea of a \emph{subtree base} for an 
$\omega_1$-tree $T$, which is a collection $\mathcal B$ of uncountable downwards closed 
subtrees of $T$ such that every uncountable downwards closed subtree of $T$ contains 
some member of $\mathcal B$ (\cite{baumgartnerbase}). 
He proved that after forcing with 
the L\'{e}vy collapse $\col(\omega_1,< \! \ka)$, where $\ka$ is an inaccessible cardinal, 
every Aronszajn tree has a base of cardinality $\omega_1$. 
A related idea called \emph{Aronszajn tree saturation} 
was introduced by K\"onig, Larson, Moore, 
and Veli\v{c}kovi\'{c} in the context of attempting to 
reduce the large cardinal assumption used to produce a model with a five element 
basis for the class of uncountable linear orders (\cite{moorebounding, moorefive}). 
An Aronszajn tree $T$ is \emph{saturated} if 
every almost disjoint family of uncountable downwards closed subtrees of $T$ 
has cardinality at most $\omega_1$.

Note that if $T$ has a subtree base of size $\omega_1$, then $T$ is saturated. 
So after forcing with the L\'{e}vy collapse $\col(\omega_1,< \! \ka)$, where 
$\ka$ is an inaccessible cardinal, 
every Aronszajn tree is saturated, and by Silver's result, 
there does not exist a Kurepa tree. 
On the other hand, Baumgartner and 
Todor\v{c}evi\'{c} proved that if there exists 
a Kurepa tree, then there exists a special Aronszajn tree 
which is not saturated (\cite{baumgartnerbase}). 
These facts lead to the following natural question of Moore.
\begin{problem4}[Moore \cite{moorestructural}]
	Is it consistent that there exists a non-saturated Aronszajn tree and 
	there does not exist a Kurepa tree?\footnote{According to Moore, this 
question is implicit in \cite{moorestructural} (see the comments after Question 9.2). 
It also appears as problem \#28 in \cite{problembook}.}
\end{problem4}

In this article, we provide solutions to Problems 1, 2, 3, and 4. 
Our main result is as follows:

\begin{maintheorem*}
	Suppose that there exists an inaccessible cardinal $\ka$ 
	and there exists 
	an infinitely splitting normal free Suslin tree $T$. 
	Then there exists a forcing poset $\p$ satisfying that the product forcing 
	$\col(\omega_1,< \! \kappa) \times \p$ forces:
	\begin{enumerate}
	\item $\ka = \omega_2$;
	\item \textsf{GCH} holds;
	\item $T$ is a Suslin tree;
	\item there exists an almost disjoint 
	family $\{ f_\tau : \tau < \omega_2 \}$ of automorphisms of $T$;
	\item there does not exist a Kurepa tree.
	\end{enumerate}
\end{maintheorem*}

If $b$ is a generic branch obtained by forcing with 
the Suslin tree $T$ over a generic extension by 
$\col(\omega_1,< \! \kappa) \times \p$, then 
$\{ f_\tau[b] : \tau < \omega_2 \}$ is a family of $\omega_2$-many cofinal branches of $T$. 
Thus, in this generic extension $T$ is a c.c.c.\ forcing of size $\omega_1$ 
which forces the existence of a Kurepa tree. 
Starting with a model with an inaccessible cardinal and forcing the 
existence of an infinitely splitting normal 
free Suslin tree (for example, by Jech's forcing), 
we get the following corollary which solves   
Problems 1 and 3.

\begin{corollary*}
	Assume that there exists an inaccessible cardinal $\ka$. 
	Then there exists a generic extension in which $\ka$ equals $\omega_2$, \textsf{CH} holds, 
	there exists an almost Kurepa Suslin tree, and there does not exist a Kurepa tree.
\end{corollary*}

Concerning Problem 2, it suffices to find a generic extension as described in the 
Main Theorem which satisfies $\Diamond$. 
Start with a model $V$ in which there exists an inaccessible cardinal $\ka$ 
and $\Diamond$ holds. 
Let $\q$ be Jech's forcing in $V$ for adding a Suslin tree. 
Let $\dot \p$ be a $\q$-name for the forcing described in the Main Theorem 
using the generic Suslin tree. 
Since $\q$ is $\omega_1$-closed, the forcings 
$\q * (\col(\omega_1, < \! \ka)^{V^\q} \times \dot \p)$ and 
$(\q * \dot \p) \times \col(\omega_1,< \! \ka)$ are forcing equivalent. 
We show in Section \ref{The Main Result} that the two-step iteration $\q * \dot \p$ is forcing 
equivalent to some $\omega_1$-closed forcing, and consequently so is 
$(\q * \dot \p) \times \col(\omega_1, < \! \ka)$. 
But $\omega_1$-closed forcings preserve $\Diamond$, so $\Diamond$ holds in the generic 
extension of $V^\q$ described in the Main Theorem. 
Since we can force $\Diamond$, we have the following corollary.

\begin{corollary*}
	Assume that there exists an inaccessible cardinal $\ka$. 
	Then there exists a generic extension in which $\ka$ equals $\omega_2$, $\Diamond$ holds, 
	there exists a normal Suslin tree with $\omega_2$-many automorphisms, 
	and there does not exist a Kurepa tree.
\end{corollary*}

Concerning Problem 4, working in the generic extension by 
$\col(\omega_1, < \! \ka) \times \p$, for all $\tau < \ka$ let 
$U_\tau = \{ (x,f_\tau(x)) : x \in T \}$. 
Then each $U_\tau$ is an uncountable downwards closed subtree of the 
Aronszajn tree $T \otimes T$, and any two such subtrees have countable 
intersection.\footnote{More generally, 
Moore has pointed out to the authors that if $T$ is a normal almost Kurepa Suslin tree, 
then the Aronszajn tree $T \otimes T$ is non-saturated. 
For suppose that 
$\langle \dot b_\tau : \tau < \omega_2 \rangle$ is a sequence of $T$-names 
for distinct cofinal branches of $T$. 
For each $\tau < \omega_2$, let $U_\tau$ be the downward closure 
of the set of $(x,y) \in T \otimes T$ such that $x \Vdash_T y \in \dot b_\tau$. 
Using the Suslinness of $T$, one can show that each $U_\tau$ is uncountable and 
any two such subtrees have countable intersection.  
So the family $\{ U_\tau : \tau < \omega_2 \}$ witnesses that $T \otimes T$ 
is not saturated.} 
We thus get the following corollary which answers Problem 4.

\begin{corollary*}
	Assume that there exists an inaccessible cardinal $\ka$. 
	Then there exists a generic extension in which $\ka$ equals $\omega_2$, 
	there exists a non-saturated Aronszajn tree, and there does not exist a Kurepa tree.
\end{corollary*}

For the entirety of the article, we fix a normal infinitely splitting $\omega_1$-tree $T$. 
Our main goal is to introduce and develop a forcing which adds almost disjoint 
automorphisms of $T$ using countable approximations. 
While many of the ideas we develop apply to any tree $T$ as above, 
in order to prove the strongest properties of the automorphism forcing, such as being totally proper, 
preserving the Suslinness of $T$, and not adding new cofinal branches of $\omega_1$-trees, 
we assume in addition that $T$ is a free Suslin tree. 
A feature of this work is that the concepts and arguments we use are quite general 
and can be developed in an abstract way for adding different types of structure 
to a free Suslin tree; this direction is explored in the sequel \cite{KSfree2}.

A natural approach for proving the main consistency result of the article would be to start with a Suslin tree $T$ and an 
inaccessible cardinal $\ka$, and iterate with countable support of length $\ka$ alternating between 
adding automorphisms of $T$ and collapsing $\omega_2$. 
When the forcing is arranged in this manner, by standard forcing iteration preservation results 
(\cite{miyamoto, SHELAHBOOK2, AS2}) 
it suffices to show that each stage of the iteration preserves the Suslinness of $T$ 
and does not add new cofinal branches of $\omega_1$-trees. 
Unfortunately, we are unable to prove either of these statements unless $T$ 
is a free Suslin tree, and adding a single automorphism kills the freeness of $T$. 
In addition, the requirement that the automorphism being added is almost disjoint from those previously added 
presents additional challenges. 
For these reasons, we adopt the less standard approach given in this article.

The main ideas on which our work is based are consistency, separation, and the Key Property. 
These properties are introduced and proven to hold for the automorphism forcing in Sections 
\ref{Consistency and Separation for Automorphisms} and \ref{Constructing and Extending Automorphisms} 
and their main impact is described in Section \ref{More Properties of the Automorphism Forcing}. 
Once this work is done, the verification of the general properties of the forcing, such as being totally proper and 
preserving the Suslinness of $T$, follow by more routine bookkeeping constructions. 
The second half of the article beginning with 
Section \ref{More About Constructing and Extending Automorphisms} 
analyzes quotients of the forcing and proves that quotients 
do not add new cofinal branches of $\omega_1$-trees over intermediate models. 
This material is significantly more difficult than the first half, 
with the most novel and complex arguments of the article appearing in Section \ref{Existence of Nice Conditions} and the beginning of Section \ref{The Automorphism Forcing Adds No New Cofinal Branches}.

\bigskip

\emph{Background and preliminaries:} 
An \emph{$\omega_1$-tree} is a tree 
with height $\omega_1$ whose levels are countable. 
Let $T$ be an $\omega_1$-tree. 
For any $x \in T$, we let $\h_T(x)$ denote the height of $x$ in $T$. 
For each $\alpha < \omega_1$, $T_\alpha = \{ x \in T : \h_T(x) = \alpha \}$ 
is \emph{level $\alpha$ of $T$}, and 
$T \res \alpha = \{ x \in T : \h_T(x) < \alpha \}$. 
For all $x \in T$ and $\alpha \le \h_T(x)$, $x \res \alpha$ denotes 
the unique $y \le_T x$ with height $\alpha$. 
If $X \subseteq T_\beta$ and $\alpha < \beta$, 
$X \res \alpha$ denotes the set $\{ x \res \alpha : x \in X \}$. 
For $\alpha < \beta < \omega_1$ and $X \subseteq T_\beta$, we say that $X$ has 
\emph{unique drop-downs to $\alpha$} if the function $x \mapsto x \res \alpha$ is 
injective on $X$; similar language is used for finite tuples of elements of $T_\beta$.

A \emph{branch} of $T$ is a maximal chain, and a branch is \emph{cofinal} if 
it meets every level of the tree. 
If $b$ is a branch and $\alpha$ is an ordinal 
less than its order type, we write $b(\alpha)$ for the 
unique element of $b$ of height $\alpha$. 
An \emph{antichain} of $T$ is a set of incomparable elements of $T$. 
A \emph{subtree} of $T$ is any subset of $T$ considered as a tree 
with the order inherited from $T$. 
The tree $T$ is \emph{infinitely splitting} 
if every element of $T$ has infinitely many 
immediate successors. 
The tree $T$ is \emph{normal} if it has a root, every element of $T$ has at least two 
immediate successors, every element of $T$ has some element above it at any higher level, 
and any two distinct elements of the same limit height do not have same set of elements 
below them. 

An \emph{Aronszajn tree} is an $\omega_1$-tree with no cofinal branch. 
A tree $T$ of height $\omega_1$ is \emph{special} if it is a union of countably 
many antichains, or equivalently, there exists a 
\emph{specializing function} $f : T \to \q$, which means that 
$x <_T y$ implies that $f(x) < f(y)$. 
A \emph{Kurepa tree} is an $\omega_1$-tree with at least $\omega_2$-many cofinal branches. 
A \emph{Suslin tree} is an uncountable tree with no uncountable chain or uncountable antichain. 
Suslin trees are $\omega_1$-trees. 
A normal $\omega_1$-tree is a Suslin tree if and only if it has no uncountable antichain.

Any $\omega_1$-tree $T$ can be considered as a forcing poset, where we let $y$ be stronger 
than $x$ in the forcing if $x \le_T y$, that is, with the order reversed. 
A normal $\omega_1$-tree $T$ is Suslin if and only if the forcing poset $T$ is c.c.c. 
When we use forcing language such as ``dense'' and ``open'' when talking about 
an $\omega_1$-tree $T$, we mean with regards to $T$ considered 
as a forcing poset as just discussed. 
We highlight the following important fact because we use it 
almost every time we invoke the Suslin property: 
\emph{A normal $\omega_1$-tree $T$ is Suslin if and only if whenever $D$ is a dense open 
subset of $T$, there exists some $\gamma < \omega_1$ such that $T_\gamma \subseteq D$.} 
Note that since $D$ is open, $T_\gamma \subseteq D$ implies that $T_\xi \subseteq D$ for all 
$\gamma \le \xi < \omega_1$.

Given finitely many $\omega_1$-trees $T_0,\ldots,T_{n-1}$, the product 
$T_0 \otimes \cdots \otimes T_{n-1}$ is the partial order, ordered componentwise, 
consisting of all tuples $(a_0,\ldots,a_{n-1})$ such that for some $\alpha < \omega_1$, 
$a_k \in (T_k)_\alpha$ for all $k < n$. 
This product is a tree, and if each factor is normal, then 
so is the product. 
Let $T$ be an $\omega_1$-tree. 
For any positive $n < \omega$, 
we write $T^n$ for the product of $n$-many copies of $T$. 
If $\vec a = (a_0,\ldots,a_{n-1})$ and $\vec b = (b_0,\ldots,b_{n-1})$ are in $T^n$, 
we write $\vec a < \vec b$ to mean that $a_i <_T b_i$ for all $i < n$, and similarly 
for $\vec a \le \vec b$.

For every $a \in T$, define $T_a$ as the subtree $\{ b \in T : a \le_T b \}$. 
For any positive $n < \omega$ and $n$-tuple 
$\vec a = (a_0,\ldots,a_{n-1})$ consisting of distinct elements of $T$ of the same height, 
define $T_{\vec a}$ as the product 
$T_{a_0} \otimes \cdots \otimes T_{a_{n-1}}$, which is called a 
\emph{derived tree of $T$ with dimension $n$}. 
The tree $T$ is said to be \emph{$n$-free} 
if all of its derived trees with dimension $n$ are Suslin, 
and is \emph{free} if it is $n$-free for all positive $n < \omega$. 
Note that by the fact we highlighted in the previous paragraph, if $T$ is $n$-free 
and $T_{\vec a}$ is a derived tree of $T$ with dimension $n$, 
then for any dense open subset 
$D$ of $T_{\vec a}$, there exists some $\gamma < \omega_1$ such that every member of 
$T_{\vec a}$ whose elements have height at least $\gamma$ is in $D$.

A function $f : T \to U$ between trees 
is \emph{strictly increasing} if $x <_T y$ implies $f(x) <_U f(y)$, 
is an \emph{embedding} if $x <_T y$ iff $f(x) <_U f(y)$, 
is \emph{level preserving} if $\h_T(x) = \h_U(f(x))$ for all $x \in T$, 
is an \emph{isomorphism} if it is a bijective embedding, and 
is an \emph{automorphism} if it is an isomorphism and $T = U$.
We use the basic fact that a strictly increasing and level preserving map 
$f : T \to U$ is an embedding if and only if it is injective, and therefore is an isomorphism 
if and only if it is a bijection. 
If $f$ is an automorphism of $T$, we write $f^1$ for $f$ and $f^{-1}$ for the 
inverse of $f$. 
An $\omega_1$-tree $T$ is \emph{rigid} if there does not exist any automorphism 
of $T$ other than the identity function, and is \emph{homogeneous} if for all 
$a$ and $b$ in $T$ with the same height, there exists an automorphism 
$f : T \to T$ such that $f(a) = b$. 
For an $\omega_1$-tree $T$, $\sigma(T)$ denotes the cardinality of the set of all 
automorphisms of $T$.

When we say that a family of sets (or sequence of sets) 
is \emph{almost disjoint}, we mean that the intersection of any two 
sets in the family (or in the sequence) is countable. 
An \emph{almost Kurepa Suslin tree} is a Suslin tree such that when you force with it, 
it becomes a Kurepa tree. 
A sufficient condition for a Suslin tree $T$ 
to be an almost Kurepa Suslin tree is that there 
exists an almost disjoint family $\{ f_\tau : \tau < \omega_2 \}$ 
of automorphisms of $T$. 
For in that case, if $b$ is a cofinal branch of $T$, then 
$\{ f_\tau[b] : \tau < \omega_2 \}$ is a family of $\omega_2$-many cofinal branches of $T$. 
An \emph{antichain of subtrees} of an Aronszajn tree $T$ is an almost disjoint 
family of uncountable downwards closed subtrees of $T$. 
An Aronszajn tree $T$ is \emph{saturated} if every antichain of subtrees of $T$ 
has size at most $\omega_1$, and otherwise is \emph{non-saturated}.

When we say that a regular cardinal $\lambda$ is \emph{large enough}, we mean that it is 
large enough so that all of the sets under discussion are members of $H(\lambda)$. 
For a forcing poset $\p$ and a countable elementary substructure $N \prec H(\lambda)$ 
with $\p \in N$, a condition $q \in \p$ is a \emph{total master condition over $N$} if 
for every dense open subset $D$ of $\p$ which is a member of $N$, 
there exists some $s \in D \cap N$ such that $q \le s$. 
A forcing poset $\p$ is \emph{totally proper} if for all large enough regular cardinals 
$\lambda$ and for any countable elementary substructure $N \prec H(\lambda)$, for all 
$p \in N \cap \p$ there exists some $q \le p$ such that $q$ is a total master 
condition over $N$. 
Clearly, totally proper forcings are proper and countably distributive. 
A separative forcing is totally proper if and only if it is proper and does not add reals. 
The L\'{e}vy collapse of an inaccessible cardinal $\ka$ to become $\omega_2$, 
denoted by $\col(\omega_1, < \! \ka)$, is the forcing poset consisting of all 
countable partial functions $p$ from $\ka \times \omega_1$ into $\ka$ such that for all 
$(\alpha,\xi) \in \dom(p)$, $p(\alpha,\xi) < \alpha$, ordered by reverse inclusion. 
The L\'{e}vy collapse is $\omega_1$-closed and $\ka$-c.c. 
Finally, we note that $\omega_1$-closed forcings do not add new cofinal branches 
of $\omega_1$-trees in the ground model (\cite[Chapter V \S 8]{SHELAHBOOK2}).

\section{Consistency and Separation for Automorphisms} 
\label{Consistency and Separation for Automorphisms}

Assume for the remainder of the article that $T$ is a fixed $\omega_1$-tree 
which is normal and infinitely splitting and $\ka$ is a fixed non-zero ordinal. 
Additional assumptions about $T$ and $\ka$ are made on occasion, including most 
notably that $T$ is a free Suslin tree or that $\ka$ is an inaccessible cardinal.

If $\beta < \omega_1$, $g$ is an automorphism of $T \res (\beta+1)$, and $\alpha < \beta$, 
we write $g \res (\alpha+1)$ for the restriction of $g$ to 
$T \res (\alpha+1)$, which is an automorphism of $T \res (\alpha+1)$. 
If $\mathcal G = \{ g_\tau : \tau \in I \}$ is an indexed family of automorphisms of 
$T \res (\beta+1)$, we write 
$\mathcal G \res (\alpha+1)$ for the indexed family 
$\{ g_\tau \res (\alpha+1) : \tau \in I \}$.\footnote{Despite the set notation, by 
$\{ g_\tau : \tau \in I \}$ we mean the function with domain $I$ which maps 
$\tau \in I$ to $g_\tau$.}

\begin{definition}[Consistency]
	Let $\alpha < \beta < \omega_1$ and let $g$ be an automorphism 
	of $T \res (\beta+1)$. 
	\begin{enumerate}
	\item Let $X \subseteq T_\beta$ be finite with unique drop-downs to $\alpha$. 
	We say that $X \res \alpha$ and $X$ are \emph{$g$-consistent} 
	if for all $x, y \in X$, $g(x \res \alpha) = y \res \alpha$ iff $g(x) = y$.
	\item Let $\vec a = (a_0,\ldots,a_{n-1})$ be an injective tuple 
	consisting of elements of $T_\beta$. 
	We say that $\vec a \res \alpha$ and $\vec a$ are \emph{$g$-consistent} if 
	for all $i, j < n$, $g(a_i \res \alpha) = a_j \res \alpha$ iff 
	$g(a_i) = a_j$.
	\end{enumerate}
\end{definition}

Note that in (1) above, 
the sets $X \res \alpha$ and $X$ are $g$-consistent 
if and only if for all $x, y \in X$, 
$g(x \res \alpha) = y \res \alpha$ implies $g(x) = y$. 
The reverse implication follows from the fact that $g$ is strictly increasing. 
A similar comment applies to (2).

The following lemma is easy to check.

\begin{lemma}[Transitivity] \label{Transitivity 1}
	Let $\alpha < \beta < \gamma < \omega_1$ and let $X \subseteq T_\gamma$ be finite 
	with unique drop-downs to $\alpha$. 
	Let $g$ be an automorphism of $T \res (\gamma+1)$. 
	If $X \res \alpha$ and $X \res \beta$ are $g \res (\beta+1)$-consistent and 
	$X \res \beta$ and $X$ are $g$-consistent, 
	then $X \res \alpha$ and $X$ are $g$-consistent. 
\end{lemma}

\begin{definition}[Separation]
	Let $\alpha < \omega_1$. 
	Suppose that $\mathcal G = \{ g_\tau : \tau \in I \}$ is an indexed family of 
	automorphisms of $T \res (\alpha+1)$ 
	and $\vec a = (a_0,\ldots,a_{n-1})$ consists of distinct elements of $T_\alpha$. 
	We say that $\mathcal G$ is \emph{separated on $\vec a$} 
	if for all $k < n$:
	\begin{enumerate}
		\item for all $\tau \in I$, $g_\tau(a_k) \ne a_k$;
		\item there exists at most one triple $(i,m,\tau)$, where 
		$i < k$, $m \in \{ -1, 1 \}$, and $\tau \in I$, such that 
		$g_\tau^m(a_k) = a_i$.
	\end{enumerate}
\end{definition}

We sometimes refer to an equation of the form 
$g_\tau^m(a_k) = a_i$ as in (2) above as a \emph{relation}. 
So separation means that each member of the tuple has at most one relation 
with previous members of the tuple, and no relation with itself. 
The way in which a tuple is ordered is essential to whether or not separation holds.

\begin{lemma}[Persistence] \label{Persistence 1}
	Let $\alpha < \beta < \omega_1$ and let 
	$\mathcal G = \{ g_\tau : \tau \in I \}$ be 
	an indexed family of automorphisms of $T \res (\beta+1)$. 
	Let $\vec b = (b_0,\ldots,b_{n-1})$ 
	consist of distinct elements of $T_\beta$ 
    with unique drop-downs to $\alpha$. 
	If the indexed family $\mathcal G \res (\alpha+1)$ is separated on 
	$\vec b \res \alpha$, then 
	$\mathcal G$ is separated on $\vec b$.
\end{lemma}

\begin{proof}
	For all $\tau \in I$, the fact that $g_\tau$ is strictly increasing 
	implies that for all $i, j < n$, if $g_\tau(b_i) = b_j$ 
	then $g_\tau(b_i \res \alpha) = b_j \res \alpha$. 
	So any violation of separation of $\mathcal G$ on $\vec b$ would imply a 
	violation of separation of $\mathcal G \res (\alpha+1)$ on $\vec b \res \alpha$.
\end{proof}

\begin{lemma} \label{nice ordering}
	Let $\alpha < \omega_1$. 
	Let $\vec a = (a_0,\ldots,a_{n-1})$ consist of 
	distinct elements of $T_\alpha$, and let 
	$\mathcal G = \{ g_\tau : \tau \in A \}$ be a finite 
	indexed family of automorphisms of $T \res (\alpha+1)$. 
	Then for all $\bar{n} < n$, there exists a sequence 
	$$
	\langle i_0,(i_1,m_1,\tau_1),\ldots,(i_{l-1},m_{l-1},\tau_{l-1}) \rangle, 
	$$
	for some $l \le \bar{n}+1$, such that:
	\begin{enumerate}
	\item $\bar{n} = i_0 > i_1 > \cdots > i_{l-1} \ge 0$;
	\item for all $0 < k < l$, $m_k \in \{ -1, 1 \}$, 
	$\tau_k \in A$, and 
	$g_{\tau_k}^{m_k}(a_{i_{k-1}}) = a_{i_{k}}$; 
	\item there does not exist a triple $(i,m,\tau)$ such that 
	$i < i_{l-1}$, $m \in \{ -1, 1 \}$, 
	$\tau \in A$, and $g_\tau^m(a_{i_{l-1}}) = a_i$.
	\end{enumerate}
	Moreover, if $\mathcal G$ is separated on $\vec a$, then the above sequence 
	is unique.
\end{lemma}

\begin{proof}
	We construct the sequence by induction. 
	Let $i_{0} = \bar{n}$. 
	Now let $k \ge 0$ and assume that we have defined 
	$\langle i_0,(i_1,m_1,\tau_1),\ldots,(i_{k},m_{k},\tau_{k}) \rangle$ 
	as described in (1) and (2). 
	If there does not exist a triple $(i,m,\tau)$ such that $i < i_k$, $m \in \{ -1, 1 \}$, 
	$\tau \in A$, and $g_\tau^m(a_{i_k}) = a_i$, 
	then let $l = k+1$ and we are done. 
	Otherwise fix such a triple $(i,m,\tau)$ (which is unique in the case that 
	$\mathcal G$ is separated on $\vec a$), and let 
	$i_{k+1} = i$, $m_{k+1} = m$, and $\tau_{k+1} = \tau$. 
	This completes the construction. 
	Note that (1) implies that $l \le \bar{n}+1$.
\end{proof}

While separation in the context of automorphisms 
depends on the way in which a tuple is ordered, 
the following notion of 
separation for sets is useful when we do not need to be explicit about what that order is.

\begin{definition}[Separation for Sets] \label{Separation for Sets 1}
	Let $\alpha < \omega_1$. 
	Suppose that $\mathcal G = \{ g_\tau : \tau \in I \}$ 
	is an indexed family of automorphisms of $T \res (\alpha+1)$ 
	and $X$ is a finite subset of $T_\alpha$. 
	We say that $\mathcal G$ is \emph{separated on $X$} if there exists some injective 
	tuple $\vec a$ which lists the elements of $X$ 
	such that $\mathcal G$ is separated on $\vec a$.
\end{definition}

\begin{lemma}[Persistence for Sets] \label{Persistence for Sets 1}
	Let $\alpha < \beta < \omega_1$ and let $\mathcal G = \{ g_\tau : \tau \in I \}$ 
	be an indexed family of automorphisms of $T \res (\beta+1)$. 
	Let $X \subseteq T_\beta$ be a finite set with unique drop-downs to $\alpha$. 
	If the indexed family $\mathcal G \res (\alpha+1)$ is separated on $X \res \alpha$, 
	then $\mathcal G$ is separated on $X$.
\end{lemma}

\begin{proof}
	Let $\vec a$ be an injective tuple which lists the elements of $X$ so that 
	$\mathcal G \res (\alpha+1)$ is separated on $\vec a \res \alpha$. 
	Now apply Lemma \ref{Persistence 1} (Persistence).
\end{proof}

The proof of the following lemma is easy.

\begin{lemma} \label{composition}
	Let $\alpha < \omega_1$ and let $X$ be a finite subset of $T_\alpha$. 
	Suppose that $\{ f_i : i \in I \}$ and $\{ g_j : j \in J \}$ 
	are indexed families of automorphisms of $T \res (\alpha+1)$, 
	$h : I \to J$ is a bijection, and for all $i \in I$, $f_i = g_{h(i)}$. 
	If $\{ g_j : j \in J \}$ is separated on $X$, then 
	$\{ f_i : i \in I \}$ is separated on $X$.
\end{lemma}

\begin{lemma} \label{subset}
	Let $\alpha < \omega_1$. 
	Suppose that $\mathcal G = \{ g_\tau : \tau \in I \}$ 
	is an indexed family of automorphisms of $T \res (\alpha+1)$ 
	and $X$ is a finite subset of $T_\alpha$. 
	If $\mathcal G$ is separated on $X$, then for any 
	$J \subseteq I$ and $Y \subseteq X$, $\{ g_\tau : \tau \in J \}$ is 
	separated on $Y$.
\end{lemma}

\begin{proof}
	Let $\vec a$ be an injective tuple which lists the elements 
	of $X$ such that $\mathcal G$ is separated on $\vec a$. 
	Let $\vec b$ be an injective tuple which lists the elements 
	of $Y$ in the same order in which they appear in $\vec a$. 
	Now any counter-example to the indexed family 
	$\{ g_\tau : \tau \in J \}$ being separated on $\vec b$ would yield a 
	counter-example to $\mathcal G$ being separated on $\vec a$.
\end{proof}

\begin{definition}
	Let $\alpha < \omega_1$. 
	Suppose that $\mathcal G = \{ g_\tau : \tau \in I \}$ is an indexed 
	family of automorphisms of $T \res (\alpha+1)$. 
	We say that $\mathcal G$ is \emph{separated} if for any finite set 
	$X \subseteq T_\alpha$, $\{ g_\tau : \tau \in I \}$ 
	is separated on $X$.
\end{definition}

\section{The Key Property for Automorphisms} \label{The Key Property for Automorphisms}

In this section, we prove two results which imply that the automorphism forcing we introduce in 
Section \ref{The Automorphism Forcing and Its Basic Properties} below 
satisfies the Key Property and the $1$-Key Property. 
These two properties are used to verify properness and the preservation of Suslinness respectively.

\begin{proposition}[Key Property] \label{Key Property 1}
	Let $\alpha < \beta < \omega_1$. 
	Suppose that $a_0,\ldots,a_{n-1}$ are distinct elements 
	of $T_\alpha$ and $\mathcal G = \{ g_\tau : \tau \in A \}$ is a finite 
	indexed set of automorphisms of $T \res (\beta+1)$ such that 
	$\mathcal G \res (\alpha+1)$ is separated on $(a_0,\ldots,a_{n-1})$. 
	Let $t \subseteq T_\beta$ be finite. 
	Then there exist $b_0,\ldots,b_{n-1}$ in $T_\beta \setminus t$ 
	such that $a_i <_T b_i$ for all $i < n$, and 
	for all $\tau \in A$, 
	$(a_0,\ldots,a_{n-1})$ and $(b_0,\ldots,b_{n-1})$ are $g_\tau$-consistent.
\end{proposition}

\begin{proof}
	Let $t^*$ be the set of all $y \in T_\beta$ such that either $y \in t$, or 
	$y = g_{\tau_{l-1}}^{m_{l-1}}(\cdots(g_{\tau_1}^{m_1}(x))$, 
	for some $x \in t$, $l \le n+1$, $\tau_1,\ldots,\tau_{l-1} \in A$, and 
	$m_1,\ldots,m_{l-1} \in \{ -1, 1 \}$. 
	Note that $t^*$ is finite.

	By induction on $k < n$, we choose $b_k \in T_\beta$ above $a_k$, 
	maintaining that for all $k < n$:
	\begin{itemize}
	\item[(a)] for all $\tau \in A$, 
	$(a_0,\ldots,a_{k})$ and $(b_0,\ldots,b_{k})$ are $g_\tau$-consistent;
	\item[(b)] if there does not exist a triple $(j,m,\tau)$, where $j < k$, 
	$m \in \{ -1, 1 \}$, and $\tau \in A$, 
	such that $g_\tau^m(a_{k}) = a_j$, then $b_{k} \notin t^*$. 
	\end{itemize}

	For the base case, let $b_0$ be an arbitrary member of $T_\beta \setminus t^*$ 
	above $a_0$, which is possible 
	since $T$ is infinitely splitting. 
	Consider any $\tau \in A$. 
	Since $\mathcal G$ is separated on $(a_0,\ldots,a_{n-1})$, 
	$g_\tau(a_0) \ne a_0$, which implies that 
	$g_\tau(b_0) \ne b_0$. 
	So $(a_0)$ and $(b_0)$ are $g_\tau$-consistent.

	Now let $0 < k < n$ be given, 
	and assume that we have chosen $b_i$ for all 
	$i < k$ so that 
	the tuple $(b_0,\ldots,b_{k-1})$ satisfies 
	the inductive hypotheses.

	Case 1: There does not exist a triple $(j,m,\tau)$, where 
	$j < k$, $m \in \{ -1, 1 \}$, and $\tau \in A$, 
	such that $g_\tau^m(a_{k}) = a_j$. 
	In this case, let $b_{k}$ be an arbitrary member of 
	$T_\beta \setminus t^*$ above $a_{k}$, 
	which is possible since $T$ is infinitely splitting. 
	Since there are no relations between $a_{k}$ and members of the 
	tuple $(a_0,\ldots,a_{k-1})$, 
	the inductive hypothesis together with the argument 
	we gave for the base case easily imply that 
	for all $\tau \in A$, 
	$(a_0,\ldots,a_{k})$ and $(b_0,\ldots,b_{k})$ are $g_\tau$-consistent. 
	So inductive hypothesis (a) holds, and (b) is immediate 
	from the choice of $b_{k}$.

	Case 2:	There exists a triple $(j,m,\sigma)$, where 
	$j < k$, $m \in \{ -1, 1 \}$, and $\sigma \in A$, 
	such that $g_\sigma^m(a_{k}) = a_j$. 
	Hence, $a_k = g_{\sigma}^{-m}(a_j)$. 
	Since $\mathcal G \res (\alpha+1)$ is separated on $(a_0,\ldots,a_{n-1})$, 
	this triple is unique. 
	Let $b_{k} = g_\sigma^{-m}(b_{j})$. 
	By the uniqueness of the triple $(j,m,\sigma)$, the inductive hypotheses, 
	and the argument we gave in the base case, 
	it easily follows that for all $\tau \in A$, 
	$(a_0,\ldots,a_{k})$ and $(b_0,\ldots,b_{k})$ are $g_\tau$-consistent. 
	So inductive hypothesis (a) holds, and (b) holds vacuously.

	It remains to show that for all $\bar{n} < n$, $b_{\bar{n}} \notin t$. 
	Suppose for a contradiction that there exists some 
	$\bar{n} < n$ such that $b_{\bar{n}} \in t$. 
	Applying Lemma \ref{nice ordering}, fix a sequence 
	$$
	\langle i_0,(i_1,m_1,\tau_1),\ldots,(i_{l-1},m_{l-1},\tau_{l-1}) \rangle,
	$$
	for some $l \le \bar{n}+1$, such that:
	\begin{enumerate}
	\item $\bar{n} = i_0 > i_1 > \cdots > i_{l-1} \ge 0$;
	\item for all $0 < k < l$, $m_k \in \{ -1, 1 \}$, 
	$\tau_k \in A$, and $g_{\tau_k}^{m_k}(a_{i_{k-1}}) = a_{i_{k}}$; 
	\item there does not exist a triple $(i,m,\tau)$ such that 
	$i < i_{l-1}$, $m \in \{ -1, 1 \}$, 
	$\tau \in I$, and $g_\tau^m(a_{i_{l-1}}) = a_i$.
	\end{enumerate}
	By (3) and inductive hypothesis (b), $b_{i_{l-1}} \notin t^*$. 
	By (2) and Case 2, we have that 
	$$
	b_{i_{l-1}} = g_{\tau_{l-1}}^{m_{l-1}}(\cdots(g_{\tau_1}^{m_1}(b_{i_0}))).
	$$
	So $b_{i_{l-1}} \in t^*$, which is a contradiction.
\end{proof}

\begin{proposition}[$1$-Key Property] \label{1-Key Property 1}
	Let $\alpha < \beta < \omega_1$. 
	Suppose that $a_0,\ldots,a_{n-1}$ are distinct elements 
	of $T_\alpha$ and $\mathcal G = \{ g_\tau : \tau \in A \}$ is a finite 
	indexed set of automorphisms of $T \res (\beta+1)$ such that 
	$\mathcal G \res (\alpha+1)$ is separated on $(a_0,\ldots,a_{n-1})$. 
	Let $\bar{n} < n$. 
	Fix $b \in T_\beta$ such that $a_{\bar{n}} <_T b$. 
	Then there exist $b_0,\ldots,b_{n-1}$ in $T_\beta$ 
	such that $a_i <_T b_i$ for all $i < n$, $b_{\bar{n}} = b$, 
	and for all $\tau \in A$, 
	$(a_0,\ldots,a_{n-1})$ and $(b_0,\ldots,b_{n-1})$ are $g_\tau$-consistent.
\end{proposition}

\begin{proof}
	Apply Lemma \ref{nice ordering} to find some $l \le \bar{n}+1$ and a sequence 
	$$
	\langle i_0,(i_1,m_1,\tau_1),\ldots,(i_{l-1},m_{l-1},\tau_{l-1}) \rangle,
	$$
	satisfying (1)-(3) of that lemma. 
	In particular, $i_0 = \bar{n}$. 
	Define $(c_0,\ldots,c_{l-1})$ inductively by letting $c_0 = b$, 
	and for all $0 < k < l$, 
	$c_{k} = g_{\tau_k}^{m_k}(c_{k-1})$. 
	Using the fact that $a_{\bar{n}} <_T b$ and (2) of Lemma \ref{nice ordering}, 
	it is easy to prove by induction that for all $k < l$, $a_{i_k} <_T c_k$.

	By induction on $i < n$ we choose $b_i$ in $T_\beta$ above $a_i$, 
	maintaining that for all $k < n$:
	\begin{enumerate}
	\item[(a)] for all $\tau \in A$, 
	$(a_0,\ldots,a_{k})$ and $(b_0,\ldots,b_{k})$ are $g_\tau$-consistent;
	\item[(b)] for all $m < l$, if $i_m \le k$ then $b_{i_m} = c_m$.
	\end{enumerate} 
	Assuming that we are able to define $(b_0,\ldots,b_{n-1})$ 
	with these properties, then for all $\tau \in A$, 
	$(a_0,\ldots,a_{n-1})$ and $(b_0,\ldots,b_{n-1})$ are $g_\tau$-consistent, 
	and $b_{\bar{n}} = b_{i_0} = c_0 = b$, which completes the proof.

	For the base case, if $0 \in \{ i_0, \ldots, i_{l-1} \}$, 
	then clearly $0 = i_{l-1}$, so in this case we let $b_0 = b_{i_{l-1}} = c_{l-1}$. 
	Otherwise, let $b_0$ be an arbitrary element of $T_\beta$ above $a_0$. 
	Consider any $\tau \in A$. 
	Since $\mathcal G \res (\alpha+1)$ is separated on $(a_0,\ldots,a_{n-1})$, 
	$g_\tau(a_0) \ne a_0$, and hence $g_\tau(b_0) \ne b_0$. 
	So $(a_0)$ and $(b_0)$ are $g_\tau$-consistent. 
	Clearly, the inductive hypotheses are maintained.

	Now let $0 < k < n$ and assume that we have chosen $b_i$ for all 
	$i < k$ so that $(b_0,\ldots,b_{k-1})$ satisfies the inductive hypotheses.

	Case 1:	There does not exist a triple $(j,m,\tau)$ such that 
	$j < k$, $m \in \{ -1, 1 \}$, $\tau \in A$, and 
	$g_{\tau}^m(a_k) = a_j$. 
	If $k \in \{ i_0, \ldots, i_{l-1} \}$, then clearly $k = i_{l-1}$, 
	and we let $b_{k} = b_{i_{l-1}} = c_{l-1}$. 
	So inductive hypothesis (b) holds. 
	Otherwise, choose $b_{k}$ above $a_{k}$ arbitrarily. 
	For all $\tau \in A$, $g_\tau(a_k) \ne a_k$, 
	which implies that $g_\tau(b_k) \ne b_k$. 
	So the inductive hypothesis together with the fact that 
	there are no relations between $a_k$ and members of $(a_0,\ldots,a_{k-1})$ 
	easily imply inductive hypothesis (a).

	Case 2:	There exists a triple $(j,m,\sigma)$ such that 
	$j < k$, $m \in \{ -1, 1 \}$, $\sigma \in A$, and $g_\sigma^m(a_k) = a_j$. 
	Then $a_k = g_\sigma^{-m}(a_j)$. 
	Define $b_{k} = g_\sigma^{-m}(b_j)$. 
	By the inductive hypothesis, 
	the uniqueness of the triple $(j,m,\sigma)$, 
	and the fact that $g_\tau(b_k) \ne b_k$ 
	for all $\tau \in A$, 
	it easily follows that for all $\tau \in A$, 
	$(a_0,\ldots,a_{k})$ and $(b_0,\ldots,b_{k})$ 
	are $g_\tau$-consistent. 
	In the case that $k \in \{ i_0, \ldots, i_{l-1} \}$, 
	by the uniqueness of the triple $(j,m,\sigma)$ and the assumption of Case 2 
	it must be the case that $k = i_{q-1}$ for some $q$ such that 
	$0 < q \le l-1$, $j = i_{q}$, $m = m_{q}$, and $\sigma = \tau_{q}$. 
	By the induction hypothesis, 
	$b_{i_{q}} = c_{q}$, and by definition of $b_{i_{q-1}}$ 
	and $c_{q}$, 
	$b_{i_{q-1}} = g_{\tau_{q}}^{-m_{q}}(b_{i_{q}}) = 
	g_{\tau_{q}}^{-m_{q}}(c_{q}) = c_{q-1}$.
\end{proof}

\section{Constructing and Extending Automorphisms} \label{Constructing and Extending Automorphisms}

We give several bookkeeping constructions for consistently extending automorphisms to higher levels while preserving separation. 
These arguments are fairly routine and can be optionally skipped 
by the reader.

\begin{lemma} \label{extending by one}
	Let $\gamma < \omega_1$ and let $\{ f_\tau : \tau \in I \}$ be a countable family 
	of automorphisms of $T \res (\gamma+1)$. 
	Then there exists a family $\{ g_\tau : \tau \in I \}$ of automorphisms 
	of $T \res (\gamma+2)$ such that:
	\begin{enumerate}
	\item for all $\tau \in I$, 
	$f_\tau \subseteq g_\tau$;
	\item for all distinct $\tau_0$ and $\tau_1$ in $I$ and 
	for all $x \in T_{\gamma+1}$, $g_{\tau_0}(x) \ne g_{\tau_1}(x)$.
	\end{enumerate}
\end{lemma}

\begin{proof}
	Fix a bijection $h : \omega \to T_{\gamma+1} \times I$. 
	Let $g_\tau \res (\gamma+1) = f_\tau$ for all $\tau \in I$. 
	We define the values of each $g_\tau$ on $T_{\gamma+1}$ 
	in $\omega$-many stages, where at any given stage 
	we have defined only finitely many values of finitely many $g_\tau$'s. 
	We also define a subset-increasing sequence $\langle X_n : n < \omega \rangle$ 
	of finite subsets of $T_{\gamma+1}$.
	
	At stage $0$, we do nothing. 
	Let $X_0 = \emptyset$. 
	Now let $n < \omega$ and suppose that stage $n$ is complete. 
	In particular, the finite set $X_n \subseteq T_{\gamma+1}$ has been defined. 
	Let $h(n) = (z,\sigma)$. 
	Stage $n+1$ consists of two steps. 
	In the first step, if $g_\sigma(z)$ is already defined, then move on to step $2$. 
	Otherwise, define $g_\sigma(z)$ to be some element of $T_{\gamma+1}$ 
	above $f_\sigma(z \res \gamma)$ which is not in $X_n$. 
	In the second step, if $g_\sigma^{-1}(z)$ is already defined, then we are done. 
	If not, then define $g_\sigma^{-1}(z)$ to be some element of $T_{\gamma+1}$ 
	above $f_\sigma^{-1}(z \res \gamma)$ which is not in $X_n$. 
	Now let $X_{n+1} = X_n \cup \{ z, g_\sigma(z), g_\sigma^{-1}(z) \}$. 

	This completes the construction. 
	It is easy to check that for all $\tau \in I$, $g_\tau$ is a strictly increasing 
	function from $T \res (\gamma+2)$ onto $T \res (\gamma+2)$. 
	Suppose for a contradiction that for some $z \in T_{\gamma+1}$, 
	$g_{\tau_0}(z) = g_{\tau_1}(z)$ for distinct $\tau_0$ and $\tau_1$ in $I$. 
	Assume that $g_{\tau_0}(z)$ was defined at stage $n$ and $g_{\tau_1}(z)$ was 
	defined at stage $m$, where $n < m$. 
	At stage $n$, either $h(n) = (\tau_0,z)$ and we defined $g_{\tau_0}(z)$, or 
	for some $y_0$, $h(n) = (\tau_0,y_0)$ and we defined $z = g_{\tau_0}^{-1}(y_0)$. 
	In either case, both $z$ and $g_{\tau_0}(z)$ are in $X_n$. 
	At stage $m$, either 
	$h(n) = (\tau_1,z)$ and we defined $g_{\tau_1}(z)$ which is not in $X_n$, or 
	for some $y_1$, $h(n) = (\tau_1,y_1)$ and we defined $z = g_{\tau_1}^{-1}(y_1)$ 
	which is not in $X_n$. 
	In the first case, $g_{\tau_1}(z)$ cannot equal $g_{\tau_0}(z)$ since the 
	latter element is in $X_n$, and the second case is impossible since $z \in X_n$. 
	So we have a contradiction. 
	A similar argument shows that each $g_\tau$ is injective, and hence is an 
	automorphism of $T \res (\gamma+2)$.
\end{proof}

\begin{lemma} \label{more extending by one}
	Assume the following:
	\begin{itemize}
	\item $\gamma < \omega_1$;
	\item $X \subseteq T_{\gamma+1}$ is finite and has unique drop-downs to $\gamma$;
	\item $\{ f_\tau : \tau \in I \}$ is a 
	countable collection of automorphisms of $T \res (\gamma+1)$;
	\item $A \subseteq I$ is finite.
	\end{itemize}
	Then there exists a family $\{ g_\tau : \tau \in I \}$ of automorphisms 
	of $T \res (\gamma+2)$ satisfying:
	\begin{enumerate}
	\item $f_\tau \subseteq g_\tau$ for all $\tau \in I$;
	\item for all $\tau \in A$, $X \res \gamma$ and $X$ are $g_\tau$-consistent;
	\item if $\{ f_\tau : \tau \in A \}$ is separated on $X \res \gamma$, 
	then $\{ g_\tau : \tau \in I \}$ is separated.
	\end{enumerate}
\end{lemma}

\begin{proof}
	Fix a bijection $h : \omega \to T_{\gamma+1} \times I$. 
	For each $\tau \in I$, define $g_\tau \res (\gamma+1) = f_\tau$.

	We define the values of the functions $g_\tau$ on $T_{\gamma+1}$ 
	in $\omega$-many stages. 
	The following describes the construction:
	\begin{itemize}
	\item at any given stage $n$, we have defined only finitely many 
	values of the functions $g_\tau$ for finitely many $\tau \in I$;
	\item we define a sequence $\langle a_k : k < \omega \rangle$ 
	which enumerates $T_{\gamma+1}$, 
	where at any stage $n$ we have defined $\langle a_k : k < l_n \rangle$ for some 
	$l_n < \omega$, and let $X_n = \{ a_k : k < l_n \}$.
	\end{itemize}
	We maintain the following inductive hypotheses:
	\begin{enumerate}
	\item[(i)] for all $n$, if the value $g^m_\tau(a) = b$ is defined at stage $n$, 
	where $\tau \in I$ and $m \in \{ -1, 1 \}$, 
	then $a$ and $b$ are in $X_n$, 
	$f^m_\tau(a \res \gamma) = b \res \gamma$, and 
	if $\{ f_\tau : \tau \in A \}$ is separated on $X \res \gamma$, then $a \ne b$;
	\item[(ii)] for all $n_0 < n_1$, if $a$ and $b$ are in $X_{n_0}$ and 
	$g^m_\tau(a) = b$ has been defined by the end of stage $n_1$, 
	where $\tau \in I$ and $m \in \{ -1, 1 \}$, then 
	$g^m_\tau(a) = b$ has been defined by the end of stage $n_0$;
	\item[(iii)] in the case that 
	$\{ f_\tau : \tau \in A \}$ is separated on $X \res \gamma$, 
	then for all $n$ and $k < l_n$ there exists at most one triple $(j,m,\tau)$, 
	where $j < k$, $m \in \{ -1, 1 \}$, 
	and $\tau \in I$, such that $g_\tau^m(a_k)$ has been defined 
	by stage $n$ and $g^m_\tau(a_k) = a_j$.
	\end{enumerate}

	\underline{Stage $0$:} For each $\tau \in A$ and $x, y \in X$, define 
	$g_\tau(x) = y$ iff $f_\tau(x \res \gamma) = y \res \gamma$. 
	Let $A_0 = A$ and let $l_0 = |X|$. 
	In the case that $\{ f_\tau : \tau \in A \}$ is separated on $X \res \gamma$, 
	fix an injective enumeration $\vec a = (a_0,\ldots,a_{l_0-1})$ of $X$ 
	such that $\{ f_\tau : \tau \in A \}$ is separated on $\vec a \res \gamma$. 
	Otherwise, let $\vec a = (a_0,\ldots,a_{l_0-1})$ be 
	an arbitrary injective enumeration of $X$. 
	It is easy to check that the required properties hold.

	\underline{Stage $n+1$:} Let $n < \omega$ and suppose that stage $n$ is complete. 
	Let $h(n) = (z,\sigma)$. 
	Stage $n+1$ consists of two steps. 
	In the first step, if $g_{\sigma}(z)$ is already defined, 
	then move on to step two. 
	Otherwise, define $g_{\sigma}(z)$ to be some member of $T_{\gamma+1}$ 
	above $f_\sigma(z \res \gamma)$ which is 
	not in $X_n \cup \{ z \}$. 
	This is possible since $T$ is infinitely splitting.  
	In the second step, if $g_{\sigma}^{-1}(z)$ is already defined, then we are done. 
	Otherwise, define $g_{\sigma}^{-1}(z)$ to be some member of $T_{\gamma+1}$ 
	above $g_\sigma^{-1}(z \res \gamma)$ 
	which is not in $X_n \cup \{ z, g_{\sigma}(z) \}$. 
	Again, this is possible since $T$ is infinitely splitting. 
	Now define $\langle a_k : k < l_{n+1} \rangle$ by adding at the end of the sequence 
	$\langle a_k : k < l_n \rangle$ the elements among 
	$z$, $g_{\sigma}(z)$, and $g_{\sigma}^{-1}(z)$ which are not already in $X_n$, 
	in the order just listed.

	Let us check that inductive hypotheses (i)-(iii) hold. 
	(i) is clear. 
	For (ii), the only new equations of the form 
	$g^m_\tau(a) = b$ which were introduced at stage $n+1$, where 
	$\tau \in I$, $m \in \{ -1, 1 \}$, and $a, b \in T_{\gamma+1}$, 
	is when $\tau = \sigma$ 
	and at least one of $a$ or $b$ is in $X_{n+1} \setminus X_n$. 
	So (ii) easily follows from the inductive hypothesis.

	Now we prove (iii). 
	Assume that $\{ f_\tau : \tau \in A \}$ is separated on $X \res \gamma$. 
	Consider first the case when $z$ is not in $X_n$. 
	Then by inductive hypothesis (i), neither 
	$g_\sigma(z)$ nor $g_\sigma^{-1}(z)$ were defined at any earlier stage. 
	So by definition, $g_\sigma(z)$ and $g_{\sigma}^{-1}(z)$ are not in $X_n$. 
	Hence, the last three elements of $\langle a_k : k < l_{n+1} \rangle$ 
	are $z$, $g_\sigma(z)$, and $g_{\sigma}^{-1}(z)$. 
	The relations introduced between these three elements at stage $n+1$ 
	yield no counter-example to (iii), and $z$, $g_\sigma(z)$, and $g_\sigma^{-1}(z)$ 
	have no relations to any elements of $\langle a_k : k < l_n \rangle$. 
	(iii) follows from these observations and the inductive hypothesis.
	
	Next, consider the case when $z$ is in $X_n$. 
	Then $z$ already appears on the sequence $\langle a_k : k < l_n \rangle$. 
	At stage $n+1$, no new relations are introduced between elements of 
	$\langle a_k : k < l_n \rangle$. 
	Each new element in $X_{n+1} \setminus X_n$ has exactly one relation between 
	any other member of $X_{n+1}$, namely $z$. 
	(iii) follows from these observations and the inductive hypothesis.

	This completes the construction. 
	It is straightforward to check that 
	each $g_\tau$ is an automorphism of $T \res (\gamma+2)$. 
	By what we did at stage $0$, for all $\tau \in A$, 
	$X \res \gamma$ and $X$ are $g_\tau$-consistent. 
	Now assume that $\{ f_\tau : \tau \in A \}$ is separated on $X \res \gamma$. 
	To show that $\{ g_\tau : \tau \in I \}$ is separated, 
	let $Y \subseteq T_{\gamma+1}$ be finite. 
	Fix a large enough $n$ so that $Y \subseteq X_n$. 
	Then by Lemma \ref{subset}, 
	it suffices to show that $\{ g_\tau : \tau \in I \}$ is separated on $X_n$, 
	as witnessed by the tuple $(a_0,\ldots,a_{l_n-1})$. 
	Suppose that $k < l_n$ and the triple $(j,m,\tau)$ satisfies that 
	$j < k$, $m \in \{ -1, 1 \}$, $\tau \in I$, and 
	$g_{\tau}^{m}(a_k) = a_j$. 
	By inductive hypothesis (ii), 
	the relation $g_\tau^m(a_k) = a_j$ was introduced by the end of stage $n$. 
	By inductive hypothesis (iii), there is at most one such triple.
\end{proof}

\begin{proposition} \label{general extending}
	Assume the following:
	\begin{itemize}
	\item $\alpha < \delta < \omega_1$;
	\item $X \subseteq T_\delta$ is finite 
	and has unique drop-downs to $\alpha$;
	\item $\{ f_\tau : \tau \in I \}$ is a countable collection 
	of automorphisms of $T \res (\alpha+1)$;
	\item $A \subseteq I$ is finite.
	\end{itemize}
	Then there exists a collection $\{ g_\tau : \tau \in I \}$ of automorphisms 
	of $T \res (\delta+1)$ such that:
	\begin{enumerate}
	\item $f_\tau \subseteq g_\tau$ for all $\tau \in I$;
	\item for all $\tau \in A$, $X \res \alpha$ and $X$ are $g_\tau$-consistent;
	\item if $\{ f_\tau : \tau \in A \}$ is separated on $X \res \alpha$, 
	then $\{ g_\tau : \tau \in I \}$ is separated.
	\end{enumerate}
\end{proposition}

\begin{proof}
	The proof is by induction on $\delta$, where the base case and the successor case 
	follow easily from Lemma \ref{more extending by one} and the inductive hypothesis. 
	Assume that $\delta$ is a limit ordinal and the statement holds for all $\beta$ 
	with $\alpha < \beta < \delta$. 
	Fix a surjection $h : \omega \to T_\delta \times I$ such that each member 
	of the codomain has an infinite preimage. 
	Fix an increasing sequence 
	$\langle \gamma_n : n < \omega \rangle$ of ordinals cofinal in $\delta$ 
	with $\gamma_0 = \alpha$.

	We define by induction in $\omega$-many stages the following objects 
	satisfying the listed properties:
	\begin{itemize}
	\item a subset-increasing sequence $\langle X_n : n < \omega \rangle$ 
	of finite subsets of $T_\delta$ with union equal to $T_\delta$;
	\item a subset-increasing sequence $\langle A_n : n < \omega \rangle$ of 
	finite subsets of $I$ with union equal to $I$;
	\item a non-decreasing sequence $\langle \delta_n : n < \omega \rangle$ of 
	ordinals cofinal in $\delta$;
	\item for each $n < \omega$ 
	a collection $\{ f^n_\tau : \tau \in I \}$ of automorphisms of 
	$T \res (\delta_n+1)$, where $f^n_\tau \subseteq f^m_\tau$ for all $m > n$;
	\item partial injective functions $h^n_\tau$ from $X_n$ to $X_n$ 
	for all $n < \omega$ and $\tau \in A_n$.
	\end{itemize}
	The following inductive hypotheses are maintained for each $n < \omega$:
	\begin{itemize}
	\item[(a)] $X_n$ has unique drop-downs to $\delta_n$;
	\item[(b)] if $n > 0$, then for all $\tau \in A_{n-1}$, 
	$X_{n-1} \res \delta_{n-1}$ and $X_{n-1} \res \delta_{n}$ are 
	$f^{n}_\tau$-consistent;
	\item[(c)]  if $\{ f_\tau : \tau \in A \}$ is separated on $X \res \alpha$, 
	then $\{ f^n_\tau : \tau \in I \}$ is separated;
	\item[(d)] for all $\tau \in A_n$ and $x, y \in X_n$, 
	$$
	h^n_\tau(x) = y \ \Longleftrightarrow \ f^n_\tau(x \res \delta_n) = y \res \delta_n.
	$$
	\end{itemize}

	\underline{Stage $0$:} Let $X_0 = X$, $A_0 = A$, and $\delta_0 = \alpha + 1$. 
	Apply the inductive hypothesis to find a collection $\{ f^0_\tau : \tau \in I \}$ 
	of automorphisms of $T \res (\alpha+2)$ satisfying:
	\begin{itemize}
	\item $f_\tau \subseteq f^0_\tau$ for all $\tau \in I$;
	\item for all $\tau \in A$, $X \res \alpha$ and $X \res (\alpha+1)$ 
	are $f_\tau^0$-consistent;
	\item if $\{ f_\tau : \tau \in A \}$ is separated on $X \res \alpha$, 
	then $\{ f_\tau^0 : \tau \in I \}$ is separated.
	\end{itemize}
	For all $x, y \in X$ and $\tau \in A$, define $h^0_\tau(x) = y$ iff 
	$f^0_\tau(x \res \delta_0) = y \res \delta_0$.
	
	\underline{Stage $n > 0$:} Let $0 < n < \omega$ and assume that we have 
	completed stage $n-1$. 
	In particular, we have defined $X_{n-1}$, $A_{n-1}$, $\delta_{n-1}$, 
	$\{ f^{n-1}_\tau : \tau \in I \}$, and $\{ h^{n-1}_\tau : \tau \in I \}$ 
	satisfying the required properties. 
	Let $h(n-1) = (z,\sigma)$.

	If $z \in X_{n-1}$, $\sigma \in A_{n-1}$, $z \in \dom(h^{n-1}_\sigma)$, 
	and $z \in \ran(h^{n-1}_\sigma)$, then there is nothing for us to do at stage $n$. 
	So let $X_{n} = X_{n-1}$, $A_n = A_{n-1}$, $\delta_n = \delta_{n-1}$, 
	$f^n_\tau = f^{n-1}_\tau$ and $h^n_\tau = h^{n-1}_\tau$ for all $\tau \in I$. 
	The required properties are immediate. 
	If not, then exactly one of the following cases holds.

	\underline{Case 1:} Either $z \notin X_{n-1}$ or $\sigma \notin A_{n-1}$. 
	Fix $\delta_{n} < \delta$ larger than $\delta_{n-1}$ and $\gamma_{n}$ 
	such that $X_{n-1} \cup \{ z \}$ has unique drop-downs to $\delta_{n}$. 
	Apply the inductive hypothesis to 
	find a collection $\{ f_\tau^{n} : \tau \in I \}$ of automorphisms 
	of $T \res (\delta_{n}+1)$ satisfying:
	\begin{enumerate}
	\item $f_\tau^{n-1} \subseteq f_\tau^{n}$ for all $\tau \in I$;
	\item for all $\tau \in A_{n-1}$, $X_{n-1} \res \delta_{n-1}$ 
	and $X_{n-1} \res \delta_{n}$ 
	are $f_\tau^{n}$-consistent;
	\item if $\{ f_\tau^{n-1} : \tau \in A_{n-1} \}$ is separated on 
	$X_{n-1} \res \delta_{n-1}$, then $\{ f_\tau^{n} : \tau \in I \}$ is separated.
	\end{enumerate}
	Define $X_{n} = X_{n-1} \cup \{ z \}$ and $A_n = A_{n-1} \cup \{ \sigma \}$.

	\underline{Case 2:} $z \in X_{n-1}$, $\sigma \in A_{n-1}$, and 
	$z \notin \dom(h^{n-1}_\sigma)$. 
	Define $\delta_n = \delta_{n-1}$, $A_n = A_{n-1}$, and 
	$f^n_\tau = f^{n-1}_\tau$ for all $\tau \in I$. 
	We claim that $f^n_\sigma(z \res \delta_n)$ is not in $X_{n-1} \res \delta_n$. 
	For otherwise let $x$ be the unique element of $X_{n-1}$ such that 
	$x \res \delta_n = f^n_\sigma(z \res \delta_n)$, which exists by unique drop-downs. 
	By inductive hypothesis (d) applied to $n-1$, 
	$h^{n-1}_\sigma(z) = x$, which is a contradiction. 
	Choose some $c \in T_\delta$ which is above $f^n_\sigma(z \res \delta_n)$. 
	Define $X_n = X_{n-1} \cup \{ c \}$, and note that by the claim just proved, 
	$X_n$ has unique drop-downs to $\delta_n$. 

	\underline{Case 3:} $z \in X_{n-1}$, $\sigma \in A_{n-1}$, 
	$z \in \dom(h^{n-1}_\sigma)$, and $z \notin \ran(h^{n-1}_\sigma)$. 
	Define $\delta_n = \delta_{n-1}$, $A_n = A_{n-1}$, and 
	$f^n_\tau = f^{n-1}_\tau$ for all $\tau \in I$. 
	By the same argument as in Case 2, 
	$(f^n_\sigma)^{-1}(z \res \delta_n)$ is not in $X_{n-1} \res \delta_{n-1}$. 
	Let $d$ be some element of $T_\delta$ which is above 
	$(f^n_\sigma)^{-1}(z \res \delta_n)$. 
	Define $X_n = X_{n-1} \cup \{ d \}$, 
	and note that $X_n$ has unique drop-downs to $\delta_n$.
	
	Now in any case, define for all $\tau \in A_n$ and $x, y \in X_n$, 
	$$
	h^n_\tau(x) = y \ \Longleftrightarrow \ f^n_\tau(x \res \delta_n) = y \res \delta_n.
	$$
	Note that in Case 2, $h^n_\sigma(z) = c$, so $z \in \dom(h^n_\sigma)$, and in Case 3, 
	$h^n_\sigma(d) = z$, so $z \in \ran(h^n_\sigma)$.

	Inductive hypotheses (a), (b), and (d) are clear. 
	Let us verify inductive hypothesis (c). 
	Suppose that $\{ f_\tau : \tau \in A \}$ is separated on $X \res \alpha$. 
	By (c) for $n-1$, $\{ f_\tau^{n-1} : \tau \in I \}$ is separated. 
	In Cases 2 and 3 we are done. 
	For Case 1, we clearly have that $\{ f_\tau^{n-1} : \tau \in A_{n-1} \}$ 
	is separated on $X_{n-1} \res \delta_{n-1}$. 
	So by statement (3) of Case 1, $\{ f_\tau^{n} : \tau \in I \}$ is separated.
		
	This completes the construction. 
	We claim that for all $0 < n < \omega$ and $\tau \in A_{n-1}$, 
	$h^n_\tau \cap X_{n-1}^2 = h^{n-1}_\tau$. 
	Let $x, y \in X_{n-1}$. 
	Then $h^n_\tau(x) = y$ iff $f^n_\tau(x \res \delta_n) = y \res \delta_n$ 
	(by (d) for $n$) iff $f^{n-1}_\tau(x \res \delta_{n-1}) = y \res \delta_{n-1}$ 
	(by (b) for $n$) iff $h^{n-1}_\tau(x) = y$ (by (d) for $n-1$). 

	Fix $\tau \in I$. 
	By our bookkeeping, it is clear that 
	$h_\tau = \bigcup \{ h^n_\tau : n < \omega, \ \tau \in A_n \}$ is 
	a bijection from $T_\delta$ to $T_\delta$. 
	And by (d), it is straightforward to show that 
	$g_\tau = \bigcup \{ f^n_\tau : n < \omega \} \cup h_\tau$ 
	is an automorphism of $T \res (\delta+1)$ satisfying that 
	$f_\tau \subseteq g_\tau$. 
	Let $\tau \in A = A_0$. 
	For any $x, y \in X = X_0$, by (d) we have that  
	$g_\tau(x) = y$ iff $h_\tau^0(x) = y$ iff $f_\tau(x \res \alpha) = y \res \alpha$ 
	iff $g_\tau(x \res \alpha) = y \res \alpha$. 
	Hence, $X \res \alpha$ and $X$ are $g_\tau$-consistent.

	Now assume that $\{ f_\tau : \tau \in A \}$ is separated on $X \res \alpha$. 
	To prove that $\{ g_\tau : \tau \in I \}$ is separated, 
	let $Y \subseteq T_\delta$ be finite. 
	Fix $n$ large enough so that $Y \subseteq X_n$. 
	Then by inductive hypothesis (a), $Y$ has unique drop-downs to $\delta_n$. 
	By inductive hypothesis (c), 
	$\{ f^n_\tau : \tau \in I \}$ is separated on $X_n \res \delta_n$. 
	By Lemma \ref{Persistence for Sets 1} (Persistence for Sets), it follows that 
	$\{ g_\tau : \tau \in B \}$ is separated on $Y$.
\end{proof}

\section{The Automorphism Forcing and Its Basic Properties} \label{The Automorphism Forcing and Its Basic Properties}

We now introduce the main forcing of the article.

\begin{definition}
	Let $\q$ be the forcing poset whose conditions are all automorphisms 
	$f : T \res (\alpha+1) \to T \res (\alpha+1)$, for some $\alpha < \omega_1$, 
	ordered by $g \le f$ if $f \subseteq g$. 
	If $f \in \q$ is an automorphism of $T \res (\alpha+1)$, then $\alpha$ is the 
	\emph{top level of $f$}.
\end{definition}

\begin{definition}[Automorphism Forcing] \label{poset definition}
	Let $\p$ be the forcing poset whose conditions are all functions $p$ satisfying:
	\begin{enumerate}
		\item the domain of $p$ is a countable subset of $\ka$;
		\item there exists an ordinal $\alpha < \omega_1$, which we call 
		the \emph{top level of $p$}, such that 
		for all $\tau \in \dom(p)$, $p(\tau)$ is an automorphism of $T \res (\alpha+1)$. 
	\end{enumerate} 
	Let $q \le p$ if $\dom(p) \subseteq \dom(q)$ and for all $\tau \in \dom(p)$, 
	$p(\tau) \subseteq q(\tau)$.
\end{definition}

\begin{definition}[Consistency]
	Let $\alpha < \beta < \omega_1$ and let $q \in \q$. 
	\begin{enumerate}
	\item Let $X \subseteq T_\beta$ be finite with unique drop-downs to $\alpha$. 
	We say that $X \res \alpha$ and $X$ are \emph{$q$-consistent} 
	if for all $x, y \in X$, $q(x \res \alpha) = y \res \alpha$ iff $q(x) = y$.
	\item Let $\vec a = (a_0,\ldots,a_{n-1})$ consist of distinct elements of $T_\beta$. 
	We say that $\vec a \res \alpha$ and $\vec a$ are \emph{$q$-consistent} if 
	for all $i, j < n$, $q(a_i \res \alpha) = a_j \res \alpha$ iff 
	$q(a_i) = a_j$.
	\end{enumerate}
\end{definition}

\begin{definition}[Separation]
	Let $\alpha < \omega_1$. 
	Suppose that $p \in \p$ has top level $\alpha$, 
	$A \subseteq \dom(p)$, and 
	$\vec a = (a_0,\ldots,a_{n-1})$ is an injective tuple whose members are in $T_\alpha$. 
	We say that $\{ p(\tau) : \tau \in A \}$ is \emph{separated on $\vec a$} if 
	for all $k < n$:
	\begin{enumerate}
	\item for all $\tau \in A$, $p(\tau)(a_k) \ne a_k$;
	\item there exists at most one triple $(j,m,\tau)$, where 
	$j < k$, $m \in \{ -1, 1 \}$, and $\tau \in A$, such that 
	$p(\tau)^m(a_k) = a_j$.
	\end{enumerate}
\end{definition}

\begin{definition}[Separation for Sets]
	Let $\alpha < \omega_1$. 
	Suppose that $p \in \p$ has top level $\alpha$, 
	$A \subseteq \dom(p)$, and $X \subseteq T_\alpha$ is finite. 
	We say that $\{ p(\tau) : \tau \in A \}$ is \emph{separated on $X$} 
	if there exists some injective tuple $\vec a$ which lists the elements of $X$ 
	such that $\{ p(\tau) : \tau \in A \}$ is separated on $\vec a$.
\end{definition}

The next two lemmas follow immediately by Lemmas \ref{Transitivity 1} and \ref{Persistence 1}.

\begin{lemma}[Transitivity] \label{Transitivity 2}
	Suppose that $\alpha < \beta < \gamma < \omega_1$, $r \le q$, and the top levels of $q$ 
	and $r$ are $\alpha$ and $\beta$ respectively. 
	Let $\vec b$ be an injective tuple whose members are in $T_\gamma$ 
	with unique drop-downs to $\alpha$. 
	If $\vec b \res \alpha$ and $\vec b \res \beta$ are $q$-consistent and 
	$\vec b \res \beta$ and $\vec b$ are $r$-consistent, then 
	$\vec b \res \alpha$ and $\vec b$ are $r$-consistent.	
\end{lemma}

\begin{lemma}[Persistence] \label{Persistence 2}
	Let $\alpha < \beta < \omega_1$. 
	Suppose that $p \in \p$ has top level $\alpha$, 
	$A \subseteq \dom(p)$, 
	$\vec a$ is an injective tuple whose members are in $T_\alpha$, 
	and $\{ p(\tau) : \tau \in A \}$ is separated on $\vec a$. 
	Then for any $q \le p$ with top level $\beta$ and any tuple 
	$\vec b$ above $\vec a$ whose members are in $T_\beta$, 
	$\{ q(\tau) : \tau \in B \}$ is separated on $\vec b$.
\end{lemma}

\begin{lemma}[Persistence for Sets] \label{Persistence for Sets 2}
	Let $\alpha < \beta < \omega_1$. 
	Assume that $p$ is a condition with top level $\alpha$, 
	$q \le p$, $q$ has top level $\beta$, $X \subseteq T_\beta$ is finite and 
	has unique drop-downs to $\alpha$, 
	and $A \subseteq \dom(p)$ is finite. 
	If $\{ p(\tau) : \tau \in A \}$ is separated on $X \res \alpha$, then 
	$\{ q(\tau) : \tau \in A \}$ is separated on $X$.	 
\end{lemma}

\begin{proof}
	Let $\vec a$ be an injective tuple which lists the elements of $X$ in such a 
	way that $\{ p(\tau) : \tau \in A \}$ is separated on $\vec a \res \alpha$. 
	Then $\{ q(\tau) : \tau \in A \}$ is separated on $\vec a$ by 
	Lemma \ref{Persistence 2} (Persistence).
\end{proof}

\begin{definition}
	A condition $p \in \p$ with top level $\alpha < \omega_1$ 
	is \emph{separated} if for any finite set $X \subseteq T_\alpha$, 
	$\{ p(\tau) : \tau \in \dom(p) \}$ is separated on $X$.
\end{definition}

The next three results follow immediately by 
Propositions \ref{general extending}, 
\ref{Key Property 1}, and \ref{1-Key Property 1}.

\begin{proposition}[Extension] \label{Extension 2}
	Let $\alpha < \beta < \omega_1$ and let $X \subseteq T_\beta$ be finite with 
	unique drop-downs to $\alpha$. 
	Suppose that $p \in \p$ has top level $\alpha$ and $A \subseteq \dom(p)$ is finite. 
	Then there exists some $q \le p$ with top level $\beta$ and 
	with the same domain as $p$ such that 
	for all $\tau \in A$, $X \res \alpha$ and $X$ are $q(\tau)$-consistent. 
	Moreover, if $\{ p(\tau) : \tau \in A \}$ is separated on $X \res \alpha$, then 
	we can find such a condition $q$ which is separated.
\end{proposition}

\begin{proposition}[Key Property] \label{Key Property 2}
	Let $\alpha < \beta < \omega_1$. 
	Suppose that $a_0,\ldots,a_{n-1}$ are distinct elements 
	of $T_\alpha$, $p \in \p$ has top level $\alpha$, 
	$A \subseteq \dom(p)$ is finite, and 
	$\{ p(\tau) : \tau \in A \}$ is separated on $(a_0,\ldots,a_{n-1})$. 
	Then for any $q \le p$ with top level $\beta$ and any finite set $t \subseteq T_\beta$, 
	there exist $b_0,\ldots,b_{n-1}$ in $T_\beta \setminus t$ 
	such that $a_i <_T b_i$ for all $i < n$, and 
	for all $\tau \in A$, 
	$(a_0,\ldots,a_{n-1})$ and $(b_0,\ldots,b_{n-1})$ are $q(\tau)$-consistent.
\end{proposition}

\begin{proposition}[$1$-Key Property] \label{1-Key Property 2}
	Let $\alpha < \beta < \omega_1$. 
	Suppose that $a_0,\ldots,a_{n-1}$ are distinct elements 
	of $T_\alpha$, $p \in \p$ has top level $\alpha$, $A \subseteq \dom(p)$ is finite, 
	and $\{ p(\tau) : \tau \in A \}$ is separated on $(a_0,\ldots,a_{n-1})$. 
	Fix $\bar{n} < n$ and $b \in T_\beta$ with $a_{\bar{n}} <_T b$. 
	Then for any $q \le p$ with top level $\beta$, 
	there exist $b_0,\ldots,b_{n-1}$ in $T_\beta$ 
	such that $a_i <_T b_i$ for all $i < n$, $b_{\bar{n}} = b$, 
	and for all $\tau \in A$, 
	$(a_0,\ldots,a_{n-1})$ and $(b_0,\ldots,b_{n-1})$ are $q(\tau)$-consistent.
\end{proposition}

\begin{lemma} \label{extending domain}
	For any $\tau < \ka$ and $\rho < \omega_1$, the set of conditions $q \in \p$ 
	such that $\tau \in \dom(q)$ and the top level of $q$ is at least $\rho$ 
	is dense open.
\end{lemma}

\begin{proof}
	Given any condition $p \in \p$, we can easily add $\tau$ to the domain of $p$, 
	for example, by attaching to it the identity automorphism. 
	The second part of the statement follows from Proposition \ref{Extension 2} (Extension).
\end{proof}

\begin{lemma} \label{another dense set} 
	For any $\tau_0 < \tau_1 < \ka$, the set of conditions $q \in \p$ 
	with some top level $\alpha$ 
	satisfying that $\tau_0, \tau_1 \in \dom(q)$ and for all $x \in T_\alpha$, 
	$q(\tau_0)(x) \ne q(\tau_1)(x)$, is dense open.
\end{lemma}

\begin{proof}
	The set of such conditions is dense by Lemmas \ref{extending by one} 
	and \ref{extending domain}. 
	It is easy to check that 
	if $\alpha < \beta$, $f$ and $g$ are automorphisms of $T \res (\beta+1)$, 
	$x \in T_\alpha$, and $f(x) \ne g(x)$, then for any $y \in T_\beta$ above 
	$x$, $f(y) \ne f(y)$. 
	This fact easily implies that the set of such conditions in open.
\end{proof}

The next lemma follows immediately from Lemma \ref{more extending by one}.

\begin{lemma}[Separated Conditions are Dense] \label{Separated Conditions are Dense}
	The set of separated conditions is dense in $\p$. 
	In fact, suppose that $p$ is a condition with top level $\gamma < \omega_1$, 
	$X \subseteq T_{\gamma+1}$ is finite with 
	unique drop-downs to $\gamma$, 
	$A \subseteq \dom(p)$ is finite, 
	and $\{ p(\tau) : \tau \in A \}$ is separated on $X \res \gamma$. 
	Then there exists $q \le p$ with top level $\gamma+1$ such that 
	$q$ is separated on $T_{\gamma+1}$ and for all 
	$\tau \in A$, $X \res \gamma$ and $X$ are $q(\tau)$-consistent.
\end{lemma}

\begin{lemma} \label{separated iff}
	A condition $p \in \p$ with top level $\alpha$ 
	is separated if and only if for any finite set $A \subseteq \dom(p)$ 
	and finite set $X \subseteq T_\alpha$, 
	$\{ p(\tau) : \tau \in A \}$ is separated on $X$.
\end{lemma}

\begin{proof}
	The forward direction of the equivalence is immediate. 
	For the converse, assume the second statement. 
	It easily follows that for any $\tau \in A$, $p(\tau)$ has no fixed-points in $T_\alpha$. 
	Let $X \subseteq T_\alpha$ be finite with size $n < \omega$ 
	and suppose for a contradiction that 
	$\{ p(\tau) : \tau \in \dom(p) \}$ is not separated on $X$. 	
	For any injective tuple $\vec a = (a_0,\ldots,a_{n-1})$ 
	which lists the elements of $X$, 
	we can find some $k < n$ and distinct triples 
	$(j_0,m_0,\tau_0)$ and $(j_1,m_1,\tau_1)$ such that for $i < 2$, 
	$j_i < k$, $m_i \in \{ -1, 1 \}$, $\tau_i \in \dom(p)$, and 
	$p(\tau_i)^{m_i}(a_k) = a_{j_i}$. 
	There are only finitely many such enumerations, so we can find a finite 
	set $A \subseteq \dom(p)$ such that for any such enumeration, 
	the fixed triples $(j_0,m_0,\tau_0)$ and $(j_1,m_1,\tau_1)$ described above 
	satisfy that $\tau_0$ and $\tau_1$ are in $A$. 
	Now it is easy to check that $\{ p(\tau) : \tau \in A \}$ is not separated on $X$, 
	which is a contradiction.
\end{proof}

\begin{lemma}[Generalized Key Property] \label{Generalized Key Property}
	Let $\alpha < \xi < \omega_1$. 
	Suppose that 
	$p \in \p$ has top level $\alpha$, $\vec b$ is a finite 
	tuple with height $\alpha$, $A \subseteq \dom(p)$ is finite, and 
	$\{ p(\tau) : \tau \in A \}$ is separated on $\vec b$. 
	Assume that $r_0, \ldots, r_{n-1} \le p$ are conditions with top level $\xi$, 
	$\{ h_\tau : \tau \in B \}$ is a finite family of automorphisms of $T \res (\xi+1)$, 
	and $t \subseteq T_\xi$ is finite. 
	Then there exist tuples 
	$\vec a^0, \ldots, \vec a^{n-1}$ above $\vec b$ with height $\xi$ 
	such that:
	\begin{enumerate}
	\item for all $\tau \in A$ and $j < n$, $\vec b$ and $\vec a^j$ are 
	$r_j(\tau)$-consistent;
	\item $\vec a^0 \res (\alpha+1), \ldots, \vec a^{n-1} \res (\alpha+1)$ 
	and $t \res (\alpha+1)$ are pairwise disjoint;
	\item for all $x \in t$, $\tau \in B$, and $m \in \{ -1, 1 \}$, 
	$h_\tau^m(x)$ is not in any of the tuples $\vec a^0,\ldots,\vec a^{n-1}$;
	\item for all $\tau \in B$, $m \in \{ -1 , 1 \}$, 
	and distinct $i, j < n$, if $x$ is in the 
	tuple $\vec a^i$ then $h_\tau^m(x)$ is not in the tuple $\vec a^j$.
	\end{enumerate} 
\end{lemma}

\begin{proof}
	The proof is by induction on $n > 0$. 
	Let $n > 0$ be given, and assume that the statement holds for $n-1$ 
	(in the case that $n > 1$). 
	Assume that $r_0,\ldots,r_{n-1} \le p$ are conditions with top level $\xi$, 
	$\{ h_\tau : \tau \in B \}$ is a finite family of automorphisms of $T \res (\xi+1)$, 
	and $t \subseteq T_\xi$ is finite. 
	Applying the inductive hypothesis in the case that $n > 1$, 
	fix $\vec a^0, \ldots, \vec a^{n-2}$ above $\vec b$ with height $\xi$ 
	satisfying (1)-(4).

	Define $r_{n-1}^*$ to be the condition 
	with the same domain as $r_{n-1}$ so that for all 
	$\tau \in \dom(r_{n-1})$, $r_{n-1}^*(\tau) = r_{n-1}(\tau) \res (\alpha+1)$. 
	Then $r_{n-1} \le r_{n-1}^* \le p$ and $r_{n-1}^*$ has top level $\alpha+1$. 
	Let $Z$ be the finite set of elements of $T_{\alpha+1}$ which are in 
	one of $\vec a^0 \res (\alpha+1), \ldots, \vec a^{n-1} \res (\alpha+1)$ or 
	$t \res (\alpha+1)$, 
	or else of the form $h_\tau^m(x) \res (\alpha+1)$, where $\tau \in B$, 
	$m \in \{ -1, 1 \}$, and $x$ is in one of 
	$\vec a^0,\ldots,\vec a^{n-1}$ or in $t$. 	
	By Proposition \ref{Key Property 2} (Key Property), find a tuple 
	$\vec b^{n-1} > \vec b$ with height $\alpha+1$ 
	such that for all $\tau \in A$, $\vec b$ and $\vec b^{n-1}$ are 
	$r_{n-1}^*(\tau)$-consistent, and $\vec b^{n-1}$ is disjoint from $Z$. 
	By Lemma \ref{Persistence 2} (Persistence), 
	$\{ r_{n-1}^*(\tau) : \tau \in A \}$ is separated on $\vec b^{n-1}$. 
	Apply Proposition \ref{Key Property 2} (Key Property) again to find 
	a tuple $\vec a^n$ with height $\xi$ above $\vec b^{n-1}$ 
	such that for all $\tau \in A$, 
	$\vec b^{n-1}$ and $\vec a^{n-1}$ are $r_{n-1}(\tau)$-consistent.
\end{proof}

\section{More Properties of the Automorphism Forcing} \label{More Properties of the Automorphism Forcing}

In this section, we prove two propositions which we use later 
to derive some of the strongest properties of the automorphism forcing. 
In these results, the assumption of the freeness of $T$ together with our analysis of 
the notions of consistency, separation, and the Key Property are of critical importance.

\begin{lemma} \label{pre consistent extensions into dense sets}
	Suppose that $D$ is a dense open subset of $\p$ 
	and $p \in \p$ has top level $\xi$. 
	Let $A \subseteq \dom(p)$ be finite. 
	Consider a tuple $\vec a$ of height $\xi$ and assume that 
	$\{ p(\tau) : \tau \in A \}$ is separated on $\vec a$. 
	Let $\mathcal X$ be the set of all $\vec c$ in the derived tree $T_{\vec a}$ for which 
	there exists some $q \le_\p p$ in $D$ whose top level equals the height of 
	$\vec c$ such that for all $\tau \in A$, 
	$\vec a$ and $\vec c$ are $q(\tau)$-consistent. 
	Then $\mathcal X$ is a dense open subset of $T_{\vec a}$.
\end{lemma}

\begin{proof}
	For openness, consider a tuple $\vec b \in \mathcal X$ as witnessed by 
	a condition $q \le_\p p$. 
	Let $\vec c$ be a tuple such that 
	$\vec b < \vec c$ and $\vec c$ has height $\zeta$. 
	By Proposition \ref{Extension 2} (Extension), 
	find $r \le_\p q$ with top level $\zeta$ such that 
	for all $\tau \in A$, $\vec b$ and $\vec c$ are $r(\tau)$-consistent. 
	Since $D$ is open, $r \in D$. 
	By Lemma \ref{Transitivity 2} (Transitivity), 
	for all $\tau \in A$, $\vec a$ and $\vec c$ are $r(\tau)$-consistent.
	So $r$ witnesses that $\vec c \in \mathcal X$.
	
	To show that $\mathcal X$ is dense, consider $\vec b \in T_{\vec a}$ 
	with height $\delta > \xi$. 
	By Proposition \ref{Extension 2} (Extension), 
	fix $q \le_\p p$ with top level $\delta$ such that 
	for all $\tau \in A$, 
	$\vec a$ and $\vec b$ are $q(\tau)$-consistent. 
	By Lemma \ref{Persistence 2} (Persistence), 
	$\{ q(\tau) : \tau \in A \}$ is separated on $\vec b$. 
	Let $E$ be the set of condition $s \in \p$ such that $s$ has top level greater than $\delta$. 
	By Proposition \ref{Extension 2} (Extension), $E$ is dense, and clearly $E$ is open. 
	So $D \cap E$ is dense open. 	
	Fix $r \le_\p q$ in $D \cap E$ and let $\rho$ be the top level of $r$. 
	Then $\rho > \delta$. 
	Since $r \le q$ and $\{ q(\tau) : \tau \in A \}$ is separated on $\vec b$, 
	by Proposition \ref{Key Property 2} (Key Property) 
	we can find some $\vec c$ with height $\rho$ 
	such that $\vec b < \vec c$ 
	and for all $\tau \in A$, $\vec b$ and $\vec c$ 
	are $r(\tau)$-consistent. 
	By Lemma \ref{Transitivity 2} (Transitivity), 
	for all $\tau \in A$, $\vec a$ and $\vec c$ are $r(\tau)$-consistent. 
	So $r$ is a witness that $\vec c \in \mathcal X$.
	\end{proof}

\begin{proposition}[Consistent Extensions Into Dense Sets] \label{Consistent Extensions Into Dense Sets}
	Suppose that $T$ is a free Suslin tree. 
	Let $\lambda$ be a large enough regular cardinal and 
	let $N$ be a countable 
	elementary substructure of $H(\lambda)$ containing 
	as members $T$, $\ka$, and $\p$. 
	Let $\delta = N \cap \omega_1$. 
	Assume that $D \in N$ is a dense open subset of $\p$, 
	$p \in N \cap \p$ has top level $\beta$, and $A \subseteq \dom(p)$ is finite. 
	Let $\vec a$ have height $\delta$ with unique drop-downs to $\beta$ 
	such that $\{ p(\tau) : \tau \in A \}$ is separated on $\vec a \res \beta$. 
	Then there exists some $q \le_\p p$ in $D \cap N$ whose top level 
	is some ordinal $\gamma < \delta$ 
	such that for all $\tau \in A$, $\vec a \res \beta$ 
	and $\vec a \res \gamma$ are $q(\tau)$-consistent.
\end{proposition}

\begin{proof}
	Let $\mathcal X$ be the set of all tuples 
	$\vec b$ in the derived tree $T_{\vec a \res \beta}$ satisfying 
	that for some $q \le_\p p$ in $D$, for all $\tau \in A$, $\vec a \res \beta$ 
	and $\vec b$ are $q(\tau)$-consistent. 
	Note that $\mathcal X \in N$ by elementarity, 
	and $\mathcal X$ is dense open in $T_{\vec a \res \beta}$ 
	by Lemma \ref{pre consistent extensions into dense sets}. 
	Since $\mathcal X$ is dense open and $T$ is free, 
    fix $\gamma > \beta$ 
	such that any member of $T_{\vec a \res \beta}$ 
	whose elements have height at least $\gamma$ is in $\mathcal X$. 
	By elementarity, we can choose $\gamma \in N \cap \omega_1 = \delta$. 
	Then $\vec a \res \gamma \in \mathcal X$, 
	which by elementarity can be witnessed by some $q \in D \cap N$. 
	Clearly, $q$ is as required.	
\end{proof}

\begin{proposition}[Consistent Extensions for Sealing] \label{Consistent Extensions for Sealing}
	Suppose that $T$ is a free Suslin tree. 
	Let $\lambda$ be a large enough regular cardinal and let $N$ be a countable 
	elementary substructure of $H(\lambda)$ containing $T$, $\ka$, and $\p$. 
	Let $\delta = N \cap \omega_1$. 
	Assume that $\dot E \in N$ is a $\p$-name for a dense open subset of $T$.

	Let $p \in N \cap \p$ have top level $\gamma$. 
	Suppose that $\vec a = (a_0,\ldots,a_{l-1})$ has height $\delta$, $\vec a$ 
	has unique drop-downs to $\gamma$, and $j < l$. 
	Assume that $A \subseteq \dom(p)$ is finite and 
	$\{ p(\tau) : \tau \in A \}$ is separated on $\vec a \res \gamma$. 
	Then there exists some $r \le_\p p$ in $N$ with some top level $\xi$  
	such that $r \Vdash_{\p} a_{j} \in \dot E$ 
	and for all $\tau \in A$, $\vec a \res \gamma$ and $\vec a \res \xi$ 
	are $r(\tau)$-consistent.
\end{proposition}

\begin{proof}
	Let $\mathcal{X}$ be the set of all 
	$\vec d = (d_0,\ldots,d_{l-1})$ in the derived tree 
	$T_{\vec a \res \gamma}$ satisfying 
	that for some $r \le_\p p$ with top level equal to the height of $\vec d$, 
	$r \Vdash_\p d_{j} \in \dot E$ and for all $\tau \in A$, 
	$\vec a \res \gamma$ and $\vec d$ are $r(\tau)$-consistent. 
	Note that $\mathcal{X} \in N$ by elementarity. 
	We claim that $\mathcal{X}$ is dense open in $T_{\vec a \res \gamma}$. 
	To show that $\mathcal{X}$ is dense, 
	consider $\vec b = (b_0,\ldots,b_{l-1})$ in $T_{\vec a \res \gamma}$ 
	whose elements have height $\zeta > \gamma$. 
	By Proposition \ref{Extension 2} (Extension), 
	find $q \le_\p p$ with top level $\zeta$ 
	such that for all $\tau \in A$, 
	$\vec a \res \gamma$ and $\vec b$ are $q(\tau)$-consistent. 
	As $\{ p(\tau) : \tau \in A \}$ is separated on $\vec a \res \gamma$, 
	Lemma \ref{Persistence 2} (Persistence) implies that 
	$\{ q(\tau) : \tau \in A \}$ is separated on $\vec b$.

	Since $\dot E$ is forced to be dense open in $T$, fix 
	$r \le_\p q$ with some top level $\xi$ and fix $c$ 
	above $b_j$ with some height $\rho$ 
	such that $r \Vdash c \in \dot E$. 
	By extending further if necessary using Lemma \ref{Extension 2} (Extension) 
	and using the fact that $\dot E$ is forced 
	to be open, we may assume without loss of generality that $\xi = \rho > \zeta$. 
	By Proposition \ref{1-Key Property 2} ($1$-Key Property), 
	we can find $\vec d = (d_0,\ldots,d_{l-1})$ above $\vec b$ 
	such that $d_{j} = c$ and for all $\tau \in A$, 
	$\vec b$ and $\vec d$ are $r(\tau)$-consistent. 
	By Lemma \ref{Transitivity 2} (Transitivity), 
	for all $\tau \in A$, $\vec a \res \gamma$ and $\vec d$ are $r(\tau)$-consistent. 
	So $\vec b < \vec d \in \mathcal{X}$.

	To show that $\mathcal{X}$ is open, 
	suppose that $\vec d = (d_0,\ldots,d_{l-1})$ is in 
	$\mathcal{X}$ as witnessed by $r$, 
	and let $\vec e = (e_0,\ldots,e_{l-1})$ be above $\vec d$ of height $\xi$. 
	We show that $\vec e \in \mathcal X$. 
	By Proposition \ref{Extension 2} (Extension), find $s \le_\p r$ with top level $\xi$ such that 
	for all $\tau \in A$, $\vec d$ and $\vec e$ are $s(\tau)$-consistent. 
	Since $\dot E$ is forced to be open and 
	$r \Vdash_\p d_j \in \dot E$, it follows that  
	$s \Vdash_\p e_j \in \dot E$. 
	By Lemma \ref{Transitivity 2} (Transitivity), 
	for all $\tau \in A$, $\vec a \res \gamma$ and $\vec e$ are $s(\tau)$-consistent. 
	So $\vec e \in \mathcal{X}$.

	Since $\mathcal{X}$ is dense open and $T$ is free, 
	we can fix some $\xi > \gamma$ such that 
	any member of $T_{\vec a \res \gamma}$ whose elements have 
	height at least $\xi$ is in $\mathcal{X}$. 
	By elementarity, we can choose $\xi$ in $N \cap \omega_1 = \delta$. 
	So $\vec a \res \xi \in \mathcal{X}$, 
	which by elementarity can be witnessed by some $r \in N$. 
	So $r \Vdash_\p a_j \res \xi \in \dot E$ and 
	for all $\tau \in A$, $\vec a \res \gamma$ and $\vec a \res \xi$ 
	are $r(\tau)$-consistent. 
	Since $\dot E$ is forced to be open, 
	$r \Vdash_{\p} a_j \in \dot E$. 
\end{proof}

\section{Preserving Cardinals and Suslinness} \label{Preserving Cardinals and Suslinness}

We are now ready to prove that if $T$ is a free Suslin tree, 
then the forcing poset $\p$ is totally proper, 
preserves the fact that $T$ is Suslin, 
and adds a sequence of length $\ka$ of almost disjoint automorphisms of $T$. 
Also, assuming $\textsf{CH}$, $\p$ is $\omega_2$-c.c. 
In particular, if $T$ is a free Suslin tree, $\textsf{CH}$ holds, and $\ka \ge \omega_2$, 
then $\p$ forces that $T$ is an almost 
Kurepa Suslin tree.

Many of the proofs in the rest of the article 
involve constructing total master conditions over countable elementary substructures. 
The next two lemmas provide tools for such constructions.

\begin{lemma}[Constructing Total Master Conditions] \label{Constructing Total Master Conditions}
	Let $\lambda$ be a large enough regular cardinal and assume that $N$ is a countable 
	elementary substructure of $H(\lambda)$ which contains as elements 
	$T$, $\ka$, $\q$, and $\p$. 
	Let $\delta = N \cap \omega_1$.

	Assume the following:
	\begin{enumerate}
	\item $\langle \delta_n : n < \omega \rangle$ is a non-decreasing sequence 
	of ordinals cofinal in $\delta$;
	\item $\langle p_n : n < \omega \rangle$ is a decreasing sequence 
	of conditions in $N \cap \p$, where each $p_n$ has top level $\delta_n$;
	\item $\langle A_n : n < \omega \rangle$ is a subset-increasing sequence of 
	finite subsets of $N \cap \ka$ with union equal to $N \cap \ka$;
	\item $\langle X_n : n < \omega \rangle$ is a subset-increasing sequence 
	of finite subsets of $T_\delta$ with union equal to $T_\delta$, where 
	each $X_n$ has unique drop-downs to $\delta_n$;
	\item $\{ h_{n,\tau} : n < \omega, \ \tau \in A_n \}$ is a family 
	of functions, where each $h_{n,\tau}$ is an injective partial function 
	from $X_n$ to $X_n$;
	\item for all $z \in T_\delta$ and $\tau \in N \cap \ka$, 
	there exists some $n < \omega$ 
	such that $z$ is in the domain and in the range of $h_{n,\tau}$;
	\item for all $n < \omega$ and $\tau \in A_n$:
	\begin{enumerate}
	\item $\tau \in \dom(p_n)$; 
	\item $X_n \res \delta_n$ 
	and $X_n \res \delta_{n+1}$ are $p_{n+1}(\tau)$-consistent;
	\item for all $x, y \in X_n$, 
	$h_{n,\tau}(x) = y$ iff $p_n(\tau)(x \res \delta_n) = y \res \delta_n$.
	\end{enumerate}
	\end{enumerate}

	Let $q$ be the function 
	with domain $N \cap \ka$ such that for all $\tau \in N \cap \ka$, 
	$$
	q(\tau) = \bigcup \{ p_n(\tau) : n < \omega, \ \tau \in A_n \} \cup 
	\bigcup \{ h_{n,\tau} : n < \omega, \ \tau \in A_n \}.
	$$
	Then $q \in \p$, $q$ has top level $\delta$, 
	and $q \le p_n$ for all $n < \omega$. 
	For all $n < \omega$ and for all $\tau \in A_n$, 
	$X_n \res \delta_n$ and $X_n$ are $q(\tau)$-consistent. 
	Moreover, if for all dense open sets $D \in N$ there exists some $n$ 
	such that $p_n \in D$, then $q$ is a total master condition over $N$.
\end{lemma}

\begin{proof}
	We claim that for all $n < \omega$ and $\tau \in A_n$, 
	$h_{n+1,\tau} \cap X_n^2 = h_{n,\tau}$. 
	Let $x, y \in X_n$. 
	By (2) and (7), 
	$h_{n,\tau}(x) = y$ iff $p_n(\tau)(x \res \delta_n) = y \res \delta_n$ iff 
	$p_{n+1}(\tau)(x \res \delta_{n+1}) = y \res \delta_{n+1}$ iff 
	$h_{n+1,\tau}(x) = y$. 
	It follows from this fact together with (4), (5), and (6) that for 
	each $\tau \in N \cap \ka$, 
	$\bigcup \{ h_{n,\tau} : n < \omega, \ \tau \in A_n \}$ 
	is a bijection from $T_\delta$ onto $T_\delta$.

	Let $\tau \in N \cap \ka$ and we show that 
	$q(\tau)$ is strictly increasing. 
	It suffices to show that if $x \in T_\delta$ and $\gamma < \delta$, 
	then $q(\tau)(x)$ is above $q(\tau)(x \res \gamma)$. 
	Fix $n < \omega$ large enough so that $x \in X_n$, $\tau \in A_n$, 
	and $\delta_n > \gamma$. 
	By (7), $q(\tau)(x) = h_{n,\tau}(x)$ is above $p_n(\tau)(x \res \delta_n)$. 
	Since $p_n$ is strictly increasing, $p_n(\tau)(x \res \delta_n)$ is above 
	$p_n(\tau)(x \res \gamma) = q(\tau)(x \res \gamma)$. 

	So $q \in \p$ and $q \le p_n$ for all $n < \omega$. 
	The other statements are easy to verify.
\end{proof}

\begin{lemma}[Augmentation] \label{augmentation}
	Let $\lambda$ be a large enough regular cardinal and assume that $N$ is a countable 
	elementary substructure of $H(\lambda)$ which contains as elements $T$, $\ka$, 
	$\q$, and $\p$. 
	Let $\delta = N \cap \omega_1$. 
	Suppose the following:
	\begin{itemize}
	\item $p \in N \cap \p$ has top level $\gamma$;
	\item $B \subseteq \dom(p)$ is finite;
	\item $z \in T_\delta$ and $\sigma \in \dom(p)$;
	\item $X \subseteq T_\delta$ is finite and 
	$X \cup \{ z \}$ has unique drop-downs to $\gamma$;
	\item $\{ p(\tau) : \tau \in B \}$ is separated on $X \res \gamma$.
	\end{itemize}
	Then there exists some $q \le p$ in $N \cap \p$ with top level $\gamma+1$ 
	and there exists a finite set $Y \subseteq T_\delta$ satisfying:
	\begin{enumerate}
	\item for all $\tau \in B$, $X \res \gamma$ 
	and $X \res (\gamma+1)$ are $q(\tau)$-consistent;
	\item $X \cup \{ z \} \subseteq Y$ and $Y$ has unique drop-downs to $\gamma+1$;
	\item $\{ q(\tau) : \tau \in B \cup \{ \sigma \} \}$ 
	is separated on $Y \res (\gamma+1)$;
	\item let $h_\sigma^+$ be the partial injective function 
	from $Y$ to $Y$ defined by letting, 
	for all $x, y \in Y$, $h_\sigma^+(x) = y$ 
	iff $q(\sigma)(x \res (\gamma+1)) = y \res (\gamma+1)$; 
	then $z$ is in the domain and range of $h_\sigma^+$.
	\end{enumerate}
\end{lemma}

\begin{proof}
	Apply Lemma \ref{Separated Conditions are Dense} (Separated Conditions are Dense) 
	and elementarity 
	to find a separated condition $q \le p$ in $N \cap \p$ 
	with top level $\gamma+1$ satisfying that for all $\tau \in B$, 
	$X \res \gamma$ and $X \res (\gamma+1)$ are $q(\tau)$-consistent. 
	Define 
	$$
	W = \{ \ z \res (\gamma+1), \ 
	q(\sigma)(z \res (\gamma+1)), \ q(\sigma)^{-1}(z \res (\gamma+1)) \ \} \ 
	\setminus \ (X \res (\gamma+1)).
	$$
	Note that by unique drop-downs of $X \cup \{ z \}$, 
	if $z \notin X$ then $z \res (\gamma+1)$ 
	is not in $X \res (\gamma+1)$. 
	So if $z$ is not in $X$, then $z \res (\gamma+1)$ is in $W$. 
	Choose a set $Y \subseteq T_\delta$ consisting of the elements of 
	$X$ together with exactly one element of $T_\delta$ above each member of $W$, 
	and such that if $z \notin X$ then the element of $Y$ above $z \res (\gamma+1)$ is $z$. 
	Note that $Y$ has unique drop-downs to $\gamma+1$ and 
	$Y \res (\gamma+1) = (X \res (\gamma+1)) \cup W$. 

	Conclusions (1), (2), and (3) are clear. 
	(4): Note that $q(\sigma)(z \res (\gamma+1))$ is either in $X \res (\gamma+1)$ or in $W$. 
	As $Y \res (\gamma+1) = (X \res (\gamma+1)) \cup W$, in either case we can 
	fix $c \in Y$ such that $c \res (\gamma+1) = q(\sigma)(z \res (\gamma+1))$. 
	Then by the definition of $h_\sigma^+$, 
	$q(\sigma)(z \res (\gamma+1)) = c \res (\gamma+1)$ implies that 
	$h_\sigma^+(z) = c$, so $z \in \dom(h_\sigma^+)$. 
	The proof that $z$ is in the range of $h_\sigma^+$ is similar.
\end{proof}

\begin{thm} \label{proper and preserve suslin}
	Suppose that $T$ is a free Suslin tree. 
	Then the forcing poset $\p$ is totally proper and preserves the 
	fact that $T$ is Suslin.
\end{thm}

\begin{proof}
	Let $\lambda$ be a large enough regular cardinal. 
	Let $N$ be a countable elementary substructure of $H(\lambda)$ 
	containing as members $T$, $\ka$, $\q$, and $\p$. 
	Let $\delta = N \cap \omega_1$. 
	Suppose that $p \in N \cap \p$ and $\dot E \in N$ is a $\p$-name 
	for a dense open subset of $T$. 
	We prove that there exists a total master condition $q \le p$ over $N$ such that 
	$q \Vdash T_\delta \subseteq \dot E$. 
	The theorem easily follows.

	So let $N$, $\delta$, $p$, and $\dot E$ be given. 
	Let $\alpha$ be the top level of $p$. 
	Fix an increasing sequence $\langle \gamma_n : n < \omega \rangle$ 
	of ordinals cofinal in $\delta$ with $\gamma_0 = \alpha$ 
	and an enumeration $\langle D_n : n < \omega \rangle$ 
	of all of the dense open subsets of $\p$ which lie in $N$. 
	Fix a surjection $g : \omega \to 3 \times T_{\delta} \times (N \cap \ka)$ 
	such that every element of the codomain has an infinite preimage.

	We define the following objects by induction in $\omega$-many stages:
	\begin{itemize}
	\item a subset-increasing sequence $\langle X_n : n < \omega \rangle$ 
	of finite subsets of $T_\delta$ with union equal to $T_\delta$;
	\item a subset-increasing sequence $\langle A_n : n < \omega \rangle$ 
	of finite subsets of $N \cap \ka$ with union equal to $N \cap \ka$;
	\item a non-decreasing sequence $\langle \delta_n : n < \omega \rangle$ 
	of ordinals cofinal in $\delta$;
	\item a decreasing sequence $\langle p_n : n < \omega \rangle$ 
	of conditions in $N \cap \p$ such that $p_0 = p$ and for all $n$, 
	$\delta_n$ is the top level of $p_n$;
	\item for each $n < \omega$ and $\sigma \in N \cap \ka$, an injective 
	partial function $h_{n,\sigma}$ from $X_n$ to $X_n$.
	\end{itemize}
	In addition to the properties listed above, 
	we maintain the following inductive hypotheses for all $n < \omega$:
	\begin{enumerate}
	\item $X_n$ has unique drop-downs to $\delta_n$;
	\item $A_n \subseteq \dom(p_n)$ and 
	$\{ p_n(\tau) : \tau \in A_n \}$ is separated on $X_n \res \delta_n$;
	\item for all $\tau \in A_n$, $X_n \res \delta_n$ and $X_n \res \delta_{n+1}$ 
	are $p_{n+1}(\tau)$-consistent;
	\item for all $\tau \in A_n$ and $x, y \in X_n$, 
	$$
	h_{n,\tau}(x) = y \ \Longleftrightarrow \ 
	p_n(\tau)(x \res \delta_n) = y \res \delta_n.
	$$
	\end{enumerate}
	
	\underline{Stage $0$:} Let $X_0 = \emptyset$, $A_0 = \emptyset$, 
	$\delta_0 = \alpha$, and $p_0 = p$.

	\underline{Stage $n+1$:} Let $n < \omega$ and assume that we have completed stage $n$. 
	In particular, we have defined $X_n$, $A_n$, $\delta_n$, $p_n$, and 
	$h_{n,\tau}$ for all $\tau \in A_n$ satisfying the required properties. 
	Let $g(n) = (n_0,z,\sigma)$.

	\underline{Case $a$:} $n_0 = 0$. 
	Apply Proposition 
	\ref{Consistent Extensions Into Dense Sets} (Consistent Extensions Into Dense Sets) and 
	Lemma \ref{extending domain} to find a condition $p_{n+1} \le p_n$ in 
	$N \cap \bigcap_{k < n} D_k$ with $\sigma \in \dom(p_{n+1})$ 
	and top level some ordinal $\delta_{n+1}$ 
	greater than $\gamma_{n+1}$ such that for all $\tau \in A_n$, 
	$X_n \res \delta_n$ and $X_n \res \delta_{n+1}$ 
	are $p_{n+1}(\tau)$-consistent. 
	Define $X_{n+1} = X_n$ and $A_{n+1} = A_n$. 
	Define for all $\tau \in A_{n+1}$ a partial function 
	$h_{n+1,\tau}$ as described in inductive hypothesis (4). 
	It is routine to check that the required properties are satisfied.

	\underline{Case $b$:} $n_0 = 1$. 
	If $z \notin X_n$, then let $X_{n+1} = X_n$, 
	$A_{n+1} = A_n$, $\delta_{n+1} = \delta_n$, $p_{n+1} = p_n$, 
	and $h_{n+1,\tau} = h_{n,\tau}$ for all $\tau \in A_n$. 
	Suppose that $z \in X_n$. 
	Fix an injective tuple $\vec a = (a_0,\ldots,a_{l-1})$ 
	which lists the elements of $X_n$, and 
	fix $j < l$ such that $z = a_j$. 
	Applying 
	Proposition \ref{Consistent Extensions for Sealing} (Consistent Extensions for Sealing), 
	fix $p_{n+1} \le p$ in $N \cap \p$ with some top level $\delta_{n+1}$ 
	such that $p_{n+1} \Vdash_{\p} a_j \in \dot E$ and 
	for all $\tau \in A_n$, $\vec a \res \delta_n$ and $\vec a \res \delta_{n+1}$ 
	are $p_{n+1}(\tau)$-consistent. 
	Then $p_{n+1} \Vdash_\p z \in \dot E$ and 
	for all $\tau \in A_n$, $X_n \res \delta_n$ and $X_n \res \delta_{n+1}$ 
	are $p_{n+1}(\tau)$-consistent. 
	Let $X_{n+1} = X_n$ and $A_{n+1} = A_n$. 
	Define for all $\tau \in A_{n+1}$ a partial function 
	$h_{n+1,\tau}$ as described in inductive hypothesis (4). 
	The required properties are clearly satisfied.

	\underline{Case $c$:} $n_0 = 2$. 
	If $\sigma \notin \dom(p_n)$, then let $X_{n+1} = X_n$, 
	$A_{n+1} = A_n$, $\delta_{n+1} = \delta_n$, $p_{n+1} = p_n$, 
	and $h_{n+1,\tau} = h_{n,\tau}$ for all $\tau \in A_n$. 
	Now suppose that $\sigma \in \dom(p_n)$. 
	Fix $\gamma < \delta$ larger than $\delta_n$ and $\gamma_{n+1}$ 
	such that $X_n \cup \{ z \}$ has unique drop-downs to $\gamma$. 
	Apply Proposition \ref{Extension 2} (Extension) to find $p_n' \le p_n$ with top level $\gamma$ 
	such that for all $\tau \in A_n$, $X_n \res \delta_n$ and $X_n \res \gamma$ 
	are $p_n'(\tau)$-consistent.

	Apply Lemma \ref{augmentation} (Augmentation) to find $p_{n+1} \le p_n'$ with top level 
	$\gamma+1$ and a finite set $Y \subseteq T_\delta$ satisfying:
	\begin{enumerate}
	\item for all $\tau \in A_n$, $X_n \res \gamma$ 
	and $X_n \res (\gamma+1)$ are $p_{n+1}(\tau)$-consistent;
	\item $X_n \cup \{ z \} \subseteq Y$ 
	and $Y$ has unique drop-downs to $\gamma+1$;
	\item $\{ p_{n+1}(\tau) : \tau \in A_n \cup \{ \sigma \} \}$ is separated 
	on $Y \res (\gamma+1)$;
	\item let $h_{n,\sigma}^+$ be the partial injective function 
	from $Y$ to $Y$ defined by letting, 
	for all $x, y \in Y$, $h_\sigma^+(x) = y$ 
	iff $p_{n+1}(\sigma)(x \res (\gamma+1)) = y \res (\gamma+1)$; 
	then $z$ is in the domain and range of $h_{n,\sigma}^+$.
	\end{enumerate}

	Let $X_{n+1} = Y$, $A_{n+1} = A_n \cup \{ \sigma \}$, and $\delta_{n+1} = \gamma+1$. 
	For each $\tau \in A_{n+1}$, 
	define a partial injective function $h_{n+1,\tau}$ from $X_{n+1}$ 
	to $X_{n+1}$ as described in inductive hypothesis (4). 
	Note that $h_{n+1,\sigma} = h_{n,\sigma}^+$, so $z$ is in the 
	domain and range of $h_{n+1,\sigma}$.
	The inductive hypotheses are clearly satisfied.
	
	This completes the construction. 
	By our bookkeeping, it is routine to check that the assumptions of 
	Lemma \ref{Constructing Total Master Conditions} (Constructing Total Master Conditions) hold. 
	Fix a total master condition $q$ over $N$ such that $q \le p_n$ for all $n$. 
	Reviewing Case (b) and our bookkeeping, it is easy to show $q$ forces 
	that $T_\delta \subseteq \dot E$.
\end{proof}

\begin{lemma} \label{master implies separated}
	Let $\lambda$ be a large enough regular cardinal and let $N$ be a countable 
	elementary substructure of $H(\lambda)$ which contains as members 
	$T$, $\ka$, $\q$, and $\p$. 
	If $q$ is a total master condition over $N$ such that 
	$\dom(q) = N \cap \ka$, then $q$ is separated.
\end{lemma}

\begin{proof}
	Let $\delta = N \cap \omega_1$. 
	By Lemma \ref{separated iff}, to show that $q$ is separated it suffices to show that 
	whenever $A \subseteq \dom(q)$ and $X \subseteq T_\delta$ are finite sets, 
	then $\{ q(\tau) : \tau \in A \}$ is separated on $X$. 
	Fix $\xi < \delta$ large enough so that $X$ has unique drop-downs to $\xi$. 
	Let $D$ be the set of conditions $s$ with top level at least $\xi$ 
	such that $A \subseteq \dom(s)$ and $s$ is separated. 
	By Lemmas \ref{extending domain} and 
	\ref{Separated Conditions are Dense} (Separated Conditions are Dense), 
	$D$ is dense and $D \in N$ by elementarity. 
	Since $q$ is a total master condition, we can find 
	some $s \in N \cap D$ such that $q \le s$. 
	Let $\rho$ be the top level of $s$. 
	As $\rho \ge \xi$, $X$ has unique drop-downs to $\rho$. 
	Since $s$ is separated, 
	$\{ s(\tau) : \tau \in A \}$ is separated on $X \res \rho$. 
	By Lemma \ref{Persistence for Sets 2} (Persistence for Sets), $\{ q(\tau) : \tau \in A \}$ is separated on $X$.
\end{proof}

\begin{proposition} \label{omega2 cc}
	Assuming $\textsf{CH}$, the forcing poset $\p$ is $\omega_2$-c.c.
\end{proposition}

This follows by a standard application of the $\Delta$-system lemma, 
assuming \textsf{CH}, to an $\omega_2$-sized collection of countable sets.

Let us say that two automorphisms of $T$ are \emph{almost disjoint} if they agree 
on only countably many elements of $T$, or in other words, their graphs have 
countable intersection.
 
\begin{proposition}
	Suppose that $T$ is a free Suslin tree. 
	Then $\p$ forces that there exists an almost disjoint sequence of length $\ka$ 
	of automorphisms of $T$.
\end{proposition}

\begin{proof}
	Let $G$ be a generic filter on $\p$. 
	For each $\tau < \ka$, let 
	$f_\tau := \bigcup \{ p(\tau) : p \in G, \ \tau \in \dom(p) \}$. 
	By Lemma \ref{extending domain} and a density argument, it is easy to check that 
	each $f_\tau$ is an automorphism of $T$. 
	Consider $\tau_0 < \tau_1 < \ka$. 
	By Lemma \ref{another dense set}, we can find 
	$\alpha = \alpha_{\tau_0,\tau_1} < \omega_1$ such that for all $x \in T_{\alpha}$, 
	$f_{\tau_0}(x) \ne f_{\tau_1}(x)$.
	But if $x <_T y$ and $f_{\tau_0}(x) \ne f_{\tau_1}(x)$, then 
	$f_{\tau_0}(y) \ne f_{\tau_1}(y)$. 
	So $\{ x \in T : f_{\tau_0}(x) = f_{\tau_1}(x) \} \subseteq T \res \alpha_{\tau_0,\tau_1}$.
\end{proof}

\begin{proposition}
	Suppose that $T$ is a free Suslin tree, $\ka \ge \omega_2$, and $\textsf{CH}$ holds. 
	Then $\p$ forces that $T$ is an almost Kurepa Suslin tree.
\end{proposition}

\begin{proof}
	Let $G$ be a generic filter on $\p$. 
	By Theorem \ref{proper and preserve suslin}, $T$ is Suslin in $V[G]$ and by 
    Proposition \ref{omega2 cc}, $\omega_2$ is preserved. 
	In $V[G]$, let $\{ f_\tau : \tau < \ka \}$ 
	be an almost disjoint family of automorphisms of $T$. 
	Force with the c.c.c.\ forcing $T$ over $V[G]$ to get a generic branch $b$ of $T$. 
	In $V[G][b]$, define $b_\tau = f_\tau[b]$ 
	for all $\tau < \ka$. 
	For any $\tau_0 < \tau_1 < \ka$, since $f_{\tau_0}$ and $f_{\tau_1}$ 
	are almost disjoint, it is easy to conclude that $b_{\tau_0} = f_{\tau_0}[b]$ and 
	$b_{\tau_1} = f_{\tau_1}[b]$ have countable intersection. 
	So $\langle b_{\tau} : \tau < \ka \rangle$ is a sequence of $\ka$-many 
	distinct cofinal branches of $T$. 
	Hence, $T$ is a Kurepa tree in $V[G][b]$.
\end{proof}

\section{More About Constructing and Extending Automorphisms} \label{More About Constructing and Extending Automorphisms}

We now turn toward proving that the automorphism forcing does not add new 
cofinal branches of $\omega_1$-trees appearing in certain intermediate extensions. 
In this section, we prove three technical lemmas about extending countable families 
of automorphisms one level higher in order to achieve some desirable properties. 
These lemmas anticipate configurations which appear in proofs occurring in later sections. 
Specifically, 
Lemmas \ref{elaborate extending by one} 
and \ref{elaborate extending by one 2} are 
used in the proof of Lemma \ref{very technical lemma}, and 
Lemma \ref{elaborate general extension} is 
used in the proof of Lemma \ref{Generalized Separated Conditions are Dense}. 
The proofs of Lemmas \ref{elaborate extending by one} and 
\ref{elaborate extending by one 2} are fairly simple and almost identical. 
Lemma \ref{elaborate general extension} is essentially an expansion of 
Lemma \ref{more extending by one} to a more elaborate context.

\begin{lemma} \label{elaborate extending by one}
	Let $\gamma < \omega_1$ and let $\{ f_\tau : \tau \in I \}$ be a countable family 
	of automorphisms of $T \res (\gamma+1)$. 
	Let $A \subseteq B \subseteq I$ be finite. 
	Let $Y, Z \subseteq T_{\gamma+1}$ be finite sets 
	each with unique drop-downs to $\gamma$ 
	such that $Y \cap Z = \emptyset$. 
	Then there exists a family $\{ g_\tau : \tau \in I \}$ 
	of automorphisms of $T \res (\gamma+2)$ such that:
	\begin{enumerate}
	\item for all $\tau \in I$, $f_\tau \subseteq g_\tau$;
	\item for all $\tau \in B$, $Y \res \gamma$ and $Y$ are $g_\tau$-consistent;
	\item for all $\tau \in A$, $Z \res \gamma$ and $Z$ are $g_\tau$-consistent;
	\item for all $\tau \in B \setminus A$ and for all $x \in Z$, 
	$g_\tau(x)$ and $g_\tau^{-1}(x)$ are not in $Y \cup Z$.
	\end{enumerate}
\end{lemma}

\begin{proof}
	Fix a bijection $h : \omega \to T_{\gamma+1} \times I$. 
	Let $g_\tau \res (\gamma+1) = f_\tau$ for all $\tau \in I$. 
	We define the values of the functions $g_\tau$ on $T_{\gamma+1}$ 
	in $\omega$-many stages, where at any given stage 
	we have defined only finitely many values for finitely many $g_\tau$'s. 
	We also define a subset-increasing sequence $\langle X_n : n < \omega \rangle$ 
	of finite subsets of $T_{\gamma+1}$.

	At stage $0$, for all $\tau \in B$ and $x, y \in Y$, let $g_\tau(x) = y$ 
	iff $f_\tau(x \res \gamma) = y \res \gamma$. 
	And for all $\tau \in A$ and $x, y \in Z$, let $g_\tau(x) = y$ 
	iff $f_\tau(x \res \gamma) = y \res \gamma$. 
	Let $X_0 = Y \cup Z$.
	
	Now let $n < \omega$ and assume that we have completed stage $n$. 
	In particular, we have defined the finite set $X_n \subseteq T_{\gamma+1}$. 
	Consider $h(n) = (z,\sigma)$. 
	Stage $n+1$ consists of two steps. 
	For the first step, if $g_\sigma(z)$ is already defined, then move on to step $2$. 
	Otherwise, define $g_\sigma(z)$ to be some element of $T_{\gamma+1}$ 
	above $f_\sigma(z \res \gamma)$ which is not in $X_n \cup \{ z \}$. 
	This is possible since $T$ is infinitely splitting. 
	For the second step, if $g_\sigma^{-1}(z)$ is already defined, then we are done. 
	Otherwise, define $g_\sigma^{-1}(z)$ to be some element of $T_{\gamma+1}$ 
	which is above $f_{\sigma}^{-1}(z \res \gamma)$ and is not in 
	$X_n \cup \{ z, g_\sigma(z) \}$. 
	Again, this is possible since $T$ is infinitely splitting. 
	Define $X_{n+1} = X_n \cup \{ z, g_\sigma(z), g_\sigma^{-1}(z) \}$.
	
	This completes the construction. 
	It is routine to check that this works, using what we did at stage $0$ to show 
	(2) and (3), and using the sets $X_n$ to show injectivity and (4).
\end{proof}

\begin{lemma} \label{elaborate extending by one 2}
	Let $\gamma < \omega_1$ and let 
	$\{ f_\tau : \tau \in I \}$ be a family of 
	automorphisms of $T \res (\gamma+1)$. 
	Let $A \subseteq B \subseteq I$ be finite sets. 
	Let $b_0,\ldots,b_{n-1}$ be distinct elements of $T_\gamma$, and let 
    $c_0,\ldots,c_{n-1}$ and $d_0,\ldots,d_{n-1}$ be distinct elements of 
	$T_{\gamma+1}$ such that for all $k < n$, $b_k <_T c_k$ and 
	$b_k <_T d_k$. 
	Define $C = \{ c_k : k < n \}$ and $D = \{ d_k : k < n \}$. 
	Let $Y \subseteq T_{\gamma+1}$ be finite 
	and assume that $(C \cup D) \cap Y = \emptyset$.

	Then there exists a family $\{ g_\tau : \tau \in I \}$ of automorphisms 
	of $T \res (\gamma+2)$ such that:
	\begin{enumerate}
	\item for all $\tau \in I$, $f_\tau \subseteq g_\tau$;
	\item for all $\tau \in A$, the tuples $(b_0,\ldots,b_{n-1})$ and $(c_0,\ldots,c_{n-1})$ 
	are $g_\tau$-consistent and the tuples 
	$(b_0,\ldots,b_{n-1})$ and $(d_0,\ldots,d_{n-1})$ 
	are $g_\tau$-consistent;
	\item for all $x \in C$: if $\tau \in A$, then $g_\tau(x)$ and $g_\tau^{-1}(x)$ 
	are not in $Y \cup D$, and if $\tau \in B \setminus A$, then 
	$g_\tau(x)$ and $g_\tau^{-1}(x)$ are not in $C \cup D \cup Y$;
	\item for all $x \in D$: if $\tau \in A$, then $g_\tau(x)$ and $g_\tau^{-1}(x)$ 
	are not in $Y \cup C$, and if $\tau \in B \setminus A$, 
	then $g_\tau(x)$ and $g_\tau^{-1}(x)$ are not in $C \cup D \cup Y$;
	\item for all $x \in Y$ and $\tau \in B$, 
	$g_\tau(x)$ and $g_\tau^{-1}(x)$ are 
	not in $C \cup D \cup Y$.
	\end{enumerate}
\end{lemma}

\begin{proof}
	Fix a bijection $h : \omega \to T_{\gamma+1} \times I$. 
	Let $g_\tau \res (\gamma+1) = f_\tau$ for all $\tau \in I$. 
	We define the values of the functions $g_\tau$ on $T_{\gamma+1}$ 
	in $\omega$-many stages, where at any given stage 
	we have defined only finitely many values for finitely many $g_\tau$'s. 
	We also define a subset-increasing sequence $\langle X_n : n < \omega \rangle$ 
	of finite subsets of $T_{\gamma+1}$.

	At stage $0$, for all $\tau \in A$ and $i,j < n$, 
	define $g_\tau(c_i) = c_j$ iff $f_\tau(b_i) = b_j$ and 
	$g_\tau(d_i) = d_j$ iff $f_\tau(b_i) = b_j$. 
	Let $X_0 = Y \cup C \cup D$.
	
	Now let $n < \omega$ and assume that we have completed stage $n$. 
	In particular, the finite set $X_n \subseteq T_{\gamma+1}$ has been defined. 
	Consider $h(n) = (z,\sigma)$. 
	Stage $n+1$ consists of two steps. 
	For the first step, if $g_\sigma(z)$ is already defined, then move on to step $2$. 
	Otherwise define $g_\sigma(z)$ to be some element of $T_{\gamma+1}$ 
	above $f_\sigma(z \res \gamma)$ which is not in $X_n \cup \{ z \}$. 
	This is possible since $T$ is infinitely splitting. 
	For the second step, if $g_\sigma^{-1}(z)$ is already defined, then we are done. 
	Otherwise define $g_\sigma^{-1}(z)$ to be some element of $T_{\gamma+1}$ 
	which is above $f_{\sigma}^{-1}(z)$ and not in $X_n \cup \{ z, g_\sigma(z) \}$. 
	Again, this is possible since $T$ is infinitely splitting. 
	Define $X_{n+1} = X_n \cup \{ z, g_\sigma(z), g_\sigma^{-1}(z) \}$.
	
	This completes the construction. 
	It is routine to check that this works.
\end{proof}

\begin{lemma} \label{elaborate general extension}
	Assume the following:
	\begin{itemize}
	\item $\gamma < \omega_1$ and $n < \omega$;
	\item $X \subseteq T_{\gamma+1}$ is finite and has unique drop-downs to $\gamma$;
	\item $\{ f_\tau : \tau \in I \}$ is a 
	countable collection of automorphisms of $T \res (\gamma+1)$;
	\item $\{ I_0 \} \cup \{ J_k : k < n \}$ is a partition of $I$;
	\item $\{ A \} \cup \{ A_k : k < n \}$ 
	is a family of finite sets, where 
	$A \subseteq I_0$ and $A_k \subseteq J_k$ 
	for each $k < n$;
	\item for all $k < n$, $\{ f_\tau : \tau \in A \cup A_k \}$ is separated 
	on $X \res \gamma$.
	\end{itemize}
	Then there exists a family $\{ g_\tau : \tau \in I \}$ of automorphisms 
	of $T \res (\gamma+2)$ satisfying:
	\begin{enumerate}
	\item $f_\tau \subseteq g_\tau$ for all $\tau \in I$;
	\item for all $\tau \in A \cup \bigcup_{k < n} A_k$, 
	$X \res \gamma$ and $X$ are $g_\tau$-consistent;
	\item for all $k < n$, $\{ g_\tau : \tau \in I_0 \cup J_k \}$ is separated.
	\end{enumerate}
\end{lemma}

\begin{proof}
	Fix a bijection $h : \omega \to T_{\gamma+1} \times I$. 
	For each $\tau \in I$, define $g_\tau \res (\gamma+1) = f_\tau$.

	We define the values of the functions $g_\tau$ on $T_{\gamma+1}$ 
	in $\omega$-many stages. 
	In addition, for each $k < n$ we define 
	an injective sequence $\langle a_l^k : l < \omega \rangle$ 
	which enumerates $T_{\gamma+1}$. 
	At any given stage $p < \omega$, we have defined a set $X_p$ 
	of some finite size $l_p$, and also defined, for each $k < n$, 
	an injective enumeration 
	$\langle a_l^k : l < l_p \rangle$ of $X_p$ which is an initial segment 
	of the sequence $\langle a_l^k : l < \omega \rangle$.

	We maintain the following inductive hypotheses:
	\begin{enumerate}
	\item[(i)] for all $p < \omega$, if the value $g^m_\tau(a) = b$ is defined at stage $p$, 
	where $\tau \in I$ and $m \in \{ -1, 1 \}$, 
	then $a$ and $b$ are in $X_p$, $a \ne b$, and 
	$f^m_\tau(a \res \gamma) = b \res \gamma$;
	\item[(ii)] for all $p_0 < p_1 < \omega$, if $a$ and $b$ are in $X_{p_0}$ and 
	$g^m_\tau(a) = b$ has been defined by the end of stage $p_1$, 
	where $\tau \in I$ and $m \in \{ -1, 1 \}$, then 
	$g^m_\tau(a) = b$ has been defined by the end of stage $p_0$;
	\item[(iii)] for all $k < n$, $p < \omega$, and $l < l_p$, 
	there exists at most one triple $(j,m,\tau)$, 
	where $j < l$, $m \in \{ -1, 1 \}$, 
	and $\tau \in I_0 \cup J_k$, 
	such that $g_\tau^m(a_l^k)$ has been defined 
	by the end of stage $p$ and $g^m_\tau(a_l^k) = a_j^k$.
	\end{enumerate}

	\underline{Stage $0$:} For each $\tau \in A \cup \bigcup_{k < n} A_k$ 
	and $x, y \in X$, define 
	$g_\tau(x) = y$ iff $f_\tau(x \res \gamma) = y \res \gamma$. 
	Define $X_0 = X$ and $l_0 = |X|$. 
	For each $k < n$, since 
	$\{ f_\tau : \tau \in A \cup A_k \}$ is separated on $X \res \gamma$, 
	we can fix an injective sequence $\langle a_0^k,\ldots,a_{l_0-1}^k \rangle$ 
	which lists the elements of $X$ so that 
	$\{ f_\tau : \tau \in A \cup A_k \}$ 
	is separated on $\vec a^k \res \gamma$, where $\vec a^k = (a_0^k,\ldots,a_{l_0-1}^k)$.

	Let us check that the inductive hypotheses hold. 
	(i) Suppose $g^m_\tau(a) = b$ is defined at stage $0$, where $\tau \in I$ 
	and $m \in \{ -1, 1 \}$. 
	By flipping $a$ and $b$ if necessary, we may assume without 
	loss of generality that $m = 1$. 
	Then by construction, $a$ and $b$ are in $X = X_0$, 
	$\tau \in A \cup \bigcup_{k < n} A_k$, 
	and $f_\tau(a \res \gamma) = b \res \gamma$. 
	Fix $k < n$ such that $\tau \in A \cup A_k$. 
	Since $\{ f_\tau : \tau \in A \cup A_k \}$ is separated on $X \res \gamma$, 
	$a \res \gamma \ne b \res \gamma$, and hence $a \ne b$.
	
	(ii) is vacuously true. 
	(iii) Fix $k < n$ and $l < l_0$. 
	Suppose that there exists a triple $(j,m,\tau)$, 
	where $j < l$, $m \in \{ -1, 1 \}$, 
	and $\tau \in I_0 \cup J_k$, 
	such that $g_\tau^m(a_l^k) = a_j^k$ has been defined 
	by the end of stage $0$. 
	Then by what we did at stage $0$, $\tau \in A \cup A_k$. 
	Since $\{ f_\tau : \tau \in A \cup A_k \}$ is separated on 
	$\vec a^k \res \gamma$, the triple $(j,m,\tau)$ must be unique.

	\underline{Stage $p+1$:} Let $p < \omega$ and suppose that stage $p$ is complete. 
	In particular, we have defined $X_p$ and for each $k < n$ we have defined 
	an injective sequence $\langle a^k_l : l < l_p \rangle$ 
	which lists the elements of $X_p$ satisfying the 
	required properties. 
	Let $h(p) = (z,\sigma)$. 

	We define $g_\sigma(z)$ and $g_\sigma^{-1}(z)$ in two steps, 
	where at each step we use the fact that $T$ is infinitely splitting. 
	In the first step, if $g_{\sigma}(z)$ is already defined, 
	then move on to step two. 
	Otherwise, define $g_{\sigma}(z)$ to be some member of $T_{\gamma+1}$ 
	which is above $f_\sigma(z \res \gamma)$ and is not in 
	$X_p \cup \{ z \}$. 
	In the second step, if $g_{\sigma}^{-1}(z)$ is already defined, then we are done. 
	Otherwise, define $g_{\sigma}^{-1}(z)$ to be some member of $T_{\gamma+1}$ 
	which is above $g_\sigma^{-1}(z \res \gamma)$ and 
	is not in $X_p \cup \{ z, g_{\sigma}(z) \}$.

	Let $X_{p+1} = X_p \cup \{ z, g_\sigma(z), g_\sigma^{-1}(z) \}$ and 
	$l_{p+1} = |X_{p+1}|$. 
	For each $k < n$, define 
	$\langle a_l^k : l < l_{p+1} \rangle$ by adding at the end of the sequence 
	$\langle a_l^k : l < l_p \rangle$ those elements among 
	$z$, $g_{\sigma}(z)$, and $g_{\sigma}^{-1}(z)$ which are not already in $X_p$, 
	in the order just listed.
	
	Let us check that inductive hypotheses (i)-(iii) hold for $p+1$. 
	(i) is clear. 
	For (ii), the only new equations of the form 
	$g^m_\tau(a) = b$ which were introduced at stage $p+1$, where 
	$\tau \in I$, $m \in \{ -1, 1 \}$, and $a, b \in T_{\gamma+1}$, is when 
	at least one of $a$ or $b$ is in $X_{n+1} \setminus X_n$. 
	So (ii) easily follows from the inductive hypothesis.

	Now we prove (iii). 
	Fix $k < n$. 
	Consider first the case when $z$ is not in $X_p$. 
	Then by inductive hypothesis (i), neither 
	$g_\sigma(z)$ nor $g_\sigma^{-1}(z)$ were defined at any stage 
	earlier than $p+1$. 
	So by how we defined $g_\sigma(z)$ and $g_\sigma^{-1}(z)$ at stage $p+1$, 
	$g_\sigma(z)$ and $g_{\sigma}^{-1}(z)$ are not in $X_p$. 
	Hence, the last three elements of $\langle a_l^k : l < l_{p+1} \rangle$ 
	are $z$, $g_\sigma(z)$, and $g_{\sigma}^{-1}(z)$. 
	The relations introduced between these three elements at stage $p+1$ 
	yield no counter-example to (iii), and $z$, $g_\sigma(z)$, and $g_\sigma^{-1}(z)$ 
	have no relations to any elements of $\langle a_l^k : l < l_p \rangle$. 
	So (iii) follows by the inductive hypothesis.
	
	Next, consider the case when $z$ is in $X_p$. 
	Then $z$ already appears on the sequence $\langle a_l^k : l < l_p \rangle$. 
	At stage $p+1$, no new relations are introduced between elements of 
	$\langle a_l^k : l < l_p \rangle$. 
	Each new element in $X_{p+1} \setminus X_p$ has exactly one relation with 
	elements of $X_{p+1}$, namely with $z$. 
	So (iii) follows by the inductive hypothesis.

	This completes the construction. 
	It is routine to check that 
	each $g_\tau$ is an automorphism of $T \res (\gamma+2)$ extending $f_\tau$. 
	By what we did at stage $0$, for all 
	$\tau \in A \cup \bigcup_n A_n$, 
	$X \res \gamma$ and $X$ are $g_\tau$-consistent.

	Now consider $k < n$, and we show that 
	$\{ g_\tau : \tau \in I_0 \cup J_k \}$ is separated. 
	By inductive hypothesis (i), $g_\tau$ has no fixed points in $T_{\gamma+1}$ 
	for any $\tau \in I_0 \cup J_k$. 
	Let $Y \subseteq T_{\gamma+1}$ be finite. 
	Fix a large enough $p < \omega$ so that $Y \subseteq X_p$. 
	Then by Lemma \ref{subset}, 
	it suffices to show that $\{ g_\tau : \tau \in I_0 \cup J_k \}$ is separated on $X_p$ 
	as witnessed by the tuple $(a_0^k,\ldots,a_{l_p-1}^k)$. 
	Suppose that $l < l_p$ and $(j,m,\tau)$ satisfies that 
	$j < l$, $m \in \{ -1, 1 \}$, $\tau \in I_0 \cup J_k$, and 
	$g_{\tau}^{m}(a_l^k) = a_j^k$. 
	By inductive hypothesis (ii), 
	the relation $g_\tau^m(a_l^k) = a_j^k$ was introduced by the end of stage $p$. 
	By inductive hypothesis (iii), there is at most one such triple.
\end{proof}

\section{Regular Suborders and Generalized Properties} \label{Regular Suborders and Generalized Properties}

In this section, we generalize many of the main properties of $\p$ to the 
context of regular suborders of $\p$. 

\begin{definition}
	For any set $X \subseteq \ka$, let $\p_X$ denote the suborder of $\p$ consisting 
	of all $p \in \p$ such that $\dom(p) \subseteq X$.
\end{definition}

The proof of the following is routine.

\begin{proposition}
	For any set $X \subseteq \ka$, $\p_X$ is a regular suborder of $\p$.
\end{proposition}

Of particular interest for us is $\p_\theta$, where $\theta < \ka$. 
The main goal of the remainder of the article is to show that 
whenever $\theta < \ka$ and $\dot U$ is a $\p_\theta$-name for an $\omega_1$-tree, 
then $\p$ forces that any cofinal branch of $\dot U$ in $V^\p$ is in $V^{\p_\theta}$. 
When $\theta < \ka$ is fixed, we write $\le_\theta$ for the ordering on $\p_\theta$, 
mainly to emphasize that the conditions we are relating are in $\p_\theta$. 
On the other hand, when we write $\le$ we mean the ordering on $\p$.

Many previously discussed properties of $\p$ were of the form that some condition can be 
extended to a higher level satisfying some additional information. 
Now we have finitely many conditions, all with the same restriction to $\p_\theta$, 
and we simultaneously extend those conditions so that the extended conditions 
also have the same restriction to $\p_\theta$. 
We refer to this type of result as \emph{generalized} versions of the earlier results.

\begin{lemma}[Simple Generalized Extension] \label{Simple Generalized Extension}
	Let $\theta < \ka$. 
	Assume the following:
	\begin{itemize}
	\item $\gamma \le \xi < \omega_1$;
	\item $p \in \p$ has top level $\gamma$, 
	$w \in \p_\theta$ has top level $\xi$, 
	and $w \le_\theta p \res \theta$;
	\item $B \subseteq \dom(p)$;
	\item $X$ is a finite subset of $T_\xi$ with unique drop-downs to $\gamma$;
	\item for all $\tau \in B \cap \theta$, 
	$X \res \gamma$ and $X$ are $w(\tau)$-consistent.
	\end{itemize}
	Then there exists a condition $q \le p$ with top level $\xi$ and domain 
	equal to $\dom(w) \cup \dom(p)$ such that 
	$q \res \theta = w$ and for all $\tau \in B$, 
	$X \res \gamma$ and $X$ are $q(\tau)$-consistent.
\end{lemma}

\begin{proof}
	Apply Proposition \ref{Extension 2} (Extension) to the condition $p \res [\theta,\ka)$ 
	to find a condition $s \le p \res [\theta,\ka)$ with top level $\xi$ 
	and with the same domain as $p \res [\theta,\ka)$ such that for all 
	$\tau \in B \setminus \theta$, $X \res \gamma$ and $X$ are $s(\tau)$-consistent. 
	Now let $q = w \cup s$.
\end{proof}

\begin{lemma}[Generalized Extension] \label{Generalized Extension} 
	Let $\theta < \ka$. 
	Assume the following:
	\begin{itemize}
	\item $\gamma \le \xi < \omega_1$;
	\item $\{ p_0,\ldots,p_{n-1} \}$ is a finite set of conditions in 
	$\p$ all with top level $\gamma$;
	\item $v \in \p_\theta$ and for all $k < n$, $p_k \res \theta = v$;
	\item $B \subseteq \bigcap_{k < n} \dom(p_k)$;
	\item $X$ is a finite subset of $T_\xi$ with unique drop-downs to $\gamma$.
	\end{itemize}
	Then there exists a set of conditions $\{ \hat{p}_0,\ldots,\hat{p}_{n-1} \}$ 
	all with top level $\xi$ and there exists some 
	$\hat{v} \in \p_\theta$ such that for all $k < n$:
	\begin{enumerate}
	\item $\hat{p}_k \le p_k$;
	\item $\hat{p}_k \res \theta = \hat{v}$;
	\item for all $\tau \in B$, 
	$X \res \gamma$ and $X$ are $\hat{p}_k(\tau)$-consistent.
	\end{enumerate}
	Moreover, if $w \le_\theta v$ is a fixed condition with top level $\xi$ 
	such that for all $\tau \in B \cap \theta$, 
	$X \res \gamma$ and $X$ are $w(\tau)$-consistent, then we can also arrange 
	that $\hat{v} = w$.
\end{lemma}

\begin{proof}
	We apply Proposition \ref{Extension 2} (Extension) several times. 
	In the case of the moreover clause, let $\hat{v} = w$. 
	Otherwise, apply Proposition \ref{Extension 2} (Extension) to 
	find $\hat{v} \le v$ with top level $\xi$ and the same domain as $v$ 
	such that for all $\tau \in B \cap \theta$, $X \res \gamma$ and $X$ 
	are $\hat{v}(\tau)$-consistent. 
	For each $k < n$, apply Proposition \ref{Extension 2} (Extension) to 
	find $s_k \le p_k \res [\theta,\ka)$ with top level $\xi$ 
	and the same domain as $p_k \res [\theta,\ka)$ such that 
	for all $\tau \in B \setminus \theta$, $X \res \gamma$ and $X$ 
	are $p_k(\tau)$-consistent. 
	Now let $\hat{p}_k = \hat{v} \cup s_k$ for all $k < n$.
\end{proof}

\begin{lemma}[Generalized Consistent Extensions Into Dense Sets] \label{Generalized Consistent Extensions Into Dense Sets}
	Suppose that $T$ is a free Suslin tree. 
	Let $\theta < \ka$. 
	Let $\lambda$ be a large enough regular cardinal and let $N$ be a countable 
	elementary substructure of $H(\lambda)$ which contains as members 
	$T$, $\ka$, $\q$, $\p$, and $\theta$. 
	Let $\delta = N \cap \omega_1$. 
	Assume that:
	\begin{itemize}
	\item $D \in N$ is a dense open subset of $\p$;
	\item $\{ p_0,\ldots,p_{n-1} \}$ is a finite set of conditions in $N \cap \p$ 
	all with top level $\xi$;
	\item $v \in N \cap \p_\theta$ and for all $k < n$, 
	$p_{k} \res \theta = v$;
	\item $X \subseteq T_\delta$ is finite and has unique drop-downs to $\xi$;
	\item $B \subseteq \bigcap_{k < n} \dom(p_{k})$ is finite;
	\item for each $k < n$, $\{ p_{k}(\tau) : \tau \in B \}$ 
	is separated on $X \res \xi$.
	\end{itemize}

	Then there exist $q_0,\ldots,q_{n-1}$ in $N \cap D$, 
	$w \in N \cap \p_\theta$, 
	and $\gamma < \delta$ satisfying that for all $k < n$, 
	\begin{enumerate}
		\item $q_{k} \le p_{k}$, 
		$q_{k}$ has top level $\gamma$, and $q_{k} \res \theta = w$;
		\item for all $\tau \in B$, $X \res \xi$ and $X \res \gamma$ 
		are $q_{k}(\tau)$-consistent.
	\end{enumerate}
\end{lemma}

\begin{proof}
	The proof is by induction on $n$. 
	For the base case $n = 1$, the statement follows immediately from 
	Proposition \ref{Consistent Extensions Into Dense Sets} (Consistent Extensions Into Dense Sets). 
	Now assume that $n \ge 1$ and the statement holds for $n$. 
	Consider $D$, $\{ p_0,\ldots,p_{n} \}$, $\xi$, $v$, $X$, and $B$ as above. 
	By the inductive hypothesis applied to the set $\{ p_0,\ldots,p_{n-1} \}$, 
	fix $q_0,\ldots,q_{n-1}$ in $N \cap D$, 
	$w \in N \cap \p_\theta$, 
	and $\gamma < \delta$ satisfying properties (1) and (2). 
	In particular, (2) implies that 
	for all $\tau \in B \cap \theta$, 
	$X \res \xi$ and $X \res \gamma$ are $w(\tau)$-consistent.

	Apply Lemma \ref{Simple Generalized Extension} (Simple Generalized Extension) 
	inside $N$ to find $q_n \le p_n$ with top level $\gamma$ 
	such that $q_n \res \theta = w$ 
	and for all $\tau \in B$, $X \res \xi$ and $X \res \gamma$ 
	are $q_n(\tau)$-consistent. 
	Since $\{ p_{n}(\tau) : \tau \in B \}$ 
	is separated on $X \res \xi$ and $q_n \le p_n$, 
	by Lemma \ref{Persistence for Sets 2} (Persistence for Sets), 
	$\{ q_n(\tau) : \tau \in B \}$ is separated on $X \res \gamma$. 
	So we can apply Proposition \ref{Consistent Extensions Into Dense Sets} (Consistent Extensions Into Dense Sets) 
	to find $\bar{q}_{n} \le q_{n}$ in $N \cap D$ 
	with some top level $\rho < \delta$ and some $\bar{w} \in N \cap \p_\theta$ 
	such that $\bar{q}_n \res \theta = \bar{w}$ and for all $\tau \in B$, 
	$X \res \gamma$ and $X \res \rho$ are $\bar{q}_n(\tau)$-consistent. 
	Now apply Lemma \ref{Generalized Extension} (Generalized Extension) 
	inside $N$ to find a family $\{ \bar{q}_k : k < n \}$ of conditions in $N$ 
	such that for all $k < n$, $\bar{q}_k \le q_{k}$, 
	$\bar{q}_k \res \theta = \bar{w}$, 
	and for all $\tau \in B$, $X \res \gamma$ and $X \res \rho$ are 
	$\bar{q}_k(\tau)$-consistent. 
	Since $D$ is open, each $\bar{q}_k$ is in $D$. 
	So $\bar{q}_0,\ldots,\bar{q}_n$ and $\bar w$ are as required.
\end{proof}

\begin{lemma}[Generalized Separated Conditions are Dense] \label{Generalized Separated Conditions are Dense}
	Let $\theta < \ka$. 
	Assume the following:
	\begin{itemize}
	\item $\{ p_0,\ldots,p_{n-1} \}$ is a finite set of conditions in 
	$\p$ all with top level $\gamma$;
	\item $v \in \p_\theta$ and for all $k < n$, $p_k \res \theta = v$;
	\item $X$ is a finite subset of $T_{\gamma+1}$ with unique drop-downs to $\gamma$;
	\item $B \subseteq \bigcap_{k < n} \dom(p_k)$ is finite and 
	for all $k < n$, $\{ p_k(\tau) : \tau \in B \}$ is separated on $X \res \gamma$.
	\end{itemize}
	Then there exist $q_0,\ldots,q_{n-1}$ in $\p$ and $w \in \p_\theta$, 
	all with top level $\gamma+1$, such that 
	$w \le_\theta v$ and for all $k < n$, $q_k \le p_k$, 
	$q_k \res \theta = w$, $q_k$ is separated, and for all $\tau \in B$, 
	$X \res \gamma$ and $X$ are $q_{k}(\tau)$-consistent.
\end{lemma}

\begin{proof}
	Let $I_0 = \dom(v)$ and $A = B \cap \theta$. 
	For each $k < n$, let $J_k = \{ k \} \times (\dom(p_k) \setminus \theta)$ and 
	$A_k = \{ k \} \times (B \setminus \theta)$. 
	Let $I = I_0 \cup \bigcup_{k < n} J_k$. 
	For all $\tau \in I_0$, let $f_\tau = v(\tau)$, and for all 
	$k < n$ and $\tau \in \dom(p_k) \setminus \theta$, 
	let $f_{(k,\tau)} = p_k(\tau)$.

	We would like to apply Lemma \ref{elaborate general extension} to the above objects. 
	The first five assumptions of this lemma clearly hold. 
	For the last assumption, 
	we need to show that for all $k < n$, $\{ f_i : i \in A \cup A_k \}$ 
	is separated on $X \res \gamma$. 
	Define $h : A \cup A_k \to B$ by $h(\tau) = \tau$ for all $\tau \in A$, 
	and $h((k,\tau)) = \tau$ for all $(k,\tau) \in A_k$. 
	Then $h$ is a bijection and 
	$f_i = p_k(h(i))$ for all $i \in A$. 
	Since $\{ p_k(\tau) : \tau \in B \}$ is separated on $X \res \gamma$, 
	by Lemma \ref{composition} so is $\{ f_i : i \in A \cup A_k \}$.

	Applying Lemma \ref{elaborate general extension} fix a 
	family $\{ g_i : i \in I \}$ of automorphisms 
	of $T \res (\gamma+2)$ satisfying:
	\begin{enumerate}
	\item $f_i \subseteq g_i$ for all $i \in I$;
	\item for all $i \in A \cup \bigcup_{k < n} A_k$, 
	$X \res \gamma$ and $X$ are $g_i$-consistent;
	\item for all $k < n$, $\{ g_i : i \in I_0 \cup J_k \}$ is separated.
	\end{enumerate}

	Define $w$ with the same domain as $v$ so that for all $\tau \in \dom(v)$, 
	$w(\tau) = g_\tau$. 
	For each $k < n$, define $q_k$ with the same domain as $p_k$ so that 
	$q_k \res \theta = w$ and for all 
	$\tau \in \dom(p_k) \setminus \theta$, $q_k(\tau) = g_{(k,\tau)}$. 
	It is easy to check that each $q_k$ is a condition below $p_k$ with 
	top level $\gamma+1$.

	Let $k < n$. 
	We claim that for all $\tau \in B$, 
	$X \res \gamma$ and $X$ are $q_{k}(\tau)$-consistent. 
	If $\tau \in B \cap \theta$, then $q_k(\tau) = w(\tau) = g_\tau$ 
	and $\tau \in A$, and by (2), 
	$X \res \gamma$ and $X$ are $g_\tau$-consistent. 
	Suppose that $\tau \in B \setminus \theta$. 
	Then $(k,\tau) \in A_k$ and $q_k(\tau) = g_{(k,\tau)}$. 
	By (2), $X \res \gamma$ and $X$ are $g_{(k,\tau)}$-consistent.
	
	Finally, we claim that $q_k$ is separated. 
	So let $Y \subseteq T_{\gamma+1}$ be finite, and we show that 
	$\{ q_k(\tau) : \tau \in \dom(q_k) \}$ is separated on $Y$. 
	Define a function $h : \dom(q_k) \to I_0 \cup J_k$ by letting 
	$h(\tau) = \tau$ if $\tau < \theta$, and 
	$h(\tau) = (k,\tau)$ if $\tau \ge \theta$. 
	Then $h$ is a bijection, for all $\tau \in \dom(q_k)$, $q_k(\tau) = g_{h(\tau)}$, 
	and by (3), $\{ g_i : i \in I_0 \cup J_k \}$ is separated on $Y$. 
	So by Lemma \ref{composition}, $\{ q_k(\tau) : \tau \in \dom(q_k) \}$ is separated on $Y$.
\end{proof}

\begin{lemma}[Generalized Augmentation] \label{Generalized Augmentation}
	Let $\lambda$ be a large enough regular cardinal and assume that $N$ is a countable 
	elementary substructure of $H(\lambda)$ which contains as elements $T$, $\ka$, 
	$\q$, and $\p$. 
	Let $\delta = N \cap \omega_1$. 
	Suppose the following:
	\begin{itemize}
	\item $\{ p_0,\ldots,p_{n-1} \}$ is a finite set of conditions in 
	$N \cap \p$ with top level $\gamma$;
	\item $v \in \p_\theta$ and for all $k < n$, $p_k \res \theta = v$;
	\item $B \subseteq \bigcap_{k < n} \dom(p_k)$ is finite;
	\item $z \in T_\delta$ and 
	$\sigma \in \bigcap_{k < n} \dom(p_k)$;
	\item $X \subseteq T_\delta$ is finite and 
	$X \cup \{ z \}$ has unique drop-downs to $\gamma$;
	\item for all $k < n$, 
	$\{ p_k(\tau) : \tau \in B \}$ is separated on $X \res \gamma$;
	\end{itemize}
	Then there exist $q_0,\ldots,q_{n-1}$ in $N \cap \p$ and 
	$w \in N \cap \p_\theta$, all with top level $\gamma+1$, 
	and a finite set $Y \subseteq T_\delta$ such that 
	$X \cup \{ z \} \subseteq Y$ and $Y$ has unique drop-downs to $\gamma+1$, 
	satisfying that for all $k < n$:
	\begin{enumerate}
	\item $q_k \le p_k$ and $q_k \res \theta = w$;
	\item for all $\tau \in B$, 
	$X \res \gamma$ and $X \res (\gamma+1)$ are $q_k(\tau)$-consistent;
	\item $\{ q_k(\tau) : \tau \in B \cup \{ \sigma \} \}$ is separated on $Y \res (\gamma+1)$;
	\item let $h_{k,\sigma}^+$ be the partial injective function 
	from $Y$ to $Y$ defined by letting, 
	for all $x, y \in Y$, $h_{k,\sigma}^+(x) = y$ iff 
	$q_k(\sigma)(x \res (\gamma+1)) = y \res (\gamma+1)$; 
	then $z$ is in the domain and range of $h_{k,\sigma}^+$.
	\end{enumerate}
\end{lemma}

\begin{proof}
	Apply 
	Lemma \ref{Generalized Separated Conditions are Dense} (Generalized 
	Separated Conditions are Dense) in $N$ to find 
	$q_0,\ldots,q_{n-1}$ in $N \cap \p$ and 
	$w \le_\theta v$ in $N \cap \p_\theta$, all with top level $\gamma+1$, 
	such that for all $k < n$, $q_k \le p_k$, 
	$q_k \res \theta = w$, $q_k$ is separated, and for all $\tau \in B$, 
	$X \res \gamma$ and $X \res (\gamma+1)$ are $q_{k}(\tau)$-consistent. 

	Define 
	$$
	W = ((\{ \ z \res (\gamma+1) \ \}) \ \cup \ 
	\{ \ q_k(\sigma)^m(z \res (\gamma+1)) \ : \ k < n, \ m \in \{ -1, 1 \} \ \}) \  
	\setminus \ (X \res (\gamma+1)).
	$$
	Note that by unique drop-downs of $X \cup \{ z \}$, 
	if $z \notin X$ then $z \res (\gamma+1)$ 
	is not in $X \res (\gamma+1)$. 
	So if $z$ is not in $X$, then $z \res (\gamma+1)$ is in $W$.

	Let $Y$ consist of the elements of 
	$X$ together with exactly one element of $T_\delta$ above each member of $W$, 
	and such that if $z \notin X$ then the element of $Y$ above $z \res (\gamma+1)$ is $z$. 
	Note that $Y$ has unique drop-downs to $\gamma+1$ and 
	$Y \res (\gamma+1) = (X \res (\gamma+1)) \cup W$. 
	Now define $h_{k,\sigma}^+$ for all $k < n$ as described in (4).

	Conclusions (1), (2), and (3) are clear. 
	(4): Let $k < n$. 
	Note that $q_k(\sigma)(z \res (\gamma+1))$ 
	is either in $X \res (\gamma+1)$ or in $W$. 
	As $Y \res (\gamma+1) = (X \res (\gamma+1)) \cup W$, in either case we can 
	find $c \in Y$ such that $c \res (\gamma+1) = q_k(\sigma)(z \res (\gamma+1))$. 
	Then by the definition of $h_{k,\sigma}^+$, 
	$q(\sigma)(z \res (\gamma+1)) = c \res (\gamma+1)$ implies that 
	$h_{k,\sigma}^+(z) = c$, so $z \in \dom(h_{k,\sigma}^+)$. 
	The proof that $z$ is in the range of $h_{k,\sigma}^+$ is similar.
\end{proof}

\section{Existence of Nice Conditions} \label{Existence of Nice Conditions}

Our goal for this section is to prove the following theorem.

\begin{thm}[No New Cofinal Branches] \label{No New Cofinal Branches}
	Suppose that $T$ is a free Suslin tree and $\textsf{CH}$ holds. 
	Let $\theta < \ka$ and suppose that 
	$\dot U$ is a $\p_\theta$-name for an $\omega_1$-tree. 
	Then $\p$ forces that every branch of $\dot U$ in $V^\p$ is in $V^{\p_\theta}$.
\end{thm}

We begin by proving that in Theorem \ref{No New Cofinal Branches}, 
it suffices to prove the statement under the assumption that $\ka < \omega_2$. 
So assume that the result holds when $\ka < \omega_2$, and now let 
$\theta < \ka$ be arbitrary. 
Suppose that $\dot U$ is a $\p_\theta$-name for an $\omega_1$-tree 
and $\dot b$ is a $\p_\ka$-name for a branch of $\dot U$. 
Without loss of generality, we may assume that $\dot U$ is forced 
to be a subset of $\omega_1$, and $\dot U$ and $\dot b$ are nice names.

Since $\p$ is $\omega_2$-c.c.\ by Lemma \ref{omega2 cc}, 
conditions have countable domain, 
and $\dot U$ is a nice name, 
we can find a set $X \subseteq \theta$ 
of size at most $\omega_1$ such that $\dot U$ is a $\p_X$-name. 
Similarly, we can find a set $Y \subseteq \ka$ of size at most $\omega_1$ such that 
$X \subseteq Y$ and $\dot b$ is a $\p_Y$-name. 
Then $\dot U$ is a $\p_{Y \cap \theta}$-name. 
	
Consider a generic filter $G$ on $\p$ and let $H = G \cap \p_\theta$. 
Let $U = \dot U^H$ and $b = \dot b^G$. 
We show that $b \in V[H]$. 
By the choice of $Y$ and the names, 
$U \in V[H \cap \p_{Y \cap \theta}]$ and $b \in V[G \cap \p_Y]$. 
Let $\ka_0$ be the order type of $Y$ and let 
$\theta_0$ be the order type of $Y \cap \theta$. 
By standard arguments, there exists an 
isomorphism $\varphi : \p_{Y} \to \p_{\kappa_0}$ such that 
$\varphi \res \p_{Y \cap \theta}$ is an isomorphism of $\p_{Y \cap \theta}$ 
onto $\p_{\theta_0}$. 
Let $\bar G = \varphi[G \cap \p_Y]$ 
and $\bar H = \varphi[H \cap \p_{Y \cap \theta}]$. 
Then $\bar G$ is a generic filter on $\p_{\ka_0}$, 
$\bar H = \bar G \cap \p_{\theta_0}$ is a generic filter on $\p_{\theta_0}$, 
$U \in V[\bar H]$, and $b \in V[\bar G]$. 
Since $Y$ has cardinality at most $\omega_1$, $\theta_0 < \ka_0 < \omega_2$. 
So $b \in V[\bar H]$. 
But $V[\bar H] \subseteq V[H]$.

For the remainder of this section assume that $\theta < \ka < \omega_2$, 
$\dot U$ is a $\p_\theta$-name for an $\omega_1$-tree, and 
$\dot b$ is a $\p$-name for a branch of $\dot U$. 
We prove that $\p$ forces that $\dot b$ is in $V^{\p_\theta}$. 
Without loss of generality assume that the underlying set of $\dot U$ 
is forced to equal $\omega_1$, and in fact that for any $\gamma < \omega_1$, 
the elements of $\dot U_\gamma$ are ordinals in the interval 
$[\omega \cdot \gamma, \omega \cdot (\gamma+1))$.

Fix a large enough regular cardinal $\lambda$ and a well-ordering $\unlhd$ of $H(\lambda)$. 
Define a set $N$ to be a \emph{suitable model} if 
it is a countable elementary substructure of $(H(\lambda),\in,\unlhd)$ which 
contains as members the objects $T$, $\ka$, $\q$, $\p$, $\theta$, $\dot U$, and $\dot b$.

\begin{definition}[Nice Conditions] \label{Nice Conditions}
	Let $N$ be a suitable model, $\delta = N \cap \omega_1$, and $\alpha < \delta$.  
	Suppose that $p \in N \cap \p$ has top level $\alpha$ and 
	$A \subseteq \dom(p)$ is finite. 
	Let $\vec b$ be a tuple consisting of distinct elements of $T_\alpha$ 
	and assume that $\{ p(\tau) : \tau \in A \}$ is separated on $\vec b$. 
	A condition $v \in \p_\theta$ is said to be 
	\emph{$N$-nice for $p$, $\vec b$, and $A$} if the following statements hold:
	\begin{enumerate}
	\item $v \le p \res \theta$, $v$ has top level $\delta$, 
	and $v$ decides $\dot U \res \delta$.
	\item For all $q \in N \cap \p$ such that $q \le p$ and $v \le q \res \theta$, 
	there exists some $r \le q$ with top level $\delta$ such that 
	$r \res \theta = v$, $N \cap \ka \subseteq \dom(r)$, 
	$r$ is separated, and $r$ decides $\dot b \cap \delta$. 
	\item Suppose that $\vec a^0$ and $\vec a^1$ are distinct tuples above $\vec b$ 
	with height $\delta$, and $q_0, q_1 \le p$ have top level $\delta$ and satisfy: 
	$q_0 \res \theta = q_1 \res \theta = v$, 
	$N \cap \ka \subseteq \dom(q_0) \cap \dom(q_1)$, 
	$q_0$ and $q_1$ are separated, $q_0$ and $q_1$ 
	decide $\dot b \cap \delta$, and for all $\tau \in A$ and $j < 2$, 
	$\vec b$ and $\vec a^{j}$ are $q_j(\tau)$-consistent. 
	Then there exist $q_0^*$ and $q_1^*$ satisfying the same properties listed above 
	for $q_0$ and $q_1$, there exists $r \le p$ with top level $\delta$, 
	and there exist disjoint tuples $\vec e^0$ and $\vec e^1$ 
	above $\vec b$ with height $\delta$ satisfying:
	$r \res \theta = v$, $N \cap \ka \subseteq \dom(r)$, 
	$r$ is separated, 
	$r$ decides $\dot b \cap \delta$, 
	and for all $\tau \in A$ and $j < 2$, 
	$\vec b$ and $\vec e^j$ are $q_j^*(\tau)$-consistent 
	and $r(\tau)$-consistent.
	\end{enumerate}
	\end{definition}

\begin{lemma} \label{niceness persistent}
	Let $M$ and $N$ be suitable models, $N \subseteq M$, 
	$\delta = N \cap \omega_1 = M \cap \omega_1$, $\alpha < \delta$, 
	$p \in N \cap \p$ has top level $\alpha$, and $A \subseteq \dom(p)$ is finite. 
	Let $\vec b$ be a tuple consisting of distinct elements of $T_\alpha$ 
	and assume that $\{ p(\tau) : \tau \in A \}$ is separated on $\vec b$. 
	If $v$ is $M$-nice for $p$, $\vec b$, and $A$, then $v$ is 
	$N$-nice for $p$, $\vec b$, and $A$.
\end{lemma}

\begin{proof}
	We verify that $v$ satisfies properties (1)-(3) of 
	Definition \ref{Nice Conditions} (Nice Conditions) for $N$. 
	(1) is immediate and (2) follows easily from the fact that $N \subseteq M$. 
	(3) is easy to check once we show that $N \cap \ka = M \cap \ka$, 
	which is where we use the assumption that $\ka < \omega_2$. 
	Let $g : \omega_1 \to \ka$ be the $\unlhd$-minimum 
	function in $H(\lambda)$ which is a surjection of $\omega_1$ onto $\ka$. 
	As $\ka \in N \cap M$, it follows that $g \in N \cap M$ by elementarity. 
	Also, by elementarity, $N \cap \ka = g[N \cap \omega_1] = g[M \cap \omega_1] = M \cap \ka$.
\end{proof}

We need the following technical lemma in order to prove the existence of nice conditions. 

\begin{lemma} \label{very technical lemma}
	Assume the following:
	\begin{enumerate}
	\item $\alpha < \rho < \omega_1$;
	\item $p \in \p$ has top level $\alpha$, 
	$v \in \p_\theta$ has top level $\rho$, 
	and $v \le_\theta p \res \theta$;
	\item $A \subseteq \dom(p)$ is finite, 
	$B \subseteq \ka$ is finite, 
	$A \subseteq B$, and $B \cap \theta \subseteq \dom(v)$;
	\item $\vec b = (b_0,\ldots,b_{n-1})$ consists of distinct elements of $T_\alpha$;
	\item $\{ p(\tau) : \tau \in A \}$ is separated on $\vec b$;
	\item $Y$ is a finite subset of $T_{\rho+1}$ with unique drop-downs to $\rho$;
	\item $\vec d^0$ and $\vec d^1$ are tuples above $\vec b$ with height $\rho+1$;
	\item for all $\tau \in A \cap \theta$, 
	$\vec b$ and $\vec d^0 \res \rho$ are $v(\tau)$-consistent and 
	$\vec b$ and $\vec d^1 \res \rho$ are $v(\tau)$-consistent;
	\item $\{ v(\tau) : \tau \in B \cap \theta \}$ is separated on $Y \res \rho$;
	\item letting $D^0$ be the set of elements of $\vec d^0$ and 
	$D^1$ the set of elements of $\vec d^1$, we have that 
	$D^0 \res (\alpha+1)$, $D^1 \res (\alpha+1)$, and $Y \res (\alpha+1)$ 
	are pairwise disjoint;
	\item for all $\tau \in B \cap \theta$ and $m \in \{ -1, 1 \}$:
	\begin{enumerate}
	\item[-] for all $x \in Y \res \rho$, 
	$v(\tau)^m(x) \notin (D^0 \cup D^1) \res \rho$;
	\item[-] for all $x \in D^0 \res \rho$, $v(\tau)^m(x) \notin D^1 \res \rho$;
	\item[-] for all $x \in D^1 \res \rho$, $v(\tau)^m(x) \notin D^0 \res \rho$.
	\end{enumerate}
	\end{enumerate}
	Then there exist $q \in \p$ and $w \in \p_\theta$, both with 
	top level $\rho+1$, satisfying:
	\begin{enumerate}
	\item[(I)] $q \le p$, $w \le_\theta v$, $q \res \theta = w$, and $B \subseteq \dom(q)$;
	\item[(II)] for all $\tau \in A$, $\vec b$ and $\vec d^0$ are $q(\tau)$-consistent 
	and $\vec b$ and $\vec d^1$ are $q(\tau)$-consistent;
	\item[(III)] for all $\tau \in B \cap \theta$, 
	$(Y \cup D^0 \cup D^1) \res \rho$ and $Y \cup D^0 \cup D^1$ are $w(\tau)$-consistent;
	\item[(IV)] $\{ q(\tau) : \tau \in B \}$ is separated on $Y \cup D^0 \cup D^1$.
	\end{enumerate}
\end{lemma}

\begin{proof}
	Let us define a condition 
	$\bar{p} \le p$ with top level $\alpha$ such that 
	$\dom(\bar{p}) = \dom(p) \cup B$. 
	Let $\bar{p} \res \dom(p) = p$. 
	For all $\tau \in B \setminus \dom(p)$, 
	define $\bar{p}(\tau) = v(\tau) \res (\alpha+1)$ 
	if $\tau < \theta$, and let $\bar{p}(\tau)$ be the identity function on 
	$T \res (\alpha+1)$ if $\tau \ge \theta$. 
	Note that $v \le_\theta \bar{p} \res \theta$.

	Define $\bar{D}^0 = D^0 \res (\alpha+1)$, $\bar{D}^1 = D^1 \res (\alpha+1)$, 
	and $\bar{Y} = Y \res (\alpha+1)$. 
	Applying Lemma \ref{elaborate extending by one 2}, fix a family 
	$\{ g_\tau : \tau \in \dom(\bar{p}) \setminus \theta \}$ 
	of automorphisms of $T \res (\alpha+2)$ satisfying:
	\begin{enumerate}
	\item[(a)] for all $\tau \in \dom(\bar{p}) \setminus \theta$, 
	$\bar{p}(\tau) \subseteq g_\tau$;
	\item[(b)] for all $\tau \in A \setminus \theta$, 
	$\vec b$ and $\vec d^0 \res (\alpha+1)$ are $g_\tau$-consistent 
	and $\vec b$ and $\vec d^1 \res (\alpha+1)$ are $g_\tau$-consistent;
	\item[(c)] for all $x \in \bar{D}^0$: if $\tau \in A \setminus \theta$, then 
	$g_\tau(x)$ and $g_\tau^{-1}(x)$ are not in $\bar{Y} \cup \bar{D}^1$, 
	and if $\tau \in (B \setminus \theta) \setminus A$, then 
	$g_\tau(x)$ and $g_\tau^{-1}(x)$ are not in $\bar{D}^0 \cup \bar{D}^1 \cup \bar{Y}$;
	\item[(d)] for all $x \in \bar{D}^1$: if $\tau \in A \setminus \theta$, 
	then $g_\tau(x)$ and $g_\tau^{-1}(x)$ are not in $\bar{Y} \cup \bar{D}^0$, 
	and if $\tau \in (B \setminus \theta) \setminus A$, 
	then $g_\tau(x)$ and $g_\tau^{-1}(x)$ are not in 
	$\bar{D}^0 \cup \bar{D}^1 \cup \bar{Y}$;
	\item[(e)] for all $x \in \bar{Y}$ and $\tau \in B \setminus \theta$, 
	$g_\tau(x)$ and $g_\tau^{-1}(x)$ are not in $\bar{D}^0 \cup \bar{D}^1 \cup \bar{Y}$.
	\end{enumerate}
	Define $s$ with domain equal to $\dom(\bar{p}) \setminus \theta$ so that for all 
	$\tau \in \dom(s)$, $s(\tau) = g_\tau$. 
	Clearly, $s$ is a condition with 
	top level $\alpha+1$ and $s \le \bar{p} \res [\theta,\ka)$.

	By (9), we can fix an injective tuple $\vec y$ which enumerates $Y$ such that 
	$\{ v(\tau) : \tau \in B \cap \theta \}$ is separated on $\vec y \res \rho$. 
	Let $\vec x$ be the tuple of height $\rho+1$ 
	which consists of the concentation of the tuples $\vec y$, $\vec d^0$, and $\vec d^1$, 
	in that order. 
	So $\vec x$ enumerates $Y \cup D^0 \cup D^1$.
	
	We claim that $\{ s(\tau) : \tau \in B \setminus \theta \}$ 
	is separated on $\vec x \res (\alpha+1)$. 
	Since $\vec d^0 \res (\alpha+1)$ and $\vec d^1 \res (\alpha+1)$ 
	are above $\vec b$, $s \le p \res [\theta,\ka)$, and 
	$\{ p(\tau) : \tau \in A \}$ is separated on $\vec b$, it follows 
	by Lemma \ref{Persistence 2} (Persistence) that 
	$\{ s(\tau) : \tau \in A \setminus \theta \}$ is separated on 
	both $\vec d^0 \res (\alpha+1)$ and $\vec d^1 \res (\alpha+1)$. 
	By properties (c), (d), and (e) above, if a relation of the form 
	$s(\tau)^m(x) = y$ holds, where $\tau \in B \setminus \theta$, $m \in \{ -1, 1 \}$, 
	and $x, y \in \bar{D}^0 \cup \bar{D}^1 \cup \bar{Y}$, 
	then it must be the case that $\tau \in A \setminus \theta$ 
	and either $x, y \in \bar{D}^0$ or $x, y \in \bar{D}^1$. 
	Based on this information, it easily follows that 
	$\{ s(\tau) : \tau \in B \setminus \theta \}$ is separated on $\vec x \res (\alpha+1)$.

	Apply Proposition \ref{Extension 2} (Extension) 
	to find a condition $z \le s$ with top level $\rho+1$ 
	with the same domain as $s$ such that for all 
	$\tau \in B \setminus \theta$, $\vec x \res (\alpha+1)$ and $\vec x$ 
	are $z(\tau)$-consistent. 
	By Lemma \ref{Persistence 2} (Persistence), 
	$\{ z(\tau) : \tau \in B \setminus \theta \}$ 
	is separated on $\vec x$. 
	
	Now apply Lemma \ref{elaborate extending by one} to 
	find a family $\{ h_\tau : \tau \in \dom(v) \}$ 
	of automorphisms of $T \res (\rho+2)$ satisfying:
	\begin{enumerate}
	\item[(f)] for all $\tau \in \dom(v)$, 
	$v(\tau) \subseteq h_\tau$;
	\item[(g)] for all $\tau \in B \cap \theta$, 
	$Y \res \rho$ and $Y$ are $h_\tau$-consistent;
	\item[(h)] for all $\tau \in A \cap \theta$, $\vec d^0 \res \rho$ and $\vec d^0$ 
	are $h_\tau$-consistent and $\vec d^1 \res \rho$ and $\vec d^1$ 
	are $h_\tau$-consistent;
	\item[(i)] for all $\tau \in (B \cap \theta) \setminus A$, $m \in \{ -1, 1 \}$, 
	and $x \in D^0 \cup D^1$, $h_\tau^m(x) \notin Y \cup D^0 \cup D^1$.
	\end{enumerate}
	Define $w$ with the same domain as $v$ so that for all $\tau \in \dom(v)$, 
	$w(\tau) = h_\tau$. 
	Clearly, $w \in \p_\theta$, $w \le_\theta v$, and $w$ has top level $\rho+1$. 
	Since $\{ v(\tau) : \tau \in B \cap \theta \}$ is separated on $\vec y \res \rho$, 
	by Lemma \ref{Persistence 2} (Persistence), 
	$\{ w(\tau) : \tau \in B \cap \theta \}$ is separated on $\vec y$. 

	Finally, let $q = w \cup z$. 
	Then $\{ q(\tau) : \tau \in B \cap \theta \}$ is separated on $\vec y$ 
	and $\{ q(\tau) : \tau \in B \setminus \theta \}$ is separated on $\vec x$.

	Let us prove conclusions (I)-(IV). 
	(I) is clear. 
	(II) Let $\tau \in A$. 
	If $\tau \in A \cap \theta$, then (II) holds by (8) and (h). 
	If $\tau \in A \setminus \theta$, then (II) holds by (b) and the choice of $z$. 

	(III) Statement (g) implies that for all $\tau \in B \cap \theta$, 
	$Y \res \rho$ and $Y$ are $w(\tau)$-consistent. 
	Statement (i) implies that for all $\tau \in (B \cap \theta) \setminus A$, 
	if $x,y \in (Y \cup D^0 \cup D^1) \res \rho$ and 
	$w(\tau)(x) = y$, then $x, y \in Y$. 
	By (II), for all $\tau \in A \cap \theta$, 
	$\vec b$ and $\vec d^0$ are $w(\tau)$-consistent 
	and $\vec b$ and $\vec d^1$ are $w(\tau)$-consistent. 
	This implies that for all $\tau \in A \cap \theta$, 
	$D^0 \res \rho$ and $D^0$ are $w(\tau)$-consistent and 
	$D^1 \res \rho$ and $D^1$ are $w(\tau)$-consistent. 
	By (11), if $\tau \in A \cap \theta$, $x, y \in (Y \cup D^0 \cup D^1) \res \rho$, 
	and $w(\tau)(x) = y$, 
	then either $x, y \in Y \res \rho$, $x, y \in D^0 \res \rho$, or $x, y \in D^1 \res \rho$. 
	Altogether this information easily implies (III).
	
	(IV) We claim that $\{ q(\tau) : \tau \in B \}$ is separated on $\vec x$. 
	First, let us show that for all $\tau \in B$, $q(\tau)$ has no fixed points in $\vec x$.
	Since $\{ q(\tau) : \tau \in B \setminus \theta \}$ is separated on $\vec x$, 
	for all $\tau \in B \setminus \theta$, $q(\tau)$ has no fixed points in $\vec x$. 
	So it suffices to consider $\tau \in B \cap \theta$. 
	As $\{ q(\tau) : \tau \in B \cap \theta \}$ is separated on $\vec y$, for all 
	$\tau \in B \cap \theta$, $q(\tau)$ has no fixed points in $\vec y$. 
	Because $\{ p(\tau) : \tau \in A \}$ is separated on $\vec b$, 
	by Lemma \ref{Persistence 2} (Persistence), 
	$\{ q(\tau) : \tau \in A \}$ is separated on $\vec d^0$ 
	and separated on $\vec d^1$. 
	Hence, for all $\tau \in A \cap \theta$, $q(\tau)$ 
	has no fixed points in $\vec d^0$ or in $\vec d^1$. 
	Finally, consider $\tau \in (B \setminus A) \cap \theta$ and we show that 
	$q(\tau)$ has no fixed points in $\vec d^0$ or $\vec d^1$. 
	But this follows from (i).
	
	Consider a relation of the form $q(\tau)^m(x) = y$, where 
	$\tau \in B$, $m \in \{ -1, 1 \}$, $x, y \in Y \cup D^0 \cup D^1$, and 
	$y$ appears earlier in the ordering of $\vec x$ than $x$. 
	First, assume that $\tau < \theta$. 
	By (11), the only possibilities are that 
	$x, y \in Y$, $x, y \in D^0$, or $x, y \in D^1$. 
	By (i), if $\tau \notin A$ then $x, y \in Y$. 
	Secondly, assume that $\tau \ge \theta$. 
	By (c), (d), and (e), $\tau \in A$ and either $x, y \in D^0$ or $x, y \in D^1$. 
	So altogether, if one of $x$ or $y$ is in $Y$, then $\tau < \theta$ and $x$ and $y$ 
	are both in $Y$. 
	Since $\{ q(\tau) : \tau \in B \cap \theta \}$ is separated on $\vec y$, 
	there is at most one such relation which holds in this case. 
	Now suppose that it is not the case that one of $x$ or $y$ is in $Y$. 
	Then $\tau \in A$, and either $x, y \in D^0$ or $x, y \in D^1$. 
	But since $\{ q(\tau) : \tau \in A \}$ is separated on $\vec d^0$ 
	and separated on $\vec d^1$, 
	there is at most one such relation which holds in this case as well.
\end{proof}

\begin{proposition}[Existence of Nice Conditions] \label{Existence of Nice Conditions prop}
	Suppose that $T$ is a free Suslin tree. 
	Let $N$ be a suitable model and let $\delta = N \cap \omega_1$. 
	Suppose that $p \in N \cap \p$ has top level $\alpha < \delta$ and 
	$p$ forces in $\p$ that $\dot b$ is a cofinal branch of $\dot U$ 
	which is not in $V^{\p_\theta}$. 
	Assume that $A \subseteq \dom(p)$ is finite, 
	$\vec b$ consists of distinct elements of $T_\alpha$, 
	and $\{ p(\tau) : \tau \in A \}$ is separated on $\vec b$. 
	Then for any $q \le p$ in $N \cap \p$, there exists some 
	$v \in \p_\theta$ which is $N$-nice for $p$, $\vec b$, and $A$ 
	such that $v \le_\theta q \res \theta$.
\end{proposition}

\begin{proof}
	Let $q \le p$ in $N \cap \p$. 
	We prove that there exists some 
	$v \in \p_\theta$ which is $N$-nice for $p$, $\vec b$, and $A$ 
	such that $v \le_\theta q \res \theta$. 
	Without loss of generality, assume that the top level of $q$ is greater than $\alpha$.

	To help with the construction of $v$, we fix the following objects: 
	\begin{itemize}
	\item[-] an enumeration 
	$\langle (\vec a^{n,0},\vec a^{n,1}) : n < \omega \rangle$ 
	of all distinct pairs of tuples above $\vec b$ 
	with height $\delta$;
	\item[-] a non-decreasing sequence $\langle \gamma_n : n < \omega \rangle$ of ordinals 
	cofinal in $\delta$ with $\gamma_0 = \gamma_1 = \alpha$;
	\item[-] an enumeration $\langle z_n : n < \omega \rangle$ of $T_\delta$;
	\item[-] an enumeration $\langle s_n : n < \omega \rangle$ of $N \cap \p$; 
	\item[-] a surjection $g : \omega \to 
	4 \times \omega \times (N \cap \ka)$ 
	such that every element of the codomain 
	has an infinite preimage;
	\item[-] an enumeration $\langle D_n : n < \omega \rangle$ of all 
	of the dense open subsets of $\p$ which lie in $N$.
	\end{itemize}

	As in previous proofs, 
	our construction of the nice condition $v$ takes place 
	over $\omega$-many stages, with the function $g$ being used for bookkeeping. 
	In order to satisfy properties (2) and (3) of 
	Definition \ref{Nice Conditions} (Nice Conditions), the construction involves building 
	not only $v$, but also infinitely many total master conditions $r$ with $r \res \theta = v$. 
	At any given stage $n$, we define an approximation $v_n$ of $v$ 
	together with finitely many approximations of the total master conditions. 
	The first coordinate $n_0$ of the value $g(n)$ splits the construction into four cases. 
	When $n_0 = 0$, we handle a case of Definition \ref{Nice Conditions}(2), and when 
	$n_0 = 3$, we handle a case of Definition \ref{Nice Conditions}(3). 
	When $n_0$ is $1$ or $2$, we take the usual steps 
	for building total master conditions, 
	namely, meeting dense sets in the first case 
	and applying augmentation in the second case.

	More specifically, 
	we define by induction the following objects in $\omega$-many steps:

	\begin{itemize}

	\item a subset-increasing sequence $\langle X_n : n < \omega \rangle$ 
	of finite subsets of $T_\delta$ with union equal to $T_\delta$;

	\item a subset-increasing sequence $\langle A_n : n < \omega \rangle$ 
	of finite subsets of $N \cap \ka$ with union equal to $N \cap \ka$;

	\item a non-decreasing sequence $\langle \delta_n : n < \omega \rangle$ 
	of ordinals cofinal in $\delta$, where each 
	$X_n$ has unique drop-downs to $\delta_n$;

	\item a decreasing sequence $\langle v_n : n < \omega \rangle$ of conditions 
	in $N \cap \p_\theta$;

	\item a sequence $\langle l_m : m < \omega \rangle$ of natural numbers;

	\item a family 
	$\{ r^{m,l}_n : m < \omega, \ l < l_m, \ m < n < \omega \}$ 
	of conditions in $N \cap \p$ below $p$ 
	such that for each $m < \omega$ and $l < l_m$, 
	$\langle r^{m,l}_n : m < n < \omega \rangle$ is a descending sequence;
		
	\item for each $m < \omega$ such that $l_m = 3$, 
	a pair $\vec e^{m,0}$ and $\vec e^{m,1}$ of tuples above $\vec b$ with height $\delta$ 
	such that $\vec e^{m,0} \res (\alpha+1)$ and $\vec e^{m,1} \res (\alpha+1)$ 
	are disjoint, and the elements of $\vec e^{m,0}$ and $\vec e^{m,1}$ 
	are in $X_{m+1}$;

	\item for all $m < \omega$, $l < l_m$, $n > m$, and $\tau \in A_n$, 
	an injective partial function $h^{m,l}_{n,\tau}$ from $X_n$ to $X_n$.

	\end{itemize}

	At stage $n$, we define $X_n$, $A_n$, $\delta_n$, $v_n$, and when $n > 0$, 
	$l_{n-1}$, $r^{m,l}_n$ and $h^{m,l}_{n,\tau}$ 
	for all $m < n$, $l < l_m$, and $\tau \in A_n$, and 
	when $l_{n-1} = 3$, $\vec c^{n-1,0}$ and $\vec c^{n-1,1}$.

	In addition to the properties listed above, we maintain the following 
	inductive hypotheses for all $m < n < \omega$ and $l < l_m$:
	\begin{enumerate}

	\item $r^{m,l}_n$ and $v_n$ have top level $\delta_n$ and 
	$r^{m,l}_n \res \theta = v_n$;

	\item for all $\tau \in A_{n}$, $X_{n} \res \delta_{n}$ 
	and $X_{n} \res \delta_{n+1}$ are $r^{m,l}_{n+1}(\tau)$-consistent;

	\item $A_n \subseteq \dom(r^{m,l}_n)$ and 
	$\{ r^{m,l}_{n}(\tau) : \tau \in A_n \}$ is separated 
	on $X_n \res \delta_n$;

	\item suppose that 
	$g(n) = (n_0,n_1,\sigma)$, where $n_0 = 0$, 
	$s_{n_1}$ has top level less than $\delta_n$, 
	$s_{n_1} \le p$, and $v_n \le_\theta s_{n_1} \res \theta$; 
	then $l_n = 1$ and $r^{n,0}_{n+1} \le s_{n_1}$;

	\item if $g(n) = (n_0,n_1,\sigma)$, where $n_0 = 1$, 
	then $r^{m,l}_{n+1} \in \bigcap_{k < n} D_k$ and 
	$\sigma \in \dom(r^{m,l}_{n+1})$;
	
	\item assuming that $g(n) = (n_0,n_1,\sigma)$, where $n_0 = 2$, and 
	$\sigma \in \bigcap \{ \dom(r^{m,l}_{n}) : m < n, \ l < l_m \}$, 	
	then $\sigma \in A_{n+1}$ and for all $m < n$ and $l < l_m$, 
	$z_{n_1}$ is in the domain and range of $h^{m,l}_{n+1,\sigma}$;

	\item for all $\tau \in A_n$ and $x, y \in X_n$, 
	$$
	h^{m,l}_{n,\tau}(x) = y \ \Longleftrightarrow \ 
	r^{m,l}_{n}(\tau)(x \res \delta_n) = y \res \delta_n.
	$$
	\end{enumerate}

	\underline{Stage $0$}: Let $X_0 = \emptyset$ and $A_0 = A$. 
	Let $\delta_{0}$ be the top level of $q$ and let $v_0 = q \res \theta$. 
	Note that by our assumption about $q$, $\delta_0 > \alpha$.
	
	\underline{Stage 1}: Let $l_0 = 1$. 
	Let $X_1 = X_0 = \emptyset$, $A_1 = A_0$, 
	$\delta_1 = \delta_0$, $r^{0,0}_1 = q$, and $v_1 = q \res \theta = v_0$. 
	Define $h^{0,0}_{1,\tau} = \emptyset$ for all $\tau \in A_1$. 
	
	\underline{Stage $n+1$} ($n > 0$): 
	Assume that stage $n$ is complete, where $n < \omega$ is positive. 
	In particular, we have defined the following objects which we assume 
	satisfy all of the required properties: 
	$X_n$, $A_n$, $\delta_n$, $v_n$, $l_m$ for all $m < n$, 
	and $r^{m,l}_n$ and $h^{m,l}_{n,\tau}$ for all $m < n$, $l < l_m$, and $\tau \in A_n$. 
	Let $g(n) = (n_0,n_1,\sigma)$.

	\underline{Case a:} $n_0 = 0$. 
	Consider $s_{n_1}$, which is in $N \cap \p$, and let $\gamma$ be 
	the top level of $s_{n_1}$. 
	We only take action in the case that $\gamma < \delta_n$, 
	$s_{n_1} \le p$, and $v_n \le s_{n_1} \res \theta$. 
	If not, then let $l_n = 0$, 
	$X_{n+1} = X_n$, $A_{n+1} = A_n$, $\delta_{n+1} = \delta_n$, 
	$v_{n+1} = v_n$, $r^{m,l}_{n+1} = r^{m,l}_n$ and 
	$h^{m,l}_{n+1,\tau} = h^{m,l}_{n,\tau}$ for all $m < n$, $l < l_m$, 
	and $\tau \in A_n$. 

	Otherwise, let $l_n = 1$. 
	Without loss of generality, we may assume that $A_n \subseteq \dom(s_{n_1})$, 
	for otherwise we can easily extend $s_{n_1}$ to have this property without 
	increasing its top level or changing the fact that $v \le_\theta s_{n_1} \res \theta$. 
	Apply Lemma \ref{Simple Generalized Extension} (Simple Generalized Extension) 
	to find some $s \le s_{n_1}$ in $N$ 
	with top level $\delta_n$ such that $s \res \theta = v_n$. 
	Applying 
	Lemma \ref{Separated Conditions are Dense} (Separated Conditions Are Dense) and 
	using the fact 
	that $\{ v_n(\tau) : \tau \in A_n \cap \theta \}$ is separated on $X_n \res \delta_n$ 
	(by inductive hypothesis (3) and Lemma \ref{subset}), 
	find a separated condition 
	$r^{n,0}_{n+1} \le s$ with top level $\delta_n+1$ 
	such that for all $\tau \in A_n \cap \theta$, $X_n \res \delta_n$ and 
	$X_n \res (\delta_{n}+1)$ are $r^{n,0}_{n+1}(\tau)$-consistent. 
	Let $v_{n+1} = r^{n,0}_{n+1} \res \theta$. 
	Then in particular, for all $\tau \in A_n \cap \theta$, $X_n \res \delta_n$ 
	and $X_n \res (\delta_n+1)$ are $v_{n+1}(\tau)$-consistent. 
	Now apply 
	Lemma \ref{Generalized Extension} (Generalized Extension) to find for each 
	$m < n$ and $l < l_m$ a condition $r^{m,l}_{n+1} \le r^{m,l}_{n}$ 
	with top level $\delta_{n}+1$ such that $r^{m,l}_{n+1} \res \theta = v_{n+1}$ 
	and for all $\tau \in A_{n}$, $X_n \res \delta_n$ and $X_n \res (\delta_{n}+1)$ 
	are $r^{m,l}_{n+1}(\tau)$-consistent.
	
	Define $X_{n+1} = X_n$, $A_{n+1} = A_n$, and $\delta_{n+1} = \delta_n+1$. 	
	Define $h^{m,l}_{n+1,\tau}$ for all $m \le n$, $l < l_m$, and $\tau \in A_{n+1}$ 
	as described in inductive hypothesis (7). 
	It is routine to check that the inductive hypotheses hold.
	
	\underline{Case b:} $n_0 = 1$. 
	Let $l_n = 0$, $X_{n+1} = X_{n}$, and 
	$A_{n+1} = A_n$. 
	Let $D$ be the set of conditions $r$ in $\bigcap_{k < n} D_k$ satisfying 
	that $\sigma \in \dom(r)$ and the top level of $r$ is at least $\gamma_{n+1}$. 
	Then $D$ is dense open in $\p$, and $D \in N$ by elementarity. 
	Apply 
	Lemma \ref{Generalized Consistent Extensions Into Dense Sets} (Generalized 
	Consistent Extensions Into Dense Sets) 
	to find an ordinal $\delta_{n+1} < \delta$, 
	a condition $v_{n+1} \le_\theta v_n$ in $N$, and for each 
	$m < n$ and $l < l_m$, a condition 
	$r^{m,l}_{n+1} \le r^{m,l}_n$ in $N \cap D$ with top level $\delta_{n+1}$ 
	satisfying that $r^{m,l}_{n+1} \res \theta = v_{n+1}$, 
	and for all $\tau \in A_n$, $X_n \res \delta_n$ and $X_n \res \delta_{n+1}$ 
	are $r^{m,l}_{n+1}(\tau)$-consistent. 
	Define $h^{m,l}_{n+1,\tau}$ for all $\tau \in A_n$ as 
	described in inductive hypothesis (7). 
	It is routine to check that the inductive hypotheses hold.

	\underline{Case c:} $n_0 = 2$. 
	If $\sigma \notin \bigcap \{ \dom(r^{m,l}_n) : m < n, \ l < l_m \}$, 
	then let $X_{n+1} = X_n$, $A_{n+1} = A_n$, $\delta_{n+1} = \delta_n$, 
	$v_{n+1} = v_n$, $l_n = 0$, and $r^{m,l}_{n+1} = r^{m,l}_n$ and 
	$h^{m,l}_{n+1,\tau} = h^{m,l}_{n,\tau}$ for all $m < n$, $l < l_m$, and $\tau \in A_n$.

	Otherwise, fix $\gamma < \delta$ large enough so that $X_n \cup \{ z_{n_1} \}$ 
	has unique drop-downs to $\gamma$. 
	Apply 
	Lemma \ref{Generalized Extension} (Generalized Extension) to 
	find some $\bar{v}_n \in N \cap \p_\theta$ 
	with top level $\gamma$, and for each $m < n$ and $l < l_m$ find a condition 
	$\bar{r}^{m,l}_n \le r^{m,l}_n$ in $N \cap \p$ with top level $\gamma$ such that 
	$\bar{r}^{m,l}_n \res \theta = \bar{v}_n$ and for all $\tau \in A_n$, 
	$X_n \res \delta_n$ and $X_n \res \gamma$ are $\bar{r}^{m,l}_n(\tau)$-consistent. 
	By Lemma \ref{Persistence for Sets 2} (Persistence for Sets), for all $m < n$ and $l < l_m$, 
	$\{ \bar{r}^{m,l}_n(\tau) : \tau \in A_n \}$ is separated on $X \res \gamma$.

	Now apply 
	Lemma \ref{Generalized Augmentation} (Generalized Augmentation) 
	to find for each $m < n$ and $l < l_m$ a 
	condition $r^{m,l}$ in $N \cap \p$ with top level $\gamma+1$, 
	a condition $v_{n+1} \in N \cap \p_\theta$ with top level $\gamma+1$, 
	and a finite set $X_{n+1} \subseteq T_\delta$ such that 
	$X_n \cup \{ z_{n_1} \} \subseteq X_{n+1}$ and 
	$X_{n+1}$ has unique drop-downs to $\gamma+1$, 
	satisfying that for all $m < n$ and $l < l_m$:
	\begin{itemize}
	\item $r^{m,l}_{n+1} \le \bar{r}^{m,l}_{n}$ and 
	$r^{m,l}_{n+1} \res \theta = v_{n+1}$;
	\item for all $\tau \in A_n$, 
	$X_n \res \gamma$ and $X_n \res (\gamma+1)$ are $r^{m,l}_{n+1}(\tau)$-consistent;
	\item $\{ r^{m,l}_{n+1}(\tau) : \tau \in A_n \cup \{ \sigma \} \}$ 
	is separated on $X_{n+1} \res (\gamma+1)$;
	\item let $h_{m,l,\sigma}^+$ be the partial injective function 
	from $X_{n+1}$ to $X_{n+1}$ defined by letting, 
	for all $x, y \in X_{n+1}$, $h_{m,l,\sigma}^+(x) = y$ iff 
	$r^{m,l}_{n+1}(\sigma)(x \res (\gamma+1)) = y \res (\gamma+1)$; 
	then $z_{n_1}$ is in the domain and range of $h_{m,l,\sigma}^+$.
	\end{itemize}

	Define $X_{n+1} = X_n \cup \{ z_{n_1} \}$, $A_{n+1} = A_n \cup \{ \sigma \}$, 
	$\delta_{n+1} = \gamma+1$, and $l_n = 0$. 
	For all $m < n$, $l < l_m$, and $\tau \in A_{n+1}$, 
	define $h^{m,l}_{n+1,\tau}$ as described in inductive hypothesis (7). 
	Note that $h^{m,l}_{n+1,\sigma} = h_{m,l,\sigma}^+$, and therefore $z_{n_1}$ 
	is in the domain and range of $h^{m,l}_{n+1,\sigma}$. 
	The inductive hypotheses are straightforward to check.

	\underline{Case d:} $n_0 = 3$. 
	Let us consider $\vec a^{n_1,0}$ and $\vec a^{n_1,1}$. 
	For simplicity in notation, write $\vec a^{0}$ for $\vec a^{n_1,0}$ 
	and $\vec a^{1}$ for $\vec a^{n_1,1}$. 
	We only take action when the elements of the tuples $\vec a^0$ and $\vec a^1$ 
	are in $X_n$ and for all $\tau \in A \cap \theta$ and $j < 2$, 
	$\vec b$ and $\vec a^{j} \res \delta_n$ are $v_n(\tau)$-consistent. 
	If not, then let $l_n = 0$, $X_{n+1} = X_n$, $A_{n+1} = A_n$, 
	$\delta_{n+1} = \delta_n$, $v_{n+1} = v_n$, 
	$r^{m,l}_{n+1} = r^{m,l}_{n}$ and 
	$h^{m,l}_{n+1,\tau} = h^{m,l}_{n,\tau}$ for all $m < n$, $l < l_m$, and $\tau \in A_n$.
	Otherwise, proceed as follows.

	Applying 
	Lemma \ref{Generalized Separated Conditions are Dense} (Generalized 
	Separated Conditions are Dense) in $N$, 
	find a family $\{ \bar{r}^{m,l}_n : m < n, \ l < l_m \}$ 
	of conditions with top level $\delta_n + 1$ and a condition 
	$\bar{v}_n \le_\theta v_n$ 
	such that for all $m < n$ and $l < l_m$, $\bar{r}^{m,l}_n \le r^{m,l}_n$, 
	$\bar{r}^{m,l}_n \res \theta = \bar{v}_n$, 
	$\bar{r}^{m,l}_n$ is separated, and for all $\tau \in A_n$, 
	$X_n \res \delta_n$ and $X_n \res (\delta_{n}+1)$ is $\bar{r}^{m,l}_n(\tau)$-consistent.

	Fix an injective tuple 
	$\vec z = (z_0,\ldots,z_{\hat{n}-1})$ which enumerates $X_{n}$. 
	Since the elements of $\vec a^0$ and $\vec a^1$ are in $X_n$, 
	we can fix distinct $j_0,\ldots,j_{d-1}$ and distinct 
	$k_0,\ldots,k_{d-1}$ in $\hat{n}$ such that 
	$\vec a^0 = (z_{j_0},\ldots,z_{j_{d-1}})$ and 
	$\vec a^1 = (z_{k_0},\ldots,z_{k_{d-1}})$. 
	Let $x_m = z_m \res (\delta_n + 1)$ for all $m < \hat{n}$, and let 
	$\vec x = (x_0,\ldots,x_{\hat{n}-1})$.

	Define $\mathcal X^*$ as the set of all tuples 
	$\vec y = (y_0,\ldots,y_{\hat{n}-1})$ in the derived tree $T_{\vec x}$ 
	for which there exist conditions $t^0, t^1 \le p$ with top level $\xi$ equal 
	to the height of $\vec y$ and there exists $w \in \p_\theta$ satisfying:
	\begin{itemize}
	\item $t^0 \res \theta = t^1 \res \theta = w \le_\theta \bar{v}_n$;
	\item for all $\tau \in A_n \cap \theta$, 
	$\vec x$ and $\vec y$ are $w(\tau)$-consistent;
	\item for $j < 2$, $A_n \subseteq \dom(t^j)$;
	\item for $j < 2$, 
	for any finite $Y \subseteq T_\xi$ with unique drop-downs to $\delta_n$, 
	$\{ t^j(\tau) : \tau \in A_n \}$ is separated on $Y$;
	\item for all $\tau \in A$, 
	$\vec b$ and $(y_{j_0},\ldots,y_{j_{d-1}})$ are $t^0(\tau)$-consistent, 
	and $\vec b$ and $(y_{k_0},\ldots,y_{k_{d-1}})$ are $t^1(\tau)$-consistent;
	\end{itemize}

	Now let $\mathcal X$ be the set of all $\vec y$ in $T_{\vec x}$ such that 
	either $\vec y \in \mathcal X^*$, or else 
	for all $\vec z \ge \vec y$, $\vec z \notin \mathcal X^*$. 
	Obviously, $\mathcal X$ is dense in $T_{\vec x}$. 
	Using Lemma \ref{Generalized Extension} (Generalized Extension) and 
	Lemma \ref{Persistence for Sets 2} (Persistence for Sets) 
	it is easy to show that $\mathcal X^*$ is open. 
	Also, $\mathcal X \in N$ by elementarity. 
	Since $T$ is a free Suslin tree, $T_{\vec x}$ is Suslin. 
	So we can fix some $\xi < \delta$ greater than $\gamma$ such that every member 
	of $T_{\vec x}$ of height at least $\xi$ is in $\mathcal X$.

	We consider two cases. 
	First, assume that $\vec z \res \xi \notin \mathcal X^*$. 
	Since $\vec z \res \xi \in \mathcal X$, it follows that 
	$\vec z$ is not in $\mathcal X^*$ either. 
	In this case, let $l_n = 0$, $X_{n+1} = X_n$, $A_{n+1} = A_n$, 
	$\delta_{n+1} = \delta_{n}+1$, $v_{n+1} = \bar{v}_n$, 
	$r^{m,l}_{n+1} = \bar{r}^{m,l}_{n}$ and 
	$h^{m,l}_{n+1,\tau} = h^{m,l}_{n,\tau}$ 
	for all $m < n$, $l < l_m$, and $\tau \in A_n$.

	Secondly, assume that $\vec z \res \xi \in \mathcal X^*$. 
	Fix $t^0$, $t^1$, and $w$ in $N$ satisfying the five statements listed 
	in the definition of $\mathcal X^*$. 
	Let $l_n = 3$. 
	Apply 
	Lemma \ref{Generalized Key Property} (Generalized Key Property) to find 
	tuples $\vec{c}^0$ and $\vec{c}^1$ above $\vec b$ with height $\xi$ satisfying:
	\begin{itemize}
	\item for all $\tau \in A$ 
	and $j < 2$, $\vec b$ and $\vec{c}^j$ are $t^j(\tau)$-consistent;
	\item $\vec{c}^0 \res (\alpha+1)$, $\vec{c}^1 \res (\alpha+1)$, and 
	$X_n \res (\alpha+1)$ are pairwise disjoint;
	\item for all $x \in X_n \res \xi$, $\tau \in A_n \cap \theta$, 
	and $m \in \{ -1, 1 \}$, 
	$w(\tau)^m(x)$ is not in $\vec{c}^0$ or in $\vec{c}^1$;
	\item for all $\tau \in A_n \cap \theta$, $m \in \{ -1, 1 \}$, and $j < 2$, if 
	$x$ is in $\vec c^j$ then $w(\tau)^m(x)$ is not in $\vec c^{1 - j}$.
	\end{itemize}

	Pick arbitrary tuples $\vec d^0$ and $\vec d^1$ 
	above $\vec{c}^0$ and $\vec{c}^1$ respectively of height $\xi + 1$. 
	Let $D^0$ be the set of elements in $\vec d^0$ and let $D^1$ the set 
	of elements in $\vec d^1$. 
	Apply Lemma \ref{very technical lemma} in $N$ to 
	find conditions $r^{n,2}_{n+1} \in N \cap \p$ 
	and $v_{n+1} \in N \cap \p_\theta$, both with top level $\xi+1$, satisfying:
	\begin{itemize}
	\item $r^{n,2}_{n+1} \le p$, $v_{n+1} \le_\theta w$, 
	$r^{n,2}_{n+1} \res \theta = v_{n+1}$, and $A_{n} \subseteq \dom(r^{n,2}_{n+1})$;
	\item for all $\tau \in A$, $\vec b$ and $\vec d^0$ are 
	$r^{n,2}_{n+1}(\tau)$-consistent 
	and $\vec b$ and $\vec d^1$ are $r^{n,2}_{n+1}(\tau)$-consistent;
	\item for all $\tau \in A_{n} \cap \theta$, 
	$(X_{n} \res \xi) \cup (D^0 \res \xi) \cup (D^1 \res \xi)$ and 
	$(X_n \res (\xi+1)) \cup D^0 \cup D^1$ are $v_{n+1}(\tau)$-consistent;
	\item $\{ r^{n,2}_{n+1}(\tau) : \tau \in A_n \}$ 
	is separated on $(X_n \res (\xi+1)) \cup D^0 \cup D^1$.
	\end{itemize}

	Apply Lemma \ref{Generalized Extension} (Generalized Extension) to 
	find $r^{n,0}_{n+1} \le t^0$ and 
	$r^{n,1}_{n+1} \le t^1$ with top level $\xi+1$ such that for each $j < 2$, 
	$r^{n,j}_{n+1} \res \theta = v_{n+1}$ and 
	for all $\tau \in A_n$, 
	$(X_{n} \res \xi) \cup (D^0 \res \xi) \cup (D^1 \res \xi)$ and 
	$(X_n \res (\xi+1)) \cup D^0 \cup D^1$ are $r^{n,j}_{n+1}(\tau)$-consistent. 
	It follows that for $j < 2$, for all $\tau \in A$, 
	$\vec b$ and $\vec d^j$ are $r^{n,j}_{n+1}(\tau)$-consistent. 
	For each $j < 2$, since $\delta_n > \alpha$, 
	the set $(X_n \res \xi) \cup (D^0 \res \xi) \cup (D^1 \res \xi)$ 
	has unique drop-downs to $\delta_n$, 
	so $\{ t^j(\tau) : \tau \in A_n \}$ is separated on it; 
	by Lemma \ref{Persistence for Sets 2} (Persistence for Sets), 
	$\{ r^{n,j}_{n+1}(\tau) : \tau \in A_n \}$ is separated 
	on $(X_n \res (\xi+1)) \cup D^0 \cup D^1$. 

	Apply Lemma \ref{Generalized Extension} (Generalized Extension) to find a family 
	$\{ r^{m,l}_{n+1} : m < n, \ l < l_m \}$ of conditions with top level $\xi+1$ 
	such that for all $m < n$ and $l < l_m$, 
	$r^{m,l}_{n+1} \le \bar{r}^{m,l}_n$, 
	$r^{m,l}_{n+1} \res \theta = v_{n+1}$, and 
	for all $\tau \in A_{n}$, $X_n \res (\delta_n + 1)$ and $X_n \res (\xi+1)$ 
	are $r^{m,l}_{n+1}(\tau)$-consistent. 
	Since each $\bar{r}^{m,l}_n$ is separated and $(X_n \res (\xi+1)) \cup D^0 \cup D^1$ 
	has unique drop-downs to $\delta_n + 1$, 
	by Lemma \ref{Persistence for Sets 2} (Persistence for Sets), 
	$\{ r^{m,l}_{n+1}(\tau) : \tau \in A_n \}$ 
	is separated on $(X_n \res (\xi+1)) \cup D^0 \cup D^1$.
	
	Define $A_{n+1} = A_n$ and $\delta_{n+1} = \xi+1$. 
	Fix arbitrary tuples $\vec e^{n,0}$ and $\vec e^{n,1}$ above 
	$\vec d^0$ and $\vec d^1$ respectively with height $\delta$. 	
	Define $X_{n+1}$ by adding to $X_n$ the elements of $\vec e^{n,0}$ and $\vec e^{n,1}$. 
	Finally, define $h^{m,l}_{n+1,\tau}$ for each $m < n$, $l < l_m$, and $\tau \in A_{n+1}$ 
	as described in inductive hypothesis (7).
	It is straightforward to check that the required properties are satisfied.
	
	To make the verification of (3) below easier to check, let us highlight the following 
	facts which we have proven:
	\begin{itemize}
	\item for all $\tau \in A$, $\vec b$ and $\vec d^0$ are 
	$r^{n,2}_{n+1}(\tau)$-consistent and $r^{n,0}_{n+1}(\tau)$-consistent;
	\item for all $\tau \in A$, $\vec b$ and $\vec d^1$ are 
	$r^{n,2}_{n+1}(\tau)$-consistent and $r^{n,1}_{n+1}(\tau)$-consistent;
	\item $\vec d^0 < \vec e^0$, $\vec d^1 < \vec e^1$, and the elements of 
	$\vec e^0$ and $\vec e^1$ are in $X_{n+1}$.
	\end{itemize}
	
	This completes the construction. 
	For all $m < \omega$ and $l < l_m$, define $r^{m,l} \in \p$ with domain 
	$N \cap \ka$ so that for all $\tau \in N \cap \ka$, 
	$$
	r^{m,l}(\tau) = \bigcup \{ r^{m,l}_n(\tau) : m < n < \omega, \ \tau \in A_n \} 
	\cup \bigcup \{ h^{m,l}_{n,\tau} : m < n < \omega, \ \tau \in A_n \}.
	$$
	Also, define $v = r^{m,l} \res \theta$ for some (any) any $m < \omega$ and $l < l_m$. 
	By Lemma \ref{Constructing Total Master Conditions} (Constructing Total Master Conditions), 
	each $r^{m,l}$ is a total master condition over $N$ which is a lower bound of 
	the sequence $\langle r^{m,l}_n : m < n < \omega \rangle$, 
	and for all $n < \omega$ and 
	for all $\tau \in A_n$, $X_n \res \delta_n$ and $X_n$ are $r^{m,l}(\tau)$-consistent. 
	By Lemma \ref{master implies separated}, each $r^{m,l}$ is separated.

	This completes the construction. 
	Note that by what we did at stage $0$, $v \le q \res \theta$. 
	So it remains to prove that $v$ is $N$-nice for $p$, $\vec b$, and $A$. 
	We verify property (1)-(3) of Definition \ref{Nice Conditions} (Nice Conditions).
	
	(1) Clearly, $v \le p \res \theta$ and $v$ has top level $\delta$. 
	Since, for example, $r^{0,0}$ is a total master condition for $\p$ over $N$, 
	by a standard argument it follows that $v = r^{0,0} \res \theta$ is 
	a total master condition for $\p_\theta$ over $N$. 
	So $v$ decides $\dot U \res \delta$.

	(2) Let $s \in N \cap \p$ have top level $\gamma$ such that 
	$s \le p$ and $v \le_\theta s \res \theta$. 
	It suffices to show that for some $n$, $r^{n,0} \le s$. 
	Fix $n_1$ such that $s = s_{n_1}$. 
	Pick $n'$ large enough so that 
	$\dom(s) \cap \theta \subseteq \dom(v_{n'})$ 
	and $\delta_{n'} > \gamma$. 
	Note that for any $n \ge n'$, $v_{n} \le_\theta s \res \theta$. 
	Now find $n \ge n'$ such that for some $\sigma \in N \cap \ka$, 
	$g(n) = (0,n_1,\sigma)$. 
	By inductive hypothesis (4), $r^{n,0} \le s$.
	
	(3) Suppose that $\vec a^0$ and $\vec a^1$ are distinct tuples above $\vec b$ 
	with height $\delta$, and $q_0, q_1 \le p$ have top level $\delta$ and satisfy: 
	$q_0 \res \theta = q_1 \res \theta = v$, 
	$N \cap \ka \subseteq \dom(q_0) \cap \dom(q_1)$, 
	$q_0$ and $q_1$ are separated, $q_0$ and $q_1$ 
	decide $\dot b \cap \delta$, and for all $\tau \in A$ and $j < 2$, 
	$\vec b$ and $\vec a^{j}$ are $q_j(\tau)$-consistent. 
	In particular, for all $\tau \in A \cap \theta$ and $j < 2$, 
	$\vec b$ and $\vec a^j$ are $v(\tau)$-consistent. 
	Fix $n'$ large enough so that the elements of $\vec a^0$ and $\vec a^1$ are in $X_{n'}$. 
	Fix $n_1$ so that $(\vec a^{n_1,0},\vec a^{n_1,1}) = (\vec a^0,\vec a^1)$. 
	Find $n \ge n'$ such that for some $\sigma \in N \cap \ka$, 
	$g(n) = (3,n_1,\sigma)$. 
	Observe that for all $\tau \in A \cap \theta$ and $j < 2$, 
	the fact that $\vec b$ and $\vec a^j$ are $v(\tau)$-consistent implies that 
	$\vec b$ and $\vec a^j \res \delta_n$ are $v_n(\tau)$-consistent.

	So the requirements described in the first paragraph of Case d are met. 
	Letting $\vec z$ be the enumeration of $X_n$ given in Case d, 
	clearly $\vec z \in \mathcal X^*$ as witnessed by 
	$q_0$, $q_1$, and $v$. 
	Using the information which we highlighted at the end of Case d and the fact that 
	for all $j < 3$ and for all $\tau \in A_{n+1}$, 
	$X_{n+1} \res \delta_{n+1}$ and $X_{n+1}$ are 
	$r^{n,j}(\tau)$-consistent, 
	it is routine to check that $r^{n,0}$, $r^{n,1}$, $r^{n,2}$, 
	$\vec e^{n,0}$, and $\vec e^{n,1}$ 
	satisfy the properties 
	described of $q_0^*$, $q_1^*$, $r$, $\vec e^{0}$, and $\vec e^{1}$ 
	in (3) in Definition \ref{Nice Conditions} (Nice Conditions).
\end{proof}

\section{The Automorphism Forcing Adds No New Cofinal Branches} \label{The Automorphism Forcing Adds No New Cofinal Branches}

In this section, we complete the 
proof of Theorem \ref{No New Cofinal Branches} (No New Cofinal Branches).

\begin{proposition} \label{completing no cofinal branches 1}
	Suppose that $\bar p \in \p$ has top level $\beta$ and 
	$\bar p$ forces that $\dot b$ is a cofinal branch of $\dot U$ 
	which is not in $V^{\p_\theta}$. 
	Assume that $A \subseteq \dom(\bar p)$ is finite, 
	$\vec x = (x_0,\ldots,x_{n-1})$ consists of distinct elements of $T_\beta$, and 
	$\{ \bar p(\tau) : \tau \in A \}$ is separated on $\vec x$. 
	Define $\mathcal{X}$ as the set of all tuples $\vec y = (y_0,\ldots,y_{n-1})$ 
	in the derived tree $T_{\vec x}$ 
	for which there exist $q_0, q_1 \le \bar p$ with top level equal to the height $\gamma$ 
	of $\vec y$ such that:
	\begin{enumerate}
		\item $q_0 \res \theta = q_1 \res \theta$;
		\item for all $\tau \in A$ and $j < 2$, $\vec x$ and $\vec y$ 
		are $q_j(\tau)$-consistent;
		\item there exists some $\zeta < \gamma$ such that 
		$q_0 \Vdash_\p \zeta \in \dot b$ 
		and $q_1 \Vdash_\p \zeta \notin \dot b$.
	\end{enumerate}
	Then $\mathcal{X}$ is dense open in $T_{\vec x}$.
\end{proposition}

\begin{proof}
	To prove that $\mathcal{X}$ is open, 
	assume that $\vec y \in \mathcal X$ has height $\gamma$ 
	and $\vec z > \vec y$ has height $\xi$. 
	Fix $q_0, q_1 \in \p$ with top level $\gamma$ which witness that $\vec y \in \mathcal X$, 
	and let $v = q_0 \res \theta = q_1 \res \theta$. 
	By Lemma \ref{Generalized Extension} (Generalized Extension), extend 
	$q_0$ and $q_1$ to $r_0$ and $r_1$ respectively 
	with top level $\xi$ such that $r_0 \res \theta = r_1 \res \theta$ 
	and for all $\tau \in A$ and $j < 2$, 
	$\vec y$ and $\vec z$ are $r_j(\tau)$-consistent. 
	Then $r_0$ and $r_1$ witness that $\vec z \in \mathcal X$.

	Now we prove that $\mathcal{X}$ is dense. 
	Suppose for a contradiction that $\vec y \in T_{\vec x}$ 
	and for all 
	$\vec z \ge \vec y$, $\vec z \notin \mathcal{X}$. 
	Let $\alpha$ be the height of $\vec y$. 
	Apply Proposition \ref{Extension 2} (Extension) to find some $p \le \bar p$ with top level $\alpha$ such that 
	for all $\tau \in A$, $\vec x$ and $\vec y$ are $p(\tau)$-consistent. 
	Since $\{ \bar p(\tau) : \tau \in A \}$ is separated on $\vec x$, 
	$\{ p(\tau) : \tau \in A \}$ is separated on $\vec y$.

	Let $\mathcal C$ be the collection of all suitable models. 
	Fix a $\in$-increasing and continuous chain $\langle N_i : i < \omega_1 \rangle$ 
	of suitable models which are elementary substructures of 
	$(H(\lambda),\in,\unlhd,\mathcal C)$ such that $p$ is in $N_0$. 
	Let $\delta_\gamma = N_\gamma \cap \omega_1$ for all $\gamma < \omega_1$. 
	Observe that for all $\gamma_1 < \gamma_2 < \omega_1$, 
	the property of a condition being $N_{\gamma_1}$-nice for 
	$p$, $\vec y$, and $A$ is definable in $N_{\gamma_2}$.

	We define a function $F$ which takes as inputs any pair $(v,r)$ satisfying 
	that for some $\gamma < \omega_1$:
	\begin{enumerate}
	\item $v$ is $N_\gamma$-nice for 
	$p$, $\vec y$, and $A$;
	\item $r \le p$, $r$ has top level $\delta_\gamma$, 
	$r \res \theta = v$, $N_\gamma \cap \ka \subseteq \dom(r)$, 
	$r$ is separated, and $r$ decides $\dot b \cap \delta_\gamma$.
	\end{enumerate}
	Let $F(v,r)$ be the unique set $b_r$ such that 
	$r \Vdash_\p \dot b \cap \delta_\gamma = b_r$.

\bigskip

	\textbf{Claim 1:} For any $\gamma < \omega_1$ and any $v$ which is $N_\gamma$-nice 
	for $p$, $\vec y$, and $A$, there exists some $r$ such that 
	$(v,r) \in \dom(F)$. 

    \emph{Proof:} This follows from (2) of Definition \ref{Nice Conditions} (Nice Conditions).

\bigskip

	\textbf{Claim 2:} For any $\gamma < \omega_1$ and any $v$ which is $N_\gamma$-nice 
	for $p$, $\vec y$, and $A$, if $(v,q_0)$ and $(v,q_1)$ are both in 
	the domain of $F$, then $F(v,q_0) = F(v,q_1)$.
    
	\emph{Proof:} Using Lemma \ref{Key Property 2} (Key Property) twice, fix disjoint 
	tuples $\vec a^0$ and $\vec a^1$ above $\vec y$ with height $\delta_\gamma$ 
	such that for all $\tau \in A$ and $j < 2$, 
	$\vec y$ and $\vec a^j$ are $q_j(\tau)$-consistent. 

	Note that $q_0$, $q_1$, $\vec a^0$, and $\vec a^1$ 
	satisfy the properties listed in (3) 
	of Definition \ref{Nice Conditions} (Nice Conditions). 
	Since $v$ is $N_\gamma$-nice, there exist 
	$q_0^*$ and $q_1^*$ satisfying the same properties of $q_0$ and $q_1$ 
	which are listed in (3) 
	of Definition \ref{Nice Conditions} (Nice Conditions), 
	there exists $r \le p$ with top level $\delta_\gamma$, 
	and there exist disjoint tuples $\vec e^0$ and $\vec e^1$ 
	above $\vec y$ with height $\delta_\gamma$ satisfying:
	$r \res \theta = v$, $N_\gamma \cap \ka \subseteq \dom(r)$, 
	$r$ is separated, $r$ decides $\dot b \cap \delta_\gamma$, 
	and for all $\tau \in A$ and $j < 2$, 
	$\vec y$ and $\vec e^j$ are $q_j^*(\tau)$-consistent 
	and $r(\tau)$-consistent.

	Observe that $(v,q_0^*)$, $(v,q_1^*)$, and $(v,r)$ are in the domain of $F$. 
	Recall that for all $\vec z \ge \vec y$, $\vec z \notin \mathcal X$. 
	In particular, $\vec e^0$, $\vec e^1$, $\vec a^0$, and $\vec a^1$ are not in $\mathcal X$.  
	If $F(v,q_0^*) \ne F(v,r)$, then $q_0^*$ and $r$ would witness that 
	$\vec e^0 \in \mathcal X$. 
	Hence, $F(v,q_0^*) = F(v,r)$. 
	If $F(v,r) \ne F(v,q_1^*)$, then $\vec e^1$ would be in $\mathcal X$. 
	So $F(v,r) = F(v,q_1^*)$. 
	Similarly, the fact that $\vec a^0$ and $\vec a^1$ are not in $\mathcal X$ 
	implies that $F(v,q_0) = F(v,q_0^*)$ and $F(v,q_1) = F(v,q_1^*)$. 
	So $F(v,q_0) = F(v,q_0^*) = F(v,r) = F(v,q_1^*) = F(v,q_1)$. 
	This completes the proof of Claim 2.

\bigskip

	Define $F(v)$ to be equal to $F(v,r)$ for any $r$ such that $(v,r)$ is in the domain of $F$. 
	By Claim 1, $F(v)$ is defined for any $v$ which is $N_\gamma$-nice for 
	$p$, $\vec y$, and $A$ for some $\gamma < \omega_1$. 
	By Claim 2, $F(v)$ is well-defined.

\bigskip

	\textbf{Claim 3:} Suppose that $\gamma_1 < \gamma_2$, $v_1 \in N_{\gamma_2}$ 
	is $N_{\gamma_1}$-nice for $p$, $\vec y$ and $A$, 
	$v_2$ is $N_{\gamma_2}$-nice for $p$, $\vec y$, and $A$, 
	and $v_2 \le_\theta v_1$. 
	Then $F(v_1) = F(v_2) \cap \delta_{\gamma_1}$.
    
	\emph{Proof:} By (2) of Definition \ref{Nice Conditions} (Nice Conditions) and elementarity, 
	fix some $r_1 \le p$ in $N_{\gamma_2}$ with top level $\delta_{\gamma_1}$ 
	such that $r_1 \res \theta = v_1$, $N_{\gamma_1} \cap \ka \subseteq \dom(r_1)$, 
	$r_1$ is separated, and $r_1$ decides $\dot b \cap \delta_{\gamma_1}$ 
	as some set $b_1$. 
	Then $F(v_1) = b_1$. 
	Since $v_2 \le_\theta v_1 = r_1 \res \theta$ and 
	$v_2$ is $N_{\gamma_2}$-nice for $p$, $\vec y$, and $A$, 
	by (2) of Definition \ref{Nice Conditions} (Nice Conditions) 
	we can find some $r_2 \le r_1$ with top level $\delta_{\gamma_2}$ 
	such that $r_2 \res \theta = v_2$, $N_{\gamma_2} \cap \ka \subseteq \dom(r_2)$, 
	$r_2$ is separated, 
	and $r_2$ decides $\dot b \cap \delta_{\gamma_2}$ as some set $b_2$. 
	Then $F(v_2) = b_2$. 
	Since $r_2 \le r_1$, $b_1 = b_2 \cap \delta_{\gamma_1}$, 
	which completes the proof of Claim 3.

\bigskip

	\textbf{Claim 4:} If $\gamma \le \xi < \omega_1$, $v$ is $N_\gamma$-nice 
	for $p$, $\vec y$, and $A$, 
	$w$ is $N_\xi$-nice for $p$, $\vec y$, 
	and $A$, and $v$ and $w$ are compatible in $\p_\theta$, 
	then $F(v) = F(w) \cap \delta_\gamma$. 
	In particular, if $\gamma = \xi$ then $F(v) = F(w)$. 
    
	\emph{Proof:} If not, then since $N_\xi \in N_{\xi+1}$, by elementarity we can find 
	counter-examples $v$ and $w$ in $N_{\xi+1}$. 
	Now find $z \le_\theta v, w$ in $N_{\xi+1}$. 
	Apply Lemma \ref{Simple Generalized Extension} (Simple Generalized Extension) 
	to find $q \le p$ in $N_{\xi+1}$ such that $q \res \theta = z$. 
	By Proposition \ref{Existence of Nice Conditions prop} (Existence of Nice Conditions), 
	fix $v_2 \le_\theta q \res \theta = z$ 
	such that $v_2$ is $N_{\xi+1}$-nice for $p$, $\vec y$, and $A$. 
	Then by Claim 3, $F(v) = F(v_2) \cap \delta_\gamma$ and 
	$F(w) = F(v_2) \cap \delta_\xi$, so 
	$F(v) = (F(v_2) \cap \delta_\xi) \cap \delta_\gamma = F(w) \cap \delta_\gamma$, 
	which is a contradiction. 
	This completes the proof of Claim 4.

\bigskip

	For each $\gamma < \omega_1$, let $\dot c_\gamma$ be a 
	$\p_\theta$-name for the unique set which is equal to $F(v)$, where 
	$v \in \dot G_\theta$ is $N_\gamma$-nice for 
	$p$, $\vec y$, and $A$, and is equal to the emptyset if there is no such $v$. 
	Note that by Claim 4, $\dot c_\gamma$ is well-defined. 
	By elementarity, we can choose the name $\dot c_\gamma$ to be in $N_{\gamma+1}$. 
	Now let $\dot c$ be a $\p_\theta$-name for 
	the union $\bigcup \{ \dot c_{\gamma} : \gamma < \omega_1 \}$. 
	Note that the names $\dot c_\gamma$ and the name $\dot c$ are also $\p$-names 
	since $\p_\theta$ is a regular suborder of $\p$.

\bigskip

	\textbf{Claim 5:} $p$ forces in $\p$ that $\dot c$ is a chain in $\dot U$.
    
	\emph{Proof:} Clearly, for all $\gamma < \omega_1$, $p$ 
	forces that $\dot c_\gamma$ is a chain in $\dot U$. 
	So it suffices to show that $p$ forces that 
	for all $\gamma < \xi$ such that $\dot c_\gamma$ and 
	$\dot c_\xi$ are non-empty, 
	$\dot c_\gamma = \dot c_\xi \cap \delta_\gamma$. 
	So let $G$ be a generic filter on $\p$ which contains $p$ 
	and let $G_\theta = G \cap \p_\theta$. 
	Let $c_\gamma = \dot c_\gamma^{G_\theta}$ 
	and $c_\xi = \dot c_\xi^{G_\theta}$, and assume that 
	$c_\gamma$ and $c_\xi$ are non-empty. 
	Fix $v \in G_\theta$ which is $N_\gamma$-nice for 
	$p$, $\vec y$, and $A$, and fix $w \in G_\theta$ which is 
	$N_\xi$-nice for $p$, $\vec y$, and $A$, which exist because 
	$c_\gamma$ and $c_\xi$ are non-empty.  
	Then $c_\gamma = F(v)$ and $c_\xi = F(w)$. 
	Since $v$ and $w$ are in $G_\theta$, they are compatible in $\p_\theta$, 
	so by Claim 4, $c_\gamma = F(v) = F(w) \cap \delta_\gamma = c_\xi$.
	This completes the proof of Claim 5.

\bigskip

	\textbf{Claim 6:} $p$ forces in $\p$ 
	that $\dot c$ is equal to $\dot b$, 
	and hence that $\dot b \in V^{\p_\theta}$. 
	Since $p \le \bar p$, this claim 
    contradicts our initial assumptions.
    
	\emph{Proof:} Since $p$ forces that 
	$\dot c$ is a chain in $\dot U$ and $\dot b$ is an uncountable branch of $\dot U$, 
	it suffices to show that $p$ forces that 
	$\dot b \subseteq \dot c$.

	Suppose for a contradiction that there exists some $q \le p$ 
	and some $\zeta < \omega_1$ 
	such that $q \Vdash_{\p} \zeta \in \dot b \setminus \dot c$. 
	Fix a $\in$-increasing and continuous sequence 
	$\langle M_\gamma : \gamma < \omega_1 \rangle$ of suitable models  
	such that $M_0$ contains as members the objects 
	$p$, $q$, $\zeta$, 
	$\langle N_i : i < \omega_1 \rangle$, and 
	$\langle \dot c_i : i < \omega_1 \rangle$. 
	Let $D \subseteq \omega_1$ be a club such that for all $\gamma \in D$, 
	$N_\gamma \cap \omega_1 = M_\gamma \cap \omega_1$.

	Fix $\gamma \in D$. 
	Apply Proposition \ref{Existence of Nice Conditions prop} (Existence of Nice Conditions) to find some $v$ which 
	is $M_\gamma$-nice for $p$, $\vec y$, and $A$ 
	such that $v \le_\theta q \res \theta$. 
	By (2) of Definition \ref{Nice Conditions} (Nice Conditions), 
	fix $r \le q$ with top level $\delta_\gamma$ 
	such that $r \res \theta = v$, $N \cap \ka \subseteq \dom(r)$, 
	$r$ is separated, and $r$ decides $\dot b \cap \delta_\gamma$ 
	as some set $b_r$. 
	By Lemma \ref{niceness persistent}, $v$ is also $N_\gamma$-nice. 
	Hence, $F(v) = b_r$. 
	Since $q \Vdash_{\p} \zeta \in \dot b$, $r \le q$, and $\zeta < \delta_\gamma$, 
	$r$ forces that $\zeta \in \dot b \cap \delta_\gamma = b_r$. 
	Hence, $\zeta \in F(v)$. 
	Now $r$ forces that $v \in \dot G_\theta$, 
	so $r$ forces that $\dot c_\gamma = F(v)$, 
	and hence $r$ forces that $\zeta \in \dot c_\gamma$. 
	But $p$ forces that $\dot c_\gamma$ is a subset of $\dot c$. 
	So $r \le q$ and $r \Vdash_\p \zeta \in \dot c$, which is a contradiction. 
    This completes the proof of Claim 6, and with it the proof of the proposition.
\end{proof}

\begin{lemma} \label{completing no cofinal branches 2}
	Suppose that $T$ is a free Suslin tree. 
	Let $N$ be a suitable model and let $\delta = N \cap \omega_1$. 
	Suppose that $p_0,\ldots,p_{l-1}$ are in $N \cap \p$ and 
	$v \in N \cap \p_\theta$, all of which have top level $\beta$. 
	Assume that for all $k < l$, $p_k \res \theta = v$ and $p_k$ forces 
	that $\dot b$ is a cofinal branch of $\dot U$ which is not in $V^{\p_\theta}$. 
	Let $X \subseteq T_\delta$ be finite with unique drop-downs to $\beta$, 
	$A \subseteq \bigcap_{k < l} \dom(p_k)$ is finite, and suppose that 
	for all $k < l$, $\{ p_k(\tau) : \tau \in A \}$ is separated on $X \res \beta$.

	Then there exist $\gamma < \delta$, $w \in N \cap \p_\theta$, 
	and for all $k < l$ conditions $q_{k,0}, q_{k,1}$ in $N \cap \p$, 
	all with top level $\gamma$, satisfying:
	\begin{enumerate}
	\item for each $j < 2$, $q_{k,j} \le p_k$ and $q_{k,j} \res \theta = w$;
	\item for each $j < 2$ and $\tau \in A$, 
	$X \res \beta$ and $X \res \gamma$ are $q_{k,j}(\tau)$-consistent;
	\item there exists some $\zeta < \gamma$ 
	such that $q_{k,0} \Vdash_\p \zeta \in \dot b$ and 
	$q_{k,1} \Vdash_\p \zeta \notin \dot b$.
	\end{enumerate}
\end{lemma}	

\begin{proof}
	The proof is by induction on $l$. 
	Let $N$ and $\delta$ be as above and let $X \subseteq T_\delta$ be finite. 
	
	\underline{Base case:} Suppose that $p \in N \cap \p$ and $v \in N \cap \p_\theta$ 
	have top level $\beta$, $p \res \theta = v$, and $p$ forces that $\dot b$ is a cofinal 
	branch of $\dot U$ which is not in $V^{\p_\theta}$. 
	Assume that $X$ has unique drop-downs to $\beta$, $A \subseteq \dom(p)$ is finite, 
	and $\{ p(\tau) : \tau \in A \}$ is separated on $X \res \beta$. 
	Fix an injective tuple $\vec a = (a_0,\ldots,a_{n-1})$ 
	which enumerates $X$ so that $\{ p(\tau) : \tau \in A \}$ 
	is separated on $\vec a \res \beta$. 
	Let $\vec x = \vec a \res \beta$.

	Define $\mathcal{X}$ as the set of all tuples $\vec b = (b_0,\ldots,b_{n-1})$ 
	in the derived tree $T_{\vec x}$ 
	for which there exist $q_0, q_1 \le p$ with top level equal to 
	the height $\rho$ of $\vec b$ such that:
	\begin{itemize}
		\item $q_0 \res \theta = q_1 \res \theta$;
		\item for all $\tau \in A$ and $j < 2$, $\vec x$ and $\vec b$ 
		are $q_j(\tau)$-consistent;
		\item there exists some $\zeta < \rho$ such that 
		$q_0 \Vdash_\p \zeta \in \dot b$ 
		and $q_1 \Vdash_\p \zeta \notin \dot b$.
	\end{itemize}
	By Proposition \ref{completing no cofinal branches 1}, 
	$\mathcal{X}$ is dense open in $T_{\vec x}$, and 
	$\mathcal{X} \in N$ by elementarity. 
	Since $T$ is a free Suslin tree, $T_{\vec x}$ is Suslin. 
	So by elementarity we can find some $\gamma < \delta$ greater than $\beta$ 
	such that every member of $T_{\vec x}$ with height at least $\gamma$ is in $\mathcal{X}$. 
	In particular, $\vec a \res \gamma \in \mathcal{X}$. 
	Fix $q_0,q_1 \le p$ which witness that $\vec a \res \gamma \in \mathcal X$. 
	Then $\gamma$, $q_0 \res \theta$, 
	$q_0$, and $q_1$ satisfy conclusions (1)-(3).
	
	\underline{Inductive Step:} 
	Let $l > 0$ be given and assume that the statement is true for $l$. 
	We prove that it is true for $l + 1$. 
	Suppose that $p_0,\ldots,p_{l}$ are in $N \cap \p$ and 
	$v \in N \cap \p_\theta$, all of which have top level $\beta$. 
	Assume that for all $k \le l$, $p_k \res \theta = v$ and $p_k$ forces 
	that $\dot b$ is a cofinal branch of $\dot U$ which is not in $V^{\p_\theta}$. 
	Suppose that $X$ has unique drop-downs to $\beta$, 
	$A \subseteq \bigcap_{k \le l} \dom(p_k)$ is finite, and 
	for all $k \le l$, $\{ p_k(\tau) : \tau \in A \}$ is separated on $X \res \beta$.

	By the inductive hypothesis, we can fix 
	$\gamma < \delta$, $w \in N \cap \p_\theta$, 
	and conditions $q_{k,0}, q_{k,1} \le p_k$ in $N \cap \p$ 
	for all $k < l$ satisfying conclusions (1)-(3). 
	By Lemma \ref{Simple Generalized Extension} (Simple Generalized Extension), 
	find $q \le p_{l}$ 
	with top level $\gamma$ such that $q \res \theta = w$ 
	and for all $\tau \in A$, 
	$X \res \beta$ and $X \res \gamma$ are $q(\tau)$-consistent.

	Since $\{ p_{l}(\tau) : \tau \in A \}$ is separated on $X \res \beta$, 
	$\{ q(\tau) : \tau \in A \}$ is separated on $X \res \gamma$ 
	by Lemma \ref{Persistence 2} (Persistence). 
	Fix an injective tuple $\vec a = (a_0,\ldots,a_{n-1})$ 
	which enumerates $X$ so that $\{ q(\tau) : \tau \in A \}$ 
	is separated on $\vec a \res \gamma$. 
	Let $\vec x = \vec a \res \gamma$.

	Define $\mathcal{X}$ as the set of all tuples $\vec b = (b_0,\ldots,b_{n-1})$ 
	in the derived tree $T_{\vec x}$ 
	for which there exist $q_0, q_1 \le q$ with top level equal to 
	the height $\rho$ of $\vec b$ such that:
	\begin{enumerate}
		\item $q_0 \res \theta = q_1 \res \theta$;
		\item for all $\tau \in A$ and $j < 2$, $\vec x$ and $\vec b$ 
		are $q_j(\tau)$-consistent;
		\item there exists some $\zeta < \rho$ such that 
		$q_0 \Vdash_\p \zeta \in \dot b$ 
		and $q_1 \Vdash_\p \zeta \notin \dot b$.
	\end{enumerate}
	By Proposition \ref{completing no cofinal branches 1}, $\mathcal{X}$ is dense open in $T_{\vec x}$, and 
	$\mathcal{X} \in N$ by elementarity. 
	Since $T$ is a free Suslin tree, $T_{\vec x}$ is Suslin. 
	So by elementarity we can find some $\xi < \delta$ greater than $\gamma$ 
	such that every member of $T_{\vec x}$ with height at least $\xi$ is in $\mathcal{X}$. 
	In particular, $\vec a \res \xi \in \mathcal{X}$.

	Fix $\bar{q}_{l,0}, \bar{q}_{l,1} \le q$ 
	which witness that $\vec a \res \xi \in \mathcal X$. 
	Let $z = \bar{q}_{l,0} \res \theta$. 
	Now apply Lemma \ref{Generalized Extension} (Generalized Extension) in $N$ to find, 
	for each $k < l$ and $j < 2$, a condition $\bar{q}_{k,j} \le q_{k,j}$ 
	in $N$ with top level $\xi$ such that $\bar{q}_{k,j} \res \theta = z$ 
	and for all $\tau \in A$, $X \res \gamma$ and $X \res \xi$ 
	are $\bar{q}_{k,j}(\tau)$-consistent. 
	Then $\xi$, $z$, and $\bar{q}_{k,j}$ for all $k < l$ and $j < 2$ are 
	as required.
\end{proof}

\begin{proof}[Proof of Theorem \ref{No New Cofinal Branches} (No New Cofinal Branches)]
	Suppose for a contradiction 
	that there exists a condition $p \in \p$ which 
	forces in $\p$ that $\dot b$ is a cofinal branch of $\dot U$ 
	which is not in $V^{\p_\theta}$. 
	We find some $v \le_\theta p \res \theta$ which forces in $\p_\theta$ that 
	$\dot U$ has an uncountable level, 
	which contradicts that $\dot U$ is a $\p_\theta$-name for an $\omega_1$-tree. 
	Let $\alpha$ be the top level of $p$.

	Fix a suitable model $N$ such that $p \in N$ and let $\delta = N \cap \omega_1$. 
	Fix an increasing sequence $\langle \gamma_n : n < \omega \rangle$ 
	of ordinals cofinal in $\delta$ with $\gamma_0 = \alpha$, and 
	fix an enumeration $\langle D_n : n < \omega \rangle$ of all dense open 
	subsets of $\p$ which lie in $N$. 
	Let $g : \omega \to 2 \times T_\delta \times (N \cap \ka)$ be a surjection 
	such that every element of the codomain has an infinite preimage. 

	We define by induction the following objects in $\omega$-many steps:
	\begin{itemize}

	\item a subset-increasing sequence $\langle X_n : n < \omega \rangle$ 
	of finite subsets of $T_\delta$ with union equal to $T_\delta$;

	\item a subset-increasing sequence $\langle A_n : n < \omega \rangle$ 
	of finite subsets of $N \cap \ka$ with union equal to $N \cap \ka$;

	\item an increasing sequence $\langle \delta_n : n < \omega \rangle$ 
	of ordinals cofinal in $\delta$;

	\item a decreasing sequence $\langle v_n : n < \omega \rangle$ of conditions 
	in $N \cap \p_\theta$;

	\item a family of conditions $\{ r^s : s \in {}^{< \omega} 2 \} \subseteq N \cap \p$ 
	and ordinals $\{ \zeta_s : s \in {}^{< \omega} 2 \} \subseteq \delta$;

	\item for all $n < \omega$, $s \in {}^{n} 2$, and $\tau \in A_{n}$, 
	an injective partial function $h^{s}_\tau$ from $X_{n}$ to $X_{n}$.
	
	\end{itemize}
	
	We maintain the following inductive hypotheses for all $n < \omega$ 
	and $s \in {}^n 2$:

	\begin{enumerate}

	\item $X_n$ has unique drop-downs to $\delta_n$;

	\item $\delta_n$ is the top level of $r^s$;
		
	\item $r^s \le p$, and if $m > n$, $t \in {}^m 2$, and $s \subseteq t$, 
	then $r^t \le r^s$;
	
	\item $r^s \res \theta = v_{n}$;
	
	\item $A_n \subseteq \dom(r^s)$;
	
	\item $\zeta_s < \delta_{n+1}$, 
	$r^{s^\frown 0} \Vdash \zeta_s \in \dot b$, and 
	$r^{s^\frown 1} \Vdash \zeta_s \notin \dot b$;

	\item for all $\tau \in A_{n}$ and $j < 2$, $X_{n} \res \delta_{n}$ 
	and $X_{n} \res \delta_{n+1}$ are $r^{s^\frown j}(\tau)$-consistent;
	
	\item $r^{s^\frown 0}$ and $r^{s^\frown 1}$ are in $D_n$;
	
	\item $\{ r^{s}(\tau) : \tau \in A_n \}$ is separated on $X_n \res \delta_n$;
	
	\item for all $\tau \in A_n$ and $x, y \in X_n$, 
	$$
	h^s_{\tau}(x) = y \ \Longleftrightarrow \ r^s(\tau)(x \res \delta_n) = y \res \delta_n.
	$$

	\end{enumerate}

\underline{Stage 0:} Let $X_0 = \emptyset$, $A_0 = \emptyset$, $\delta_0 = \alpha$, 
$v_0 = p \res \theta$, $r^{\emptyset} = p$.

\underline{Stage $n+1$:} Let $n < \omega$ and assume that we have completed stage $n$. 
In particular, we have defined the following objects which satisfy the 
required properties: 
$X_n$, $A_n$, $\delta_n$, $v_n$, $r^s$, $\zeta_s$, and $h^s_\tau$ for all 
$s \in {}^n 2$ and $\tau \in A_n$. 
Let $g(n) = (n_0,z,\sigma)$.

Fix $\rho < \delta$ larger than $\delta_n$ and $\gamma_{n+1}$ 
such that $X_n \cup \{ z \}$ has unique drop-downs to $\rho$. 
Let $D$ be the set of conditions $r$ in $D_n$ which have some top level $\xi \ge \rho$ 
such that $A_n \cup \{ \sigma \} \subseteq \dom(r)$. 
Then $D \in N$ and $D$ is dense open in $\p$.

Apply Lemma \ref{Generalized Consistent Extensions Into Dense Sets} (Generalized 
Consistent Extensions Into Dense Sets) to find a 
family $\{ \bar{r}^s : s \in {}^n 2 \}$ of conditions in $N \cap D$, a 
condition $\bar{v}_n \in N \cap \p_\theta$, and $\gamma < \delta$ 
so that for each $s \in {}^n 2$:
\begin{enumerate}
\item[(a)] $\bar{r}^s \le r^s$, $\bar{r}^s$ has top level $\gamma$, 
and $\bar{r}^s \res \theta = \bar{v}_n$;
\item[(b)] for all $\tau \in A_n$, $X_n \res \delta_n$ and $X_n \res \gamma$ 
are $\bar{r}^s(\tau)$-consistent.
\end{enumerate}

Apply Lemma \ref{Generalized Augmentation} (Generalized Augmentation) to find a family 
$\{ \hat{r}^s : s \in {}^n 2 \}$ of conditions in $N \cap \p$ and 
a condition $\hat{v}_n \in N \cap \p_\theta$, 
all with top level $\gamma+1$, 
and a finite set $Y \subseteq T_\delta$ 
such that $X_n \cup \{ z \} \subseteq Y$ and $Y$ has unique 
drop-downs to $\gamma+1$, satisfying that for all $s \in {}^n 2$:
	\begin{enumerate}
	\item[(c)] $\hat{r}^s \le \bar{r}^s$ and $\hat{r}^s \res \theta = \hat{v}_n$;
	\item[(d)] for all $\tau \in A_n$, 
	$X_n \res \gamma$ and $X_n \res (\gamma+1)$ are $\hat{r}^s(\tau)$-consistent;
	\item[(e)] $\{ \hat{r}^s(\tau) : \tau \in A_n \cup \{ \sigma \} \}$ 
	is separated on $Y \res (\gamma+1)$;
	\item[(f)] let $h_{s,\sigma}^+$ be the partial injective function 
	from $Y$ to $Y$ defined by letting, 
	for all $x, y \in Y$, $h_{s,\sigma}^+(x) = y$ iff 
	$\hat{r}^s(\sigma)(x \res (\gamma+1)) = y \res (\gamma+1)$; 
	then $z$ is in the domain and range of $h_{s,\sigma}^+$.
	\end{enumerate}
Define $X_{n+1} = Y$ and $A_{n+1} = A_n \cup \{ \sigma \}$. 
So for all $s \in {}^n 2$, $\{ \hat{r}^s(\tau) : \tau \in A_{n+1} \}$ 
is separated on $X_{n+1} \res (\gamma+1)$.

Now apply Lemma \ref{completing no cofinal branches 2} to 
find $\delta_{n+1} < \delta$, $v_{n+1} \in N \cap \p_\theta$, 
and for all $s \in {}^n 2$, conditions $r^{s^\frown 0}, r^{s^\frown 1} \le \hat{r}^s$ 
in $N \cap \p$ satisfying that for all $s \in {}^n 2$:
	\begin{enumerate}
	\item[(g)] for each $j < 2$, $r^{s^\frown j}$ has top level $\delta_{n+1}$ 
	and $r^{s^\frown j} \res \theta = v_{n+1}$;
	\item[(h)] for each $j < 2$ and $\tau \in A_{n+1}$, 
	$X_{n+1} \res (\gamma+1)$ and $X_{n+1} \res \delta_{n+1}$ 
	are $r^{s^\frown j}(\tau)$-consistent;
	\item[(i)] there exists some $\zeta_s < \delta_{n+1}$ 
	such that $r^{s^\frown 0} \Vdash_\p \zeta_s \in \dot b$ and 
	$r^{s^\frown 1} \Vdash_\p \zeta \notin \dot b$.
	\end{enumerate}
For each $s \in {}^n 2$, $j < 2$, and $\tau \in A_{n+1}$, define 
a partial injective function 
$h^{s^\frown j}_\tau$ from $X_{n+1}$ to $X_{n+1}$ 
by letting, for all $x, y \in X_{n+1}$, $h^{s^\frown j}_\tau(x) = y$ iff 
$r^{s^\frown j}(\tau)(x \res \delta_{n+1}) = y \res \delta_{n+1}$. 
Observe that by (f) and (h), we have:
\begin{enumerate}
\item[(j)] for all $s \in {}^n 2$ and $j < 2$, 
$h_{s,\sigma}^+ = h^{s^\frown j}_\sigma$, and hence  
$z$ is in the domain and range of $h^{s^\frown j}_\sigma$.
\end{enumerate}
This completes stage $n+1$. 
It is easy to check that the inductive hypotheses are satisfied.

This completes the construction. 
For each $f \in {}^{\omega} 2$, define a condition $r_f$ with domain 
equal to $N \cap \ka$ as follows. 
For any $\tau \in N \cap \ka$, define 
$$
r_f(\tau) = 
\bigcup \{ r^{f \res n}(\tau) : n < \omega, \ \tau \in A_n \} \cup 
\bigcup \{ h^{f \res n}_\tau : n < \omega, \ n \in A_n \}.
$$
By Lemma \ref{Constructing Total Master Conditions} (Constructing Total Master Conditions), 
each $r_f$ is a total master condition for $\p$ over $N$. 
Define $v = r_f \res \theta$ for some (any) $f \in {}^{\omega} \omega$.

For each $f \in {}^\omega 2$, let $b_f$ be such that 
$r_f \Vdash_\p \dot b \cap \delta = b_f$. 
Due to our assumption about the levels of $\dot U$, it is easy to argue that $r_f$ forces 
that $b_f$ is a cofinal branch of $\dot U \res \delta$. 
Suppose that $f \ne g$. 
Let $n$ be least such that $f(n) \ne g(n)$, and assume without loss of generality that 
$f(n) = 0$ and $g(n) = 1$. 
Let $s = f \res n = g \res n$. 
Then $r_f \le r^{s^\frown 0}$ and $r_g \le r^{s^\frown 1}$. 
So $r_f \Vdash_\p \zeta_s \in \dot b \res \delta$ and 
$r_g \Vdash_\p \zeta_s \notin \dot b \res \delta$. 
Hence, $b_f \ne b_g$.

Now we are ready to get a contradiction. 
Let $G_\theta$ be a generic filter for $\p_\theta$ such that $v \in G_\theta$. 
In $V[G_\theta]$, let $\p / G_\theta$ be the 
suborder of $\p$ consisting of all $q \in \p$ 
such that $q \res \theta \in G_\theta$. 
Let $U = \dot U^{G_\theta}$. 
In $V[G_\theta]$, each $b_f$ is a cofinal branch of $U \res \delta$, and 
$r_f$ forces in $\p / G_\theta$ that $b_f = \dot b \cap \delta$. 
Hence $r_f$ forces that $\dot b(\delta)$ is an upper bound of $b_f$. 
Since having an upper bound in $U$ is absolute between $V[G_\theta]$ and any 
generic extension by $\p / G_\theta$, $b_f$ does in fact 
have an upper bound in $U_\delta$, 
which we denote by $x_f$. 
By construction, if $f \ne g$ then $b_f \ne b_g$, so $x_f \ne x_g$. 
Hence, $U_\delta$ is uncountable, which contradicts that $U$ is 
an $\omega_1$-tree in $V[G_\theta]$.
\end{proof}

\section{The Main Result} \label{The Main Result}

We are now prepared to prove the main result of the article.

\begin{thm}[Main Theorem] \label{main theorem}
	Suppose that there exists an inaccessible cardinal $\ka$ and an 
	infinitely splitting normal free Suslin tree $T$. 
	Then there exists a forcing poset $\p$ satisfying that the product forcing 
	$\col(\omega_1,< \! \kappa) \times \p$ forces:
	\begin{enumerate}
	\item $\ka = \omega_2$;
	\item \textsf{GCH} holds;
	\item $T$ is a Suslin tree;
	\item there exists an almost disjoint 
	family $\{ f_\tau : \tau < \omega_2 \}$ of automorphisms of $T$;
	\item there does not exist a Kurepa tree.
	\end{enumerate}
\end{thm}

\begin{proof}
	Let $V$ be the ground model in which $T$ and $\ka$ are as above. 
	For simplicity in notation, let $\q = \col(\omega_1,< \! \ka)$. 
	Note that $\q \times \p$ has size $\ka$. 
	Let $\p$ be the forcing poset of Definition \ref{poset definition} in $V$.  
	Since $\omega_1$-closed forcings preserve Suslin trees, 
	$T$ is still free in $V^{\q}$. 
	Also, the definition of $\p$ is easily seen to be 
	absolute between $V$ and $V^{\q}$. 
	So $\q \times \p$ is forcing equivalent to the two-step 
	forcing iteration of $\q$ followed by the forcing of 
	Definition \ref{poset definition} (with $\ka = \omega_2$). 
	Since \textsf{CH} holds in $V^{\q}$, in $V^{\q}$ we have that $\p$ 
	is $\omega_2$-c.c. 
	Consequently, $\q \times \p$ is $\ka$-c.c. 
	As $\q$ is $\omega_1$-closed and hence totally proper, and $\q$ 
	forces that $\p$ is totally proper and preserves the fact that $T$ is Suslin, 
	$\q \times \p$ is totally proper and forces that $T$ is Suslin. 
	Statements (1)-(4) are now clear.

	For (5), let $\dot U$ be a nice $(\q \times \p)$-name for an $\omega_1$-tree 
	(with underlying set $\omega_1$). 
	By the $\ka$-c.c.\ property and the fact that conditions in $\q \times \p$ 
	have countable domain, there exists some $\theta < \ka$ such that 
	$\dot U$ is a $(\col(\omega_1,< \! \theta) \times \p_\theta)$-name. 
	Let $G \times H$ be a generic filter on $\q \times \p$. 
	By the usual factor analysis for product forcings, we can write 
	$V[G] = V[G_\theta][G^\theta][H_\theta][H]$, where 
	$G_\theta = G \cap \, \col(\omega_1,< \! \theta)$, 
	$G^\theta = G \cap \, \col(\omega_1,[\theta,\ka))$, 
	$H_\theta = H \cap \, \p_\theta$, and we consider $H$ to be a 
	$V[G][H_\theta]$-generic filter on the quotient forcing $\p / H_\theta$.
	
	Let $U = \dot U^{G_\theta \times H_\theta}$.  
	Then $U$ is in $V[G_\theta][H_\theta]$, and hence 
	$U$ is in $V[G][H_\theta]$. 
	Consider a cofinal branch $b$ of $U$ in $V[G][H]$. 
	Applying Theorem \ref{No New Cofinal Branches} in $V[G]$, 
	$b$ is in the model $V[G][H_\theta]$, 
	which by the product lemma is equal to the model $V[G_\theta][H_\theta][G^\theta]$. 
	Since $\col(\omega_1,[\theta,\ka))$ is $\omega_1$-closed in $V[G_\theta][H_\theta]$ 
	and $\omega_1$-closed forcings do not add new cofinal branches of $\omega_1$-trees, 
	$b$ is in $V[G_\theta][H_\theta]$. 
	As $\ka$ is inaccessible and $\theta < \ka$, $\ka$ is also inaccessible 
	in $V[G_\theta][H_\theta]$. 
	Because every cofinal branch of $U$ in $V[G][H]$ is in $V[G_\theta][H_\theta]$, there 
	are fewer than $\ka$ many cofinal branches of $U$ in $V[G][H]$. 
	So $U$ is not a Kurepa tree in $V[G][H]$.
\end{proof}

It remains to verify the following claim from Section \ref{Introduction}.

\begin{proposition}
Let $\ka$ be an inaccessible cardinal, 
let $\q$ be Jech's forcing for adding a Suslin tree, and let 
$\dot \p$ be a $\q$-name for the forcing of 
Definition \ref{poset definition} using the generic Suslin tree. 
Then $\q * \dot \p$ is forcing equivalent to some $\omega_1$-closed forcing.
\end{proposition}

\begin{proof}
In $V$ define a forcing poset $\mathbb A$ as follows. 
A condition in $\mathbb A$ is a pair $(t,f)$ satisfying:
\begin{itemize}
\item $t \in \q$;
\item $f$ is a function whose domain is a countable subset of $\ka$ and whose 
range is a set of automorphisms of $t$.
\end{itemize}
Let $(u,g) \le (t,f)$ if $u \le_{\q} t$, $\dom(f) \subseteq \dom(g)$, and for all 
$\alpha \in \dom(f)$, $f(\alpha) \subseteq g(\alpha)$. 
It is easy to check that $\mathbb A$ is $\omega_1$-closed. 
(The forcing $\mathbb A$ is similar to \cite[Theorem 3]{jech72}, but we are not 
requiring that $f$ be injective nor that the range of $f$ be a group).

Let $\mathbb A'$ be the suborder of $\mathbb A$ consisting of all conditions $(t,f)$ 
such that $t$ has successor height. 
We claim that $\mathbb A'$ is dense in $\mathbb A$. 
So let $(t,f) \in \mathbb A$ 
be such that $t$ has height a limit ordinal $\delta$. 
Enumerate $t$ as $\langle x_n : n < \omega \rangle$ and $\dom(f)$ as 
$\langle \gamma_n : n < \omega \rangle$. 
For each $n < \omega$ fix a cofinal branch $b_n$ of $t$ with $x_n \in b_n$.
For purposes of bookkeeping, fix a surjection 
$$
h : \omega \to 2 \times \{ -1, 1 \} \times \omega \times \omega
$$ 
such that: for all $n < \omega$, letting $h(n) = (i,m,k,l)$, if $n = 0$ then $i = 0$, 
and if $n > 0$ then $k < n$.

We build a sequence $\langle c_n : n < \omega \rangle$ of cofinal branches 
of $t$ in $\omega$-many stages as follows. 
Assuming that $n < \omega$ and we have completed stages before $n$, let 
$h(n) = (i,m,k,l)$. 
If $i = 0$, let $c_n = f(\gamma_l)^{m}[b_k]$. 
If $i = 1$, let $c_n = f(\gamma_l)^{m}[c_k]$. 
This completes the construction. 
Now let $u = t \cup \{ b_n : n < \omega \} \cup \{ c_n : n < \omega \}$. 
Define $g$ with domain equal to $\dom(f)$ so that for all $\tau \in \dom(f)$, 
$g(\tau) \res t = f(\tau)$ and for all $b$ in $u \setminus t$, 
$g(\tau)(b) = f(\tau)[b]$. 
By our bookkeeping, 
it is straightforward to check that $(u,g) \in \mathbb A'$ and $(u,g) \le (t,f)$.

Now the map from $\mathbb A'$ to $\q * \dot \p$ given by $(t,f) \mapsto (t,\check f)$ 
is clearly an embedding. 
To see that the range is dense, consider $(t,\dot f) \in \q * \dot \p$. 
Extend $t$ to $u$ in $\q$ 
which decides $\dot f$ as $f$ and has successor height greater than the 
height of $t$. 
Now applying Proposition \ref{general extending} (replacing $T \res (\alpha+1)$ with $u$, 
letting $X = \emptyset$, and ignoring (3)), there exists some $g$ such that 
$\dom(f) = \dom(g)$ and for all $\tau \in \dom(f)$, $g(\tau)$ is an automorphism of $u$ 
such that $f(\tau) \subseteq g(\tau)$. 
Then $(u,g) \in \mathbb A'$ and $(u,\check g) \le (t,\dot f)$ in $\q * \dot \p$.
\end{proof}

By the arguments given in Section \ref{Introduction}, we have the following consequences of 
Theorem \ref{main theorem} (Main Theorem).

\begin{corollary}[$\text{Almost Kurepa Suslin Tree} + \neg \textsf{KH}$]
	Assume that there exists an inaccessible cardinal $\ka$. 
	Then there exists a generic extension in which $\ka$ equals $\omega_2$, \textsf{CH} holds, 
	there exists an almost Kurepa Suslin tree, and there does not exist a Kurepa tree.
\end{corollary}

\begin{corollary}[$\text{$T$ is Suslin} + \sigma(T) = \omega_2 + \Diamond + \neg \textsf{KH}$]
	Assume that there exists an inaccessible cardinal $\ka$. 
	Then there exists a generic extension in which $\ka$ equals $\omega_2$, $\Diamond$ holds, 
	there exists a normal Suslin tree with $\omega_2$-many automorphisms, 
	and there does not exist a Kurepa tree.
\end{corollary}

\begin{corollary}[$\text{Non-Saturated Aronszajn Tree} + \neg \textsf{KH}$]
	Assume that there exists an inaccessible cardinal $\ka$. 
	Then there exists a generic extension in which $\ka$ equals $\omega_2$, 
	there exists a non-saturated Aronszajn tree, and there does not exist a Kurepa tree.
\end{corollary}

\section*{Acknowledgements}

Some of the work for this article was completed while the first 
author was visiting the second author at the Department of Logic at Charles University 
during August and October of 2023. 
The first author thanks the second author and Radek Honz\'{i}k for their 
hospitality during his visits. 
The first author thanks Justin Moore for pointing out that an almost 
Kurepa Suslin tree implies the existence of a non-saturated Aronszajn tree. 
The second author thanks Assaf Rinot for bringing Problem 3 to her attention 
while she was a postdoc at Bar-Ilan University in 2022. 
The first author acknowledges support from the Simons Foundation 
under the Travel Support for Mathematicians gift 631279. 
The second author acknowledges support from the Czech Science Foundation (GAČR), Grant no.\ 24-12141S.
Finally, we thank the referee for helpful feedback.


\end{document}